\documentclass[10pt,english]{smfart}
\usepackage[T1]{fontenc}
\usepackage[english]{babel}
\usepackage{latexsym,amscd,color}
\usepackage{amsmath,amsfonts,amssymb,mathrsfs,amscd}
\usepackage{enumerate,euscript}
\usepackage{amssymb,url,xspace,smfthm}
\usepackage[colorlinks,linkcolor=blue,anchorcolor=blue,citecolor=blue]{hyperref}
\input xy
\xyoption{all}

\DeclareMathOperator{\gal}{Gal}

\newcommand{\BibTeX}{{\scshape Bib}\kern-.08em\TeX}

\newcommand{\T}{\S\kern .15em\relax }
\newcommand{\AMS}{$\mathcal{A}$\kern-.1667em\lower.5ex\hbox
        {$\mathcal{M}$}\kern-.125em$\mathcal{S}$}

\DeclareMathOperator{\vol}{vol}

\DeclareMathOperator{\im}{Im}

\DeclareMathOperator{\spm}{Spm}

\DeclareMathOperator{\rg}{rk}

\DeclareMathOperator{\spec}{Spec}

\renewcommand{\P}{\mathbb{P}}
\newcommand{\wmu}{\widehat{\mu}}
\newcommand{\C}{\mathbb{C}}
\newcommand{\Q}{\mathbb{Q}}

\newcommand{\adeg}{\widehat{\deg}}
\newcommand{\sF}{\mathcal{F}}
\newcommand{\p}{\mathfrak{p}}
\DeclareMathOperator{\sym}{sym}

\newcommand{\E}{\overline{E}}
\newcommand{\F}{\overline{F}}
\newcommand{\sE}{\mathcal{E}}
\newcommand{\G}{\overline{G}}

\renewcommand{\O}{\mathcal{O}}
\renewcommand{\L}{\overline{L}}

\newcommand{\f}{\mathbb{F}}
\newcommand{\ndot}{\raisebox{.4ex}{.}}
\newcommand{\ee}{\mathrm e}
\newcommand{\abs}[1]{\left|#1\right|}
\newcommand{\SA}{S_{\mathrm{aux}}}

\makeatletter
\AtBeginDocument{\def\@oddhead{\normalfont\scriptsize
  {\fontsize{6}{7}\selectfont\rightmark{}{}}\hfil{\bfseries\thepage}}}
\makeatother

\title{Uniform lower bound of arithmetic Hilbert--Samuel function of hypersurfaces}
\alttitle{Minoration uniforme de la fonction arithm\'etiques de Hilbert--Samuel des hypersurfaces}

\date{\today}
\author{Chunhui Liu}
\address{Institute for Advanced Study in Mathematics\\
Harbin Institute of Technology\\
150001 Harbin\\P. R. China}
\email{chunhui.liu@hit.edu.cn}

\begin{document}
\def\smfbyname{}
\begin{abstract}
In this article, we give an explicit and uniform lower bound for the arithmetic Hilbert--Samuel function of projective hypersurfaces in the framework of Arakelov geometry. The bound has optimal dependence on the height of the varieties. As an application, we use this lower bound to study the determinant method.
\end{abstract}
\begin{altabstract}
Dans cet article, on donne une minoration explicite et uniforme de la fonction arithm\'etique de Hilbert--Samuel des hypersurfaces sous le cadre de la g\'eom\'etrie d'Arakelov, qui a la d\'ependance optimale de la hauteur de vari\'et\'es. Comme une application, on applique cette minoration dans l'\'etude de la m\'ethode de d\'eterminant.
\end{altabstract}

\maketitle

\tableofcontents

\section{Introduction}
In this article, we study the arithmetic Hilbert--Samuel function of arithmetic varieties in the framework of Arakelov geometry. More precisely, we give a uniform lower bound for the arithmetic Hilbert--Samuel function of projective hypersurfaces.

\subsection{History}
Let $X$ be a closed subscheme of $\mathbb P^n_k$ over the field $k$ of dimension $d$, and $L$ be an ample line bundle over $X$. We have (cf. \cite[Corollary 1.1.25, Theorem 1.2.6]{LazarsfeldI})
\[\dim_k\left(H^0(X,L^{\otimes D})\right)=\frac{\deg\left(c_1(L)^d\right)}{d!}D^d+o(D^{d})\]
when $D\to\infty$. We call $\dim_k\left(H^0(X,L^{\otimes D})\right)$ the (geometric) Hilbert--Samuel function of $X$ with respect to $L$ in the variable $D\in\mathbb N^+$.

One of the central problems in Arakelov geometry is to find an arithmetic analogue of the Hilbert--Samuel function defined above. More precisely, let $K$ be a number field, $\O_K$ be the ring of integers, and $M_{K,\infty}$ be the set of infinite places. We suppose that $\pi:\mathscr X\rightarrow \spec\O_K$ is an arithmetic variety of dimension $d+1$, which means that $\mathscr X$ is integral and the morphism $\pi$ is flat and projective. Let $\overline{\mathscr L}=\left(\mathscr L,(\|\ndot\|_v)_{v\in M_{K,\infty}}\right)$ be a Hermitian line bundle over $\mathscr X$ (on the generic fiber). Let
\[\widehat{h}^0(\mathscr X,\overline{\mathscr L}^{\otimes D})=\log\#\left\{s\in H^0(\mathscr X,\mathscr L^{\otimes D})\mid\|s\|_{\sup,v}\leqslant1,\forall\;v\in M_{K,\infty}\right\},\]
where $\|s\|_{\sup,v}=\sup\limits_{x\in\mathscr X(\mathbb C_v)}\|s(x)\|_v$ for all $v\in M_{K,\infty}$. Alternatively, we consider $H^0(\mathscr X,\mathscr L^{\otimes D})$ as a normed vector bundle over $\spec\O_K$ equipped with some induced norms, and we consider $\adeg_n\left(\overline{H^0(\mathscr X,\mathscr L^{\otimes D})}\right)$ as a function of the variable $D\in\mathbb N^+$. We call it the \textit{arithmetic Hilbert--Samuel function} of $\left(\mathscr X,\overline{\mathscr L}\right)$.

Motivated by the geometric case, we expect an asymptotic expansion
\[\widehat{h}^0(\mathscr X,\overline{\mathscr L}^{\otimes D})=\frac{\widehat{\vol}(\overline{\mathscr L})}{(d+1)!}D^{d+1}+o(D^{d+1})\]
for some $\widehat{\vol}(\overline{\mathscr L})\in\mathbb R_{\geqslant0}$ when $D\to\infty$, or at least it is valid for some special cases, and some reported results on this topic will be introduced below. For an arithmetically ample Hermitian line bundle, $\widehat{\vol}(\overline{\mathscr L})=\adeg\left(\widehat{c}_1(\overline{\mathscr L})^{d+1}\right)$.

Apart from asymptotic properties, uniform bounds of $\widehat{h}^0(\mathscr X,\overline{\mathscr L}^{\otimes D})$ for all $D\in\mathbb N^+$ are also of interest.
\subsubsection{}
Gillet and Soul\'e proved such an asymptotic result (\cite[Theorem 8]{Gillet-Soule}) for a Hermitian line ample bundle as a consequence of their arithmetic Riemann--Roch theorem \cite[Theorem 7]{Gillet-Soule} under the assumption that $\mathscr X$ has a regular generic fiber. In \cite[Th\'eor\`eme principal]{Abbes-Bouche}, Abbes and Bouche proved the same result without using the arithmetic Riemann--Roch theorem, instead they applied some tools of asymptotic analysis of Demailly. In \cite[Theorem 1.4]{Zhang95}, S. Zhang proved this result without the condition of smooth generic fiber by some technique of analytic torsion. Later, Randriambololona generalized this result to the case of coherent sheaf induced from a sub-quotient of the normed vector bundle in \cite[Th\'eor\`eme A]{Randriam06}. X. Yuan considers this topic in \cite{Yuan2008} for the case where $\overline{\mathscr L}$ is arithmetically big.

In \cite{Philippon_Sombra_2008}, Philippon and Sombra proposed another definition of the arithmetic Hilbert--Samuel function, and they proved an asymptotic formula for the case of toric varieties (see \cite[Th\'eor\`eme 0.1]{Philippon_Sombra_2008}). Hajli proved the same asymptotic formula for general projective varieties in \cite{Hajli_2015} using the above formulation.
\subsubsection{}\label{intro:def of F_D}
Uniform bounds of the arithmetic Hilbert--Samuel function of arithmetic varieties are also an important topic. In \cite{Yuan_Zhang13}, X. Yuan and T. Zhang considered such a uniform upper bound for the case of arithmetic surface, and in \cite{Yuan_Zhang18} for the general case. H. Chen studied the uniform upper bound in \cite{Chen2015} by the $\mathbb R$-filtration method of graded linear series. Compared with the asymptotic estimate, these results have the optimal dependence on the height of varieties.

\subsubsection{}
The uniform lower bound seems more difficult to study. In \cite{David_Philippon99}, David and Philippon gave an explicit uniform lower bound of the arithmetic Hilbert--Samuel function. This result was reformulated by H. Chen in \cite[Theorem 4.8]{Chen1} for a study of the determinant method of Arakelov formulation. In fact, let $\mathscr X\rightarrow\spec\O_K$ be an arithmetic variety, and $\overline{\mathscr L}$ be an ample Hermitian line bundle over $\mathscr X$ which determines a polarization in $\mathbb P^n_{\O_K}$ such that $\deg\left(\mathscr X\times_{\spec\O_K}\spec K\right)=\delta$ as a closed subscheme of $\mathbb P^n_K$. We denote by $\F_D$ a particular normed sub-vector bundle of $\overline{H^0(\mathscr X,\mathscr L^{\otimes D})}$ over $\spec \O_K$ (see \S \ref{definition of F_D} for the precise definition original from \cite[\S 2.3]{Chen1}), which equals $H^0(\mathscr X,\mathscr L^{\otimes D})$ when $D$ is sufficiently large. In some cases, the Arakelov degree $\adeg(\F_D)$ can play the role of the number of small sections in $H^0(\mathscr X,\mathscr L^{\otimes D})$ in the study of arithmetic Hilbert--Samuel functions, since 
\[\lim_{D\to\infty}\frac{\adeg(\F_D)}{[K:\Q]D^{d+1}/(d+1)!}=\frac{1}{[K:\Q]}\adeg\left(\widehat{c}_1(\overline{\mathscr L})^{d+1}\right)\]
is valid in this case by \cite[Th\'eor\`eme A]{Randriam06}.

We regard the slope $\wmu(\F_D)$ as the arithmetic Hilbert--Samuel function of the above $\mathscr X$ of the variable $D\in\mathbb N^+$. In the above literatures, the inequality
\[\frac{\wmu(\F_D)}{D}\geqslant\frac{d!}{\delta(2d+2)^{d+1}}h_{\overline{\mathscr L}}(\mathscr X)-\log(n+1)-2^d\]
holds uniformly for all $D\geqslant2(n-d)(\delta-1)+d+2$, where $h_{\overline{\mathscr L}}(\mathscr X)$ is the height of $\mathscr X$ defined by the arithmetic intersection theory \cite{Gillet_Soule-IHES90}.

However, this estimate is far from optimal; by this we mean that the lower bound does not have the optimal dependence on the height $h_{\overline{\mathscr L}}(\mathscr X)$ compared with that in the asymptotic formula.
\subsection{Main result}
In this article, we will give a lower bound of the arithmetic Hilbert--Samuel function of projective hypersurfaces with optimal dependence on the height of the arithmetic varieties. 
\subsubsection{}
Let $K$ be a number field, $\overline{\sE}$ be a Hermitian vector bundle over $\spec\O_K$ of rank $n+1$, and $\overline{\O(1)}$ be the universal bundle of $\mathbb P(\sE)$ equipped with the corresponding Fubini--Study norms for each $v\in M_{K,\infty}$. Let $X$ be a hypersurface of degree $\delta$ in $\mathbb P(\sE_K)$, and $\mathscr X$ be its Zariski closure in $\mathbb P(\sE)$. Let $\F_D$ be the Hermitian vector bundle defined in \S \ref{intro:def of F_D}, and $\wmu(\F_D)$ be the arithmetic Hilbert--Samuel function of the variable $D\in\mathbb N^+$.

We prove the following result in \S \ref{uniform lower bound - final version}; see \eqref{unifomr lower bound of mu(F_D)/D with l^2-norm for hypersurface}.
\begin{theo}\label{upper and lower bound of arithmetic Hilbert-intro}
We keep all the above notations, assume $n\in\mathbb N_{\geqslant2}$, and we choose $\overline{\sE}$ as the one defined in \S \ref{uniform lower bound - final version}. Then the inequality
\[\frac{\wmu(\F_D)}{D}\geqslant\frac{{n+D-\delta\choose n}}{D\left({n+D\choose n}-{n+D-\delta\choose n}\right)}h(X)+C_1(n)\]
holds uniformly for all $D\geqslant\delta$, where $h(X)$ is a logarithmic height of $X$ (Definition \ref{classical height of hypersurface}), and the constants $C_1(n)$ depending on $n$ will be given explicitly in \S \ref{uniform lower bound - final version}.
\end{theo}

A direct calculation gives
\[\frac{{n+D-\delta\choose n}}{D\left({n+D\choose n}-{n+D-\delta\choose n}\right)}=\frac{1}{n\delta}+O(D^{-1}),\]
and we will give an explicit estimate of the remainder in Lemma \ref{control of R(D)}. 
\subsubsection{}
Compared with the asymptotic result, $1/n\delta$ gives the optimal dependence on the height of $X$. Since can compare the relevant heights of $X$ explicitly (see \cite[\S3,\S4]{BGS94}, \cite[\S B.7]{Hindry}, \cite[Proposition 3.6]{Chen1}, \cite{Liu-reduced}, and Proposition \ref{height na\"ive-slope} for the comparison), the choice of which height of $X$ to use in the statement of Theorem \ref{upper and lower bound of arithmetic Hilbert-intro} is not important if we do not care about the constant $C_1(n)$ in that theorem.

In Theorem \ref{upper and lower bound of arithmetic Hilbert-intro}, we consider a special case, but we obtain the optimal dependence on the height of varieties, which improves upon those in \cite{David_Philippon99} and \cite[Theorem 4.8]{Chen1}.

In \cite[Proposition 3.8]{Odaka2018}, for a Hermitian line bundle $\overline{\mathcal L}$ on an arithmetic variety $\mathscr X\rightarrow\spec\O_K$ under the hypotheses of that proposition, the author proved the following asymptotic estimate as $D\to\infty$, with $\F_D$ here equipped with the $L^2$-norms defined using the K\"ahler metrics multiplied by $D$:
\[
\adeg_n(\F_D)
=
\frac{\overline{\mathcal L}^{d+1}}{(d+1)!}D^{d+1}
-\frac{\mathcal L^d}{4(d-1)!}D^d\log D
-\frac{\overline{\mathcal L}^d\cdot\overline{K_{\mathscr X/\O_K}}}{2 d!}D^d
+o(D^d),
\]
where $\overline{K_{\mathscr X/\O_K}}$ is the canonical bundle equipped with some induced norms. The arithmetic intersection products in this formula are normalized by $[K:\Q]$. Example \ref{ex:optimality-height-coefficient} proves that the coefficient of $h(X)$ in Theorem \ref{upper and lower bound of arithmetic Hilbert-intro} is optimal for every fixed $D\geqslant\delta$.

Theorem \ref{upper and lower bound of arithmetic Hilbert-intro} also holds when $X$ is a hypersurface in a linear subspace of $\mathbb P(\sE_K)$; see \S \ref{uniform lower bound - final version} for details.
\subsection{An application in counting rational points}
The determinant method is one of the main methods for counting rational points of bounded height. It was introduced by Bombieri and Pila \cite{Bombieri_Pila} and generalized by Heath-Brown \cite{Heath-Brown} to higher dimensions. Its key ingredient is the construction of several auxiliary hypersurfaces in $\mathbb P(\sE_K)$ which cover all rational points of bounded height in $X$ but do not contain the generic point of $X$. One then optimizes the number and the maximal degree of these hypersurfaces.

The determinant method can be formulated in the framework of Arakelov geometry, where some estimates are from a calculation of determinants appearing in the application of Siegel's lemma, which is replaced by the evaluation map in the slope method of Arakelov geometry \cite{BostBour96}; see \cite[Theorem 3.1]{Chen2}, \cite[Theorem 3.1]{Liu-global_determinant} and \cite{Liu2026} for more details on this formulation.

If we study the Arakelov formulation of the determinant method, the uniform lower bound of the arithmetic Hilbert--Samuel function will be in a significant role.
\subsubsection{}
Applying Theorem \ref{upper and lower bound of arithmetic Hilbert-intro} to \cite[Proposition 2.12]{Chen1}, we obtain the following result.
\begin{theo}\label{covered by one hypersurface-intro}
  Let $K$ be a number field, $X$ be an integral hypersurface in $\mathbb P^n_K$ of degree $\delta$ with $n\geqslant2$, and $H_K(X)$ be the height of $X$ defined in Definition \ref{classical height of hypersurface}. Let $S(X;B)$ be the set of rational points of $X$ whose heights (see Definition \ref{weil height} and Definition \ref{arakelov height}) are bounded by $B\geqslant1$. For every $\epsilon>0$, if
  \[H_K(X)\gg_{n,K,\delta,\epsilon}B^{n\delta+\epsilon},\]
  then $S(X;B)$ can be covered by one hypersurface of degree smaller than or equal to \[
\left[ \frac{(n+1)(\delta-1)(n\delta+\epsilon)}{2\epsilon} \right]+1,
\] which does not contain the generic point of $X$.
\end{theo}
The implicit constant depending on $n$, $K$, $\delta$ and $\epsilon$ in Theorem \ref{covered by one hypersurface-intro} is made explicit in Theorem \ref{covered by one hypersurface}.

These are useful in the determinant method (see \cite{Heath-Brown,Salberger07,Salberger_preprint2013,CCDN2020,ParedesSasyk2022} for the classic version and \cite{Chen1,Chen2,Liu-global_determinant} for the approach of Arakelov geometry). In \cite[Theorem 4]{Heath-Brown}, \cite[Lemma 3]{Browning_Heath05}, \cite[Lemma 6.3]{Salberger07}, and \cite[Lemma 1.7]{Salberger_preprint2013}, the exponent of $B$ in similar results to Theorem \ref{covered by one hypersurface-intro} is of $O_{n}(\delta^3)$.

\subsubsection{}
Similarly, applying Theorem \ref{upper and lower bound of arithmetic Hilbert-intro} to \cite[Theorem 3.1]{Liu-global_determinant} yields the following result.
\begin{theo}\label{improved global determinant method}
Let $X$ be a geometrically integral hypersurface in $\mathbb P^n_K$ of degree $\delta\geqslant2$, where $n\geqslant2$ and $B\geqslant1$. Then there exists a hypersurface in $\mathbb P^n_K$ of degree $\varpi$ which covers $S(X;B)$ but does not contain the generic point of $X$. The degree $\varpi$ satisfies
\[\varpi\ll_{n,K,\epsilon}\delta^2+B^{n/\left((n-1)\delta^{1/(n-1)}\right)+\epsilon}\frac{\delta^{4-1/(n-1)}b'(\mathscr X)}{H_K(X)^{\frac{1}{(n-1)\delta^{1+1/(n-1)}}}},\ \forall\ \epsilon>0, \]
where $b'(\mathscr X)$ is an upper bound of some non-geometrically integral reductions of the Zariski closure of $X$ in $\mathbb P^n_{\O_K}$, and see \eqref{constant b'(X)} for the precise definition.
\end{theo}

Compared with \cite[Theorem 1.3]{Walsh_2015} and \cite[Theorem 5.4]{Liu-global_determinant}, the application of Theorem \ref{upper and lower bound of arithmetic Hilbert-intro} in Theorem \ref{improved global determinant method} improves the dependence on the height of varieties. It also allows us to consider more general cases than those in \cite[Theorem 3.1.1]{CCDN2020}, and \cite[Theorem 5.14]{ParedesSasyk2022}, and see also \cite[(2.8)]{Ellenberg-Venkatesh2005} for a similar involvement of height in the determinant method.
\subsection{Organization of article}
This article is organized as follows: in \S \ref{chap. 2}, we provide the basic setting of the whole problem. In \S \ref{uniform bound of arithmetic Hilbert}, we provide an estimate of the Arakelov degree of the global section space of projective spaces equipped with some particular norms. In \S \ref{case of hypersurface}, we compare the norm on the global section spaces, and we apply these comparison results to a short exact sequence of global section spaces to obtain a lower bound of arithmetic Hilbert--Samuel function of hypersurface. In \S \ref{application in the determinant method}, we compare some height functions of arithmetic varieties, and apply the above lower bound to improve some estimates in the determinant method of Arakelov formulation.

In Appendix \ref{dominant terms of C(n,D)}, we give a uniform explicit estimate of the arithmetic Hilbert--Samuel function of projective spaces with respect to the symmetric norm, which improves that in \cite[Annexe]{Gaudron08}. 
\subsection*{Acknowledgement}
The author would like to thank Prof. Huayi Chen for his suggestions on this work. In addition, the author would like to thank Prof. Xinyi Yuan for some helpful discussions, and for Dr. Yijun Yuan for his help in carrying out some elementary calculations in this work by AI tools. Chunhui Liu was supported by National Science Foundation of China No. 12201153.
\section{Fundamental settings}\label{chap. 2}
In this section, we will provide the fundamental settings of this subject. Let $K$ be a number field, and $\O_K$ be its ring of integers. We denote by $M_{K,f}$ the set of finite places of $K$, and by $M_{K,\infty}$ the set of infinite places of $K$. In addition, we denote by $M_K=M_{K,f}\sqcup M_{K,\infty}$ the set of places of $K$.

For every $v\in M_K$, we define the absolute value \[|x|_v=\left|N_{K_v/\Q_v}(x)\right|_v^\frac{1}{[K_v:\Q_v]}\] for each $v\in M_K$, which extends the usual absolute values on $\Q_p$ or $\mathbb{R}$.

With all the above notations, we have the product formula (cf. \cite[Chap. III, Proposition 1.3]{Neukirch})
\begin{equation}\label{product formula}
\prod_{v\in M_K}|x|_v^{[K_v:\Q_v]}=1
\end{equation}
for all $x\in K^\times$.
\subsection{Normed vector bundles}
First we introduce the notion of a normed vector bundle.
\subsubsection{}
A \textit{normed vector bundle} over $\spec\O_K$ is a pair $\E=\left(E,\left(\|\ndot\|_v\right)_{v\in M_{K,\infty}}\right)$, where:
\begin{itemize}
\item $E$ is a projective $\O_K$-module of finite rank;
\item $\left(\|\ndot\|_v\right)_{v\in M_{K,\infty}}$ is a family of norms, where $\|\ndot\|_v$ is a norm over $E\otimes_{\O_K,v}\C$ invariant under the action of $\gal(\C/K_v)$.
\end{itemize}

For a normed vector bundle $\E$ over $\spec\O_K$, we define the \textit{rank} of $\E$ as that of $E$, denoted by $\rg_{\O_K}(E)$ or $\rg(E)$ if there is no ambiguity.

If for every $v\in M_{K,\infty}$, the norm $\|\ndot\|_v$ is Hermitian, i.e. induced by an inner product, we call $\E$ a \textit{Hermitian vector bundle} over $\spec\O_K$. If $\rg_{\O_K}(E)=1$, we say that $\E$ is a \textit{Hermitian line bundle} over $\spec\O_K$.
\subsubsection{}
Let $E$ be a $\O_K$-module, and $F$ be a sub-$\O_K$-module of $E$. We say that $F$ is a \textit{saturated} sub-$\O_K$-module if $E/F$ is a torsion-free $\O_K$-module.

Let $\E=\left(E,\left(\|\ndot\|_{E,v}\right)_{v\in M_{K,\infty}}\right)$ and $\F=\left(F,\left(\|\ndot\|_{F,v}\right)_{v\in M_{K,\infty}}\right)$ be two normed vector bundles over $\spec\O_K$. If $F$ is a saturated sub-$\O_K$-module of $E$ and $\|\ndot\|_{F,v}$ is the restriction of $\|\ndot\|_{E,v}$ over $F\otimes_{\O_K,v}\C$ for every $v\in M_{K,\infty}$, we say that $\F$ is a \textit{Hermitian subbundle} of $\E$ over $\spec\O_K$.

We say that $\G=\left(G,\left(\|\ndot\|_{G,v}\right)_{v\in M_{K,\infty}}\right)$ is a \textit{quotient normed vector bundle} of $\E$ over $\spec\O_K$, if the module $G$ is a projective quotient of $E$ and $\|\ndot\|_{G,v}$ is the induced quotient space norm of $\|\ndot\|_{E,v}$ for every $v\in M_{K,\infty}$.
\subsubsection{}
Let $\overline E$ be a normed vector bundle over $\spec\O_K$, $\p\in\spm\O_K$ corresponding to a place $v\in M_{K,f}$. For all $x\in E\otimes_{\O_K}K_\p$, we define
\begin{eqnarray*}
\|x\|_v=\|x\|_{E,\p}&=&\inf\{|a|_{\p}\mid a\in K_\p^\times,\; a^{-1}x\in E\otimes_{\O_K}\widehat{\O}_{K,\p}\}\\
&=&\inf\{|a|_{\p}\mid a\in K^\times,\; a^{-1}x\in E\},
\end{eqnarray*}
where the last equality is valid for $x\in E\otimes_{\O_K}K$. This is called the \textit{norm given by model} at $\p\in\spm\O_K$.
\subsubsection{}
Let
\[\begin{CD}
  0@>>>F@>>>E@>>>G@>>>0
\end{CD}\]
be an exact sequence of $\O_K$-modules. If for all $v\in M_{K,\infty}$, $\F$ is equipped with the subspace norms of $\E$, and $\G$ is equipped with the quotient norms of $\E$, then we say that
\[\begin{CD}
  0@>>>\F@>>>\E@>>>\G@>>>0
\end{CD}\]
is an exact sequence of normed vector bundles over $\spec\O_K$, and we denote $\G=\E/\F$.

For simplicity, we denote by $E_K=E\otimes_{\O_K}K$ below.

\subsection{Some Arakelov invariants}
In this part we introduce the definitions and properties of some invariants in Arakelov geometry.
\subsubsection{}
Let $\E$ be a normed vector bundle over $\spec\O_K$, and $s_1,\ldots, s_r$ be a $K$-basis of $E_K$. We define the \textit{Arakelov degree} of $\overline E$ as
\[\adeg(\E)=-\sum_{v\in M_K}[K_v:\Q_v]\log\|s_1\wedge\cdots\wedge s_r\|_v,\]
where
\[\|s_1\wedge\cdots\wedge s_r\|_v=\inf_{\begin{subarray}{c}e_1,\ldots,e_r\in E_{K}\\e_1\wedge\cdots\wedge e_r=s_1\wedge\cdots\wedge s_r\end{subarray}}\|e_1\|_v\cdots\|e_r\|_v\]
for each $v\in M_{K,\infty}$. The Arakelov degree of a normed vector bundle is independent of the choice of the basis $\{s_1,\ldots,s_r\}$ by the product formula.

In particular, if $\E$ is a Hermitian vector bundle over $\spec\O_K$, and $\{s_1,\ldots,s_r\}$ is a $K$-basis of $E_K$ with $s_1,\ldots,s_r\in E$, then we have
\begin{eqnarray*}
  \adeg(\E)&=&-\sum_{v\in M_{K}}[K_v:\Q_v]\log\left\|s_1\wedge\cdots\wedge s_r\right\|_v\\
  &=&\log\left(\#\left(E/(\O_Ks_1+\cdots+\O_Ks_r)\right)\right)\\
  & &-\frac{1}{2}\sum_{v\in M_{K,\infty}}[K_v:\Q_v]\log\det\left(\langle s_i,s_j\rangle_{1\leqslant i,j\leqslant r,v}\right),
\end{eqnarray*}
where $\langle\ndot,\ndot\rangle_v$ is the inner product which induces the norm $\|\ndot\|_v$ for $v\in M_{K,\infty}$, and $\langle s_i,s_j\rangle_{1\leqslant i,j\leqslant r,v}$ is the Gram matrix of the basis $\{s_1,\ldots,s_r\}$.

We refer the readers to \cite[\S 2.4.1]{Gillet-Soule91} for a proof of the equivalence of the above two definitions.

We define
\[\adeg_n(\E)=\frac{1}{[K:\Q]}\adeg(\E)\]
as the \textit{normalized Arakelov degree} of $\E$, which is invariant under the base-field extension.
\subsubsection{}
Let $\E$ be a normed vector bundle over $\spec\O_K$, and $s\in E_K$ be a nonzero element. We define
\[\adeg(s)=-\sum_{v\in M_K}[K_v:\Q_v]\log\|s\|_v,\text{ and }\adeg_n(s)=\frac{\adeg(s)}{[K:\Q]}.\]
\subsubsection{}
Let $\E$ be a nonzero Hermitian vector bundle over $\spec\O_K$, and $\rg(E)$ be the rank of $E$. The \textit{slope} of $\E$ is defined as
\[\wmu(\E):=\frac{1}{\rg(E)}\adeg_n(\E).\]
In addition, we denote by $\wmu_{\max}(\E)$ the maximal value of slopes of all nonzero Hermitian subbundles, and by $\wmu_{\min}(\E)$ the minimum of the slopes of all nonzero Hermitian quotient bundles of $\E$.

Let $\F$ be a Hermitian quotient bundle of $\E$ over $\spec\O_K$. We have
\begin{equation}\label{slope of quotient>minimal slope}
\wmu(\F)\geqslant\wmu_{\min}(\E)
\end{equation}
directly from the definition.
\subsubsection{}\label{additivity of Arakelov degree in exact sequence}
Let
\[\begin{CD}
  0@>>>\F@>>>\E@>>>\G@>>>0
\end{CD}\]
be an exact sequence of Hermitian vector bundles. In this case, we have
\begin{equation}\label{Po}
  \adeg(\E)=\adeg(\F)+\adeg(\G).
\end{equation}
We refer the reader to \cite[(3.3)]{Bost_Kunnemann} for a proof of the equality \eqref{Po}.
\subsection{Arithmetic Hilbert--Samuel function}\label{basic setting}
Let $\overline{\mathcal E}$ be a Hermitian vector bundle of rank $n+1$ over $\spec\O_K$, and $\mathbb P(\sE)$ be the projective space which represents the functor from the category of commutative $\O_K$-algebras to the category of sets mapping all $\O_K$-algebra $A$ to the set of projective quotient $A$-module of $\sE\otimes_{\O_K}A$ with rank $1$.

Let $\O_{\mathbb P (\sE)}(1)$ (also denoted $\O_{\sE}(1)$ or $\O(1)$ when no confusion can arise) be the universal bundle on $\mathbb P(\sE)$, and $\O_{\mathbb P (\sE)}(D)$ (or $\O_{\sE}(D)$, $\O(D)$ for simplicity) be the line bundle $\O_{\mathbb P (\sE)}(1)^{\otimes D}$. For each $v\in M_{K,\infty}$, the Hermitian metrics on $\overline{\sE}$ induce by quotient of Hermitian metrics (i.e. Fubini--Study metrics) on $\O_{\mathbb P(\sE)}(1)$ which define a Hermitian line bundle $\overline{\O_{\mathbb P(\sE)}(1)}$ on $\mathbb P(\sE)$.
\subsubsection{}
For every $D\in\mathbb N^+$, let
\begin{equation}\label{definition of E_D}
  E_D=H^0\left(\mathbb P(\sE),\O_{\mathbb P (\sE)}(D)\right),
\end{equation} and let $r(n,D)$ be its rank over $\O_K$. In fact, we have
\begin{equation}\label{def of r(n,D)}
  r(n,D)={n+D\choose D}.
\end{equation}

For each $v\in M_{K,\infty}$, let $\|\ndot\|_{\sup,v}$ be the norm on $E_{D,v}=E_D\otimes_{\O_K,v}\C$ defined by
\begin{equation}\label{definition of sup norm}
  \forall\;s\in E_{D,v},\;\|s\|_{\sup,v}=\sup_{x\in\mathbb P(\sE_K)_v(\C)}\|s(x)\|_{\mathrm{FS},v},
\end{equation}
where $\|\ndot\|_{\mathrm{FS},v}$ is the corresponding Fubini--Study norm on $\O(D)$. In this case, $(E_D,(\|\ndot\|_{\sup,v})_{v\in M_{K,\infty}})$ is a normed vector bundle over $\spec\O_K$.
\subsubsection{}\label{metric of John}
Next, we will introduce the so-called \textit{metric of John}, and see \cite{Thompson96} for a systematic introduction of this notion. In general, for a given symmetric convex body $C$, there exists the unique ellipsoid, called \textit{ellipsoid of John}, contained in $C$ whose volume is maximal.

For the $\O_K$-module $E_D$ and any place $v\in M_{K,\infty}$, we take the ellipsoid of John of its unit closed ball defined via the norm$\|\ndot\|_{\sup,v}$, and this ellipsoid induces a Hermitian norm, noted by $\|\ndot\|_{J,v}$. For every section $s\in E_{D}$, the inequality
\begin{equation}\label{john norm}
  \|s\|_{\sup,v}\leqslant\|s\|_{J,v}\leqslant\sqrt{r(n,D)}\|s\|_{\sup,v}
\end{equation}
holds by \cite[Theorem 3.3.6]{Thompson96}. In fact, these constants do not depend on the choice of the symmetric convex body.

By the estimate \eqref{john norm}, we have the proposition below, whose statement is valid for general normed vector bundles, and only the Hermitian case is needed in this article.
\begin{prop}[\cite{Chen10b}, Proposition 2.1.14]\label{slope of different norms}
  Let $\E$ be a normed vector bundle of rank $r>0$ over $\spec \O_K$. Then we have the following inequalities:
\begin{equation*}
  \wmu(\E)-\frac{1}{2}\log r\leqslant\wmu(\E_J)\leqslant\wmu(\E),
  \end{equation*}
  \begin{equation*}
  \wmu_{\max}(\E)-\frac{1}{2}\log r\leqslant\wmu_{\max}(\E_J)\leqslant\wmu_{\max}(\E),
\end{equation*}
  \begin{equation*}
  \wmu_{\min}(\E)-\frac{1}{2}\log r\leqslant\wmu_{\min}(\E_J)\leqslant\wmu_{\min}(\E),
\end{equation*}
where $\E_J$ is the Hermitian vector bundle equipped with the norms of John induced from the original norms for each $v\in M_{K,\infty}$.
\end{prop}
\subsubsection{}\label{Bombieri norm}
Let $A$ be a ring, and $E$ be an $A$-module. We denote by $\operatorname{Sym}^D_{A}(E)$ the symmetric product of degree $D$ of the $A$-module $E$, or by $\operatorname{Sym}^D(E)$ if there is no confusion on the base ring.

Regard $E_D$ defined in \eqref{definition of E_D} as an $\O_K$-module, we have an isomorphism of $\O_K$-modules $E_D\cong\operatorname{Sym}^D(\mathcal{E})$. Then for every place $v\in M_{K,\infty}$, the Hermitian norm $\|\ndot\|_v$ over $\mathcal{E}_{v,\C}$ induces a Hermitian norm $\|\ndot\|_{\sym,v}$ by the symmetric product. More precisely, this norm is the quotient norm induced by the quotient morphism
\[\sE^{\otimes D}\rightarrow\operatorname{Sym}^D(\sE),\]
where the vector bundle $\overline{\sE}^{\otimes D}$ is equipped with the norms of tensor product of $\overline{\sE}$ over $\spec\O_K$ (cf. \cite[D\'efinition 2.10]{Gaudron08} for the definition). We call $\|\ndot\|_{\sym,v}$ the \textit{symmetric norm} on $\operatorname{Sym}^D(\sE)$ at $v\in M_{K,\infty}$.
\begin{rema}
The norm $\|\ndot\|_{\sym,v}$ is often called the \textit{Bombieri norm} arising from an initial study of its properties in \cite{BBEM1990}.
\end{rema}
The following proposition gives an explicit description of the symmetric norm introduced above.
\begin{prop}\label{prop:symmetric-norm}
Let \((E,\|\ndot\|)\) be a Hermitian complex vector space of dimension \(n+1\). Under a unitary identification \(E\cong\mathbb C^{n+1}\), a homogeneous polynomial
\[
f=\sum_{\begin{subarray}{c}(i_0,\ldots,i_n)\in\mathbb N^{n+1}\\ i_0+\cdots+i_n=D\end{subarray}}a_{i_0\ldots i_n}T_0^{i_0}\cdots T_n^{i_n}=\sum_{|I|=D}a_IT^I
\]
has norm
\[
\|f\|_{\sym}^2
=
\sum_{|i|=D}\frac{I!}{D!}|a_I|^2,
\]
where $I=(i_0,\ldots,i_n)$, $I!=i_0!\cdots i_n!$. Under this identification, multiplication by a homogeneous polynomial is ordinary polynomial multiplication.
\end{prop}

\begin{proof}
We choose an orthonormal basis \(e_0,\ldots,e_n\) of \(E\). For a multi-index \(I=(i_0,\ldots,i_n)\) with \(|I|=D\), let \(u_I\) be the average of the \(D!\) tensor words obtained from a word with \(i_k\) copies of \(e_k\). Distinct \(u_I\) are orthogonal, and each distinct word occurs \(I!\) times in the average. Since there are \(\frac{D!}{I!}\) distinct choices, then
\[
\|u_I\|^2=\frac{D!}{I!}\left(\frac{I!}{D!}\right)^2=\frac{I!}{D!}.
\]
The quotient norm is the norm of the orthogonal projection onto the symmetric subspace. Hence the displayed formula follows. Concatenation of tensors induces multiplication of homogeneous polynomials, which proves the final assertion.
\end{proof}
\subsubsection{}
We recall the argument in \cite[Remark 2.6]{Chen1} below. For any place $v\in M_{K,\infty}$, the norms $\|\ndot\|_{J,v}$ and $\|\ndot\|_{\sym,v}$ are invariant under the action of the unitary group $U(\sE_{v,\C},\|\ndot\|_v)$ of order $n+1$. Then they are proportional and the ratio is independent of the choice of $v\in M_{K,\infty}$ (see \cite[Lemma 4.3.6]{BGS94} for a proof, which depends only on the inner-product structure). Let $R_0(n,D)$ be the constant such that, for every section $s\in E_{D}\otimes_{\O_K}K_v$, $s\neq0$, the equality
\begin{equation}\label{symmetric norm vs John norm}
  \log\|s\|_{J,v}=\log\|s\|_{\sym,v}+R_0(n,D)
\end{equation}
holds for all $v\in M_{K,\infty}$.
\begin{defi}\label{definition of E_D with norm}
Let $E_D$ be the $\O_K$-module defined in \eqref{definition of E_D}. We denote by $\E_{D,J}$ the Hermitian vector bundle over $\spec\O_K$ where $E_D$ is equipped with the norm of John $\|\ndot\|_{v,J}$ induced by the norms $\|\ndot\|_{v,\sup}$ defined in \eqref{definition of sup norm} for every place $v\in M_{K,\infty}$. Similarly, we denote by $\E_{D,\sym}$ the Hermitian vector bundle over $\spec\O_K$ which $E_D$ is equipped with the norms $\|\ndot\|_{\sym,v}$ introduced above for every place $v\in M_{K,\infty}$.
\end{defi}
With all the notations in Definition \ref{definition of E_D with norm}, we have the following result on the comparison of $\E_{D,J}$ and $\E_{D,\sym}$.
\begin{prop}[\cite{Chen1}, Proposition 2.7]\label{symmetric norm vs John norm, constant}
  With all the notations in Definition \ref{definition of E_D with norm}, we have
\[\wmu_{\min}(\E_{D,J})=\wmu_{\min}(\E_{D,\sym})-R_0(n,D).\]
In the above equality, the constant $R_0(n,D)$ defined in the equality \eqref{symmetric norm vs John norm} satisfies the inequality
\begin{equation*}
  0\leqslant R_0(n,D)\leqslant\log\sqrt{r(n,D)},
\end{equation*}
where the constant $r(n,D)=\rg(E_D)$ follows the definition in the equality \eqref{def of r(n,D)}.
\end{prop}
\subsubsection{}\label{definition of F_D}
Let $X$ be a pure dimensional closed subscheme of $\mathbb{P}(\mathcal{E}_K)$, and $\mathscr{X}$ be the Zariski closure of $X$ in $\mathbb{P}(\mathcal{E})$. We denote by
\begin{equation}\label{evaluation map}
\eta_{X,D}:\;E_{D,K}=H^0\left(\mathbb{P}(\mathcal{E}_K),\O_{\mathbb P(\sE_K)}(D)\right)\longrightarrow H^0\left(X,\O_{\mathbb P(\sE_K)}(1)|_X^{\otimes D}\right)
\end{equation}
the \textit{evaluation map} over $X$ induced by the closed immersion of $X$ in $\mathbb P(\sE_K)$. Let $F_D$ be the saturated image of $\eta_{X,D}$ in $H^0\left(\mathscr{X},\O_{\mathbb P(\sE)}(1)|_\mathscr{X}^{\otimes D}\right)$. In other words, $F_D$ is the largest saturated sub-$\O_K$-module of $H^0\left(\mathscr{X},\O_{\mathbb P(\sE)}(1)|_\mathscr{X}^{\otimes D}\right)$ such that $F_{D,K}=\im(\eta_{X,D})$.
\begin{rema}
When the integer $D$ is large enough, the homomorphism $\eta_{X,D}$ is surjective, which means $F_D=H^0(\mathscr{X},\O_{\mathbb P(\sE)}(1)|_\mathscr{X}^{\otimes D})$ for all $D\gg0$. See \cite[Page 10]{Chen1} for such a claim.

For some particular cases, $\eta_{X,D}$ is surjective for all $D\in\mathbb N^+$. For instance, when $X$ is a hypersurface in $\mathbb P(\sE_K)$ with $n\geqslant2$, the above assertion holds.
\end{rema}

The $\O_K$-module $F_D$ is equipped with the quotient metrics induced from $\E_{D,J}$, making it a Hermitian vector bundle over $\spec \O_K$ denoted by $\F_{D,J}$.
\begin{defi}\label{arithmetic hilbert function}
Let $\F_{D,J}$ be the Hermitian vector bundle over $\spec\O_K$ determined by the map \eqref{evaluation map}. We call the function
\[\begin{array}{rcl}
\mathbb N^+&\longrightarrow&\mathbb R\\
D&\longmapsto&\wmu(\F_{D,J})
\end{array}\]
the \textit{arithmetic Hilbert--Samuel function} of $X$ with respect to the Hermitian line bundle $\overline{\O(1)}$.
\end{defi}
\begin{rema}
In Definition \ref{arithmetic hilbert function}, we sometimes define the arithmetic Hilbert--Samuel function as the one with the image $\adeg_n(\F_{D,J})$ instead of $\wmu(\F_{D,J})$, since they are essentially equivalent. We refer the readers to \cite[Proposition 4.1, Proposition 4.4]{Chen1} for a survey on the uniform estimate of $\rg(F_D)$ as an evidence for this viewpoint. 
\end{rema}

\section{Arithmetic Hilbert--Samuel function of projective spaces}\label{uniform bound of arithmetic Hilbert}
To study $\F_{D,J}$ introduced in Definition \ref{arithmetic hilbert function}, it is necessary first to consider the case where the base scheme is a projective space. In this section, we will give both an upper boundand a lower bound of $\adeg_n(\E_{D,J})$ in Definition \ref{definition of E_D with norm}.

By Proposition \ref{symmetric norm vs John norm, constant}, it suffices to study the property $\adeg_n(\E_{D,\sym})$. In the remainder of this section, we focus on this case.
\subsection{A naive lower bound}
By \eqref{slope of quotient>minimal slope} (see also \cite[Proposition 2.9]{Chen1} and \cite[\S 2.4]{Chen2}), taking $\overline{\sE}=\left(\O_K^{\oplus(n+1)},\left(\|\ndot\|_v\right)_{v\in M_{K,\infty}}\right)$, where
\begin{equation}\label{Hermitian vector bundle with l^2-norm}
\begin{array}{rrcl}
\|\ndot\|_v:&\mathbb C_v^{\oplus(n+1)}&\longrightarrow&\mathbb R\\
&(x_0,\ldots,x_n)&\longmapsto&\sqrt{|x_0|_v^2+\cdots+|x_n|_v^2}
\end{array}\end{equation}
for every $v\in M_{K,\infty}$, gives the following naive uniform lower bound 
\begin{equation}\label{trivial estimate of F_D}
  \wmu(\F_{D,J})\geqslant\wmu_{\min}(\E_{D,J})\geqslant-\frac{1}{2}D\log(n+1),
\end{equation}
which holds uniformly for all integers $D\geqslant1$. This naive lower bound is valid for all projective varieties but is far from optimal.
\subsection{Arithmetic Hilbert--Samuel function of projective spaces}
Next, we give a uniform estimate of $\adeg_n(\E_{D,J})$ defined in Definition \ref{arithmetic hilbert function} in terms of $\adeg_n(\E_{D,\sym})$ defined in Definition \ref{definition of E_D with norm}.
\subsubsection{}
First we recall a numerical result on $\adeg_n(\E_{D,\sym})$.
\begin{prop}[\cite{Gaudron06}, Proposition 4.2]\label{arithmetic hilbert for Pn}
Let $\overline{\sE}$ be a Hermitian vector bundle over $\spec\O_K$ of rank $n+1$, $\E_{D,\sym}$ be as in Definition \ref{definition of E_D with norm} with respect to $\overline{\sE}$, and $r(n,D)=\rg(E_D)={n+D\choose D}$, $D\geqslant1$. Then we have
  \begin{equation*}
    \wmu(\E_{D,\sym})=\frac{1}{2r(n,D)}\sum\limits_{\begin{subarray}{c} i_0+\cdots+i_n=D \\ i_0,\ldots,i_n\geqslant0\end{subarray}}\log\left(\frac{D!}{i_0!\cdots i_n!}\right)+D\wmu(\overline{\mathcal{E}}).
  \end{equation*}
\end{prop}

Let
\begin{equation}\label{constant C}
C(n,D)=\sum\limits_{\begin{subarray}{c} i_0+\cdots+i_n=D \\ i_0,\ldots,i_n\geqslant0\end{subarray}}\log\left(\frac{i_0!\cdots i_n!}{D!}\right),
\end{equation}
then we have
\[\adeg_n(\E_{D,\sym})=-\frac{1}{2}C(n,D)+Dr(n,D)\wmu(\mathcal E)\]
by Proposition \ref{arithmetic hilbert for Pn}. By \eqref{symmetric norm vs John norm}, we obtain
\begin{equation*}
 \adeg_n(\E_{D,J})=-\frac{1}{2}C(n,D)+Dr(n,D)\wmu(\overline{\mathcal{E}})-r(n,D)R_0(n,D),
\end{equation*}
where $\E_{D,J}$ is equipped with the norms of John induced by the supremum norms in \S \ref{metric of John}, and the constant $R_0(n,D)$ is defined in Proposition \ref{symmetric norm vs John norm, constant}, which satisfies
\begin{equation*}
  0\leqslant R_0(n,D)\leqslant\log\sqrt{r(n,D)}.
\end{equation*}
\subsubsection{}
By the proposition in \cite[Annexe]{Gaudron08}, we have
\begin{eqnarray*}
  C(n,D)&=&-\left(\frac{1}{2}+\cdots+\frac{1}{n+1}\right)\left(D+o(D)\right)r(n,D)\\
&=&-\frac{1}{n!}\left(\frac{1}{2}+\cdots+\frac{1}{n+1}\right)D^{n+1}+o(D^{n+1})
\end{eqnarray*}
 when $n\geqslant2$ and $D\to\infty$, but the estimate of remainder is implicit. In fact, we have the following better explicit estimate (in Theorem \ref{thm:estimate of C(n,D)})
\begin{eqnarray}\label{estimation of C}
C(n,D)&=&\frac{1-\mathcal H_{n+1}}{n!}D^{n+1}+\frac1{2(n-1)!}D^n\log D\\
& &+\frac{(-n^2-2n+2)\mathcal H_n+n^2+3n-5-(n-2)\log(2\pi)}{2n!}D^n\nonumber\\
& &+R(n,D),\nonumber
\end{eqnarray}
where $\displaystyle\mathcal H_n=1+\frac{1}{2}+\cdots+\frac{1}{n}$ and $\displaystyle\abs{R(n,D)}\leqslant\left(\frac{7}{6}n^3+\frac{7}{4}n^2+\frac{7}{12}n-\frac{17}{12}+\frac{7}{12\ee}\right)D^n$. This estimate will be used later in the study of uniform bounds. Details of the estimate of $C(n,D)$ are given in Appendix \ref{dominant terms of C(n,D)}.
\section{Uniform lower bound of hypersurfaces}\label{case of hypersurface}
In this section, we give a uniform lower bound of the arithmetic Hilbert--Samuel function defined in Definition \ref{arithmetic hilbert function} in the hypersurface case. Compared with the asymptotic results, this estimate has optimal dependence on the height of varieties as $D$ tends into infinity. 
\subsection{Comparison of norms by the multiplication of a section}
Let $E_{D,K}=H^0(\mathbb P(\sE_K),\O(D))$. Let $s\in E_{D,K}$ and $t\in E_{D',K}$ for some $D,D'\in\mathbb N$. In this part, we will provide some useful results on the comparison of $\|s\cdot t\|$ and $\|s\|\cdot\|t\|$ for some particular norm $\|\ndot\|$.

\subsubsection{}
Let $k$ be a finite extension of a $p$-adic field with an absolute value $|\ndot|$ extending a $p$-adic absolute value, and
\[f=\sum_{\begin{subarray}{c}i_0,\ldots,i_n\geqslant 0\\i_0+\cdots+i_n=D\end{subarray}}a_{i_0,\ldots,i_n}T_0^{i_0}\cdots T_n^{i_n}\in k[T_0,\ldots,T_n]\]
be a homogeneous polynomial. We define
\[|f|=\max_{\begin{subarray}{c}i_0,\ldots,i_n\geqslant 0\\i_0+\cdots+i_n=D\end{subarray}}\left\{|a_{i_0,\ldots,i_n}|\right\}.\]

In the non-archimedean case, we have the classical Gauss lemma below, which was applied in the proof of \cite[Proposition B.7.3]{Hindry}.
\begin{prop}\label{Gauss lemma}
Let $k$ be a finite extension of a $p$-adic field equipped with an absolute value $|\ndot|$ extended from a $p$-adic one, and $f,g\in k[T_0,\ldots,T_n]$ be two homogeneous polynomials. Then
  \[|f\cdot g|=|f|\cdot|g|.\]
  \end{prop}
\subsubsection{}
Let $\overline{\sE}$ be a Hermitian vector bundle of rank $n+1$ over $\spec\O_K$, with $n\geqslant2$, and let $0\neq s_K\in H^0(\mathbb P(\sE_K),\O(\delta))$. Let $\mathscr X$ be the Zariski closure of its zero locus, and put $S=E_\delta\cap Ks_K$, equipped with the induced symmetric norms. Multiplication gives an exact sequence of $\O_K$-modules
\begin{equation}\label{exact sequence of F_D}
\begin{CD}
0@>>>S\otimes_{\O_K}E_{D-\delta}@>>>E_D@>>>F_D@>>>0.
\end{CD}
\end{equation}
Here $D\geqslant\delta$. Locally over $\O_K$, a generator of $S$ is a primitive equation; Gauss's lemma identifies its multiples with the saturated kernel. The vanishing of $H^1(\mathbb P^n,\O(D-\delta))$ identifies the quotient with the global sections on $\mathscr X$.

Let $r(n,D)=\rg(E_D)={n+D\choose n}$, and $r_1(n,D)=\rg(F_D)$. Then we have
\[r_1(n,D)=r(n,D)-r(n,D-\delta)={n+D\choose n}-{n+D-\delta\choose n}\]
from \eqref{exact sequence of F_D}.

The following result is a uniform version of \cite[\S 2, \S 3]{Abbes-Bouche} and \cite[Theorem 2.10]{Yuan2008}.
\begin{theo}\label{lower bound of multiplication of bombieri norm}
Let $E_D=H^0(\mathbb P(\sE),\O(D))$ be the global section space of $\O(D)$ on $\mathbb P(\sE)$, $D\in\mathbb N$, and $\E_{D,\sym}=(E_D,(\|\ndot\|_{\sym,v})_{v\in M_{K,\infty}})$. For $s\in H^0(\mathbb P(\sE),\O(\delta))\smallsetminus\{0\}$ and an integer $D\geqslant\delta$, let $\E^s_{D-\delta,\sym}=(E_{D-\delta},(\|\ndot\|^s_{\sym,v})_{v\in M_{K,\infty}})$, where $\|\ndot\|^s_{\sym,v}$ is the subspace norm on $E_{D-\delta,v}$ induced from $E_{D,v}$ via the embedding
\[\begin{array}{rcl}
E_{D-\delta}&\longrightarrow&E_D\\
t&\longmapsto&s\cdot t
\end{array}\]
for each $v\in M_{K,\infty}$. Then we have
\[\adeg_n(\E^s_{D-\delta,\sym})\leqslant\adeg_n(\E_{D-\delta,\sym})+r(n,D-\delta)\left(\adeg_n(s)+\frac{1}{2}\delta\mathcal H_n+\frac{1}{2}\log{n+\delta\choose n}\right),\]
where
\[\adeg_n(s)=-\sum_{v\in M_{K,f}}\frac{[K_v:\Q_v]}{[K:\Q]}\log\|s\|_{v}-\sum_{v\in M_{K,\infty}}\frac{[K_v:\Q_v]}{[K:\Q]}\log\|s\|_{\sym,v},\]
$\mathcal H_n=1+\cdots+\frac{1}{n}$, and $r(n,D)=\rg(E_D)={n+D\choose n}$.
\end{theo}
For a \(D\geqslant\delta\), the Hermitian vector bundle \(\overline E^s_{D-\delta,\sym}\) on the \(\mathcal O_K\)-module \(E_{D-\delta}\) is
\[
\|t\|^s_{\sym,v}=\|st\|_{\sym,v}\text{, }v\in M_{K,\infty},
\]
which is the same as that in the statement of Theorem \ref{lower bound of multiplication of bombieri norm}. 
\subsection{Proof of Theorem \ref{lower bound of multiplication of bombieri norm}}\label{calculation for the proof of Thm 4.2} In this part, we prove Theorem \ref{lower bound of multiplication of bombieri norm} in several steps. To simplify the statements of the intermediate results, we do not repeat the notation used throughout this subsection.

We first state a useful auxiliary proposition, proved later.  
\begin{prop}\label{thm:spherical}
Let \(n\geqslant1\) and \(\delta\geqslant1\), and $U(n)$ be the unitary group of order $n$. For an $x=(x_1,\ldots,x_n)\in\mathbb C^n$, let $\|x\|_2=\sqrt{|x_1|^2+\cdots|x_n|^2}$, where we ignore the dimension $n$ in this notation. Let
\[
S^{2n+1}=\{x\in\mathbb C^{n+1}\mid\|x\|_2=1\}
\]
carry normalized \(U(n+1)\)-invariant probability measure \(\sigma\), and let \(0\neq s\in\mathbb C[T_0,\ldots,T_n]\) be homogeneous of degree \(\delta\). If
\[
M=\|s\|_{\sup,S}:=\max_{x\in S^{2n+1}}|s(x)|,
\]
then \(\log|s|\in L^1(S^{2n+1},\sigma)\) and
\[
\int_{S^{2n+1}}\log|s|\,\mathrm d\sigma
\geqslant
\log M-\frac{\delta}{2}\mathcal H_n.
\]
The above equality holds if and only if \(s=c\,\ell^\delta\) for some \(c\in\mathbb C^\times\) and some nonzero complex linear form \(\ell\) on \(\mathbb C^{n+1}\).
\end{prop}

We first focus on the proof of Proposition \ref{thm:spherical} first, and then prove Theorem \ref{lower bound of multiplication of bombieri norm} with its help. Some of the calculations are inspired by \cite{BGS94}. 
\subsubsection{}
We keep the above notation and equip
\[
S^{2n+1}=\{z\in\mathbb C^{n+1}\mid\|z\|_2=1\}
\]
with normalized \(U(n+1)\)-invariant probability measure \(\sigma\). We have the following property. 

\begin{lemm}\label{lem:monomial-moments}
Let \(n\) be a nonnegative integer. For multi-indices \(I=(i_0,\ldots,i_n),J=(j_0,\ldots,j_n)\in\mathbb N^{n+1}\), write \(|I|=\sum\limits_k i_k\), \(I!=\prod\limits_k i_k!\), and \(z^I=\prod\limits_k z_k^{i_k}\). Then we have
\[
\int_{S^{2n+1}}z^I\overline{z^J}\,\mathrm d\sigma
=
\begin{cases}
0,&I\neq J,\\[2mm]
\dfrac{I!\,n!}{(|I|+n)!},&I=J.
\end{cases}
\]
\end{lemm}

\begin{proof}
We put \(q=n+1\) for simplicity. A product of one-dimensional complex Gaussian integrals gives
\[
\int_{\mathbb C^q}z^I\overline{z^J}\exp(-\|z\|_2^2)\,\mathrm d V
=
\begin{cases}
0,&I\neq J,\\
\pi^q I!,&I=J.
\end{cases}
\]
Indeed, polar coordinates in one complex variable give
\(\displaystyle\int_{\mathbb C}|z|^{2a}e^{-|z|^2}\,\mathrm d A=\pi a!\), while the angular integral vanishes for unequal exponents.

Write \(z=\rho x\), with \(\rho>0\) and \(x\in S^{2n+1}\). Since the sphere has surface area \(2\pi^q/n!\), we have
\[
\mathrm d V=\frac{2\pi^q}{n!}\rho^{2q-1}\,\mathrm d\rho\,\mathrm d\sigma(x).
\]
Therefore
\[
\int_{\mathbb C^q}z^I\overline{z^J}e^{-\|z\|_2^2}\,\mathrm d V
=
\frac{2\pi^q}{n!}
\left(\int_0^\infty
\rho^{|I|+|J|+2q-1}e^{-\rho^2}\,\mathrm d\rho\right)
\left(\int_Sx^I\overline{x^J}\,\mathrm d\sigma\right).
\]
When \(I\neq J\), independent coordinate phase rotations also force the sphere integral to vanish. When \(I=J\), the radial integral is \((|I| +n)!/2\). Equating the two Gaussian evaluations gives the stated formula.
\end{proof}

\begin{lemm}[Lemma 4.3.6, \cite{BGS94}]\label{prop:bombieri-l2}
Let \(m\) be a nonnegative integer and 
\[
f(z)=\sum_{|I|=m}a_I z^I.
\]
If
\[
\|f\|_{\sym}^2
=
\sum_{|i|=m}\frac{i!}{m!}|a_i|^2,
\]
then
\[
\|f\|_{L^2(\sigma)}^2
=
\binom{n+m}{n}^{-1}\|f\|_{\sym}^2.
\]
\end{lemm}

\begin{lemm}\label{lem:bergman}
We keep all the notations in Lemma \ref{lem:monomial-moments}. For an \(m\geqslant0\) and an index $I\in\mathbb N^{n+1}$, we define
\[
e_I(z)=\frac{z^I}{\sqrt{I!\,n!/(m+n)!}},\ |I|=m.
\]
Then \((e_I)_{|I|=m}\) is an \(L^2(\sigma)\)-orthonormal basis of the degree \(m\) homogeneous polynomials, and
\[
\sum_{|I|=m}|e_I(x)|^2=\binom{n+m}{n}, \ x\in S^{2n+1}.
\]
\end{lemm}

\begin{proof}
The orthonormality follows from Lemma \ref{lem:monomial-moments}. The multinomial theorem gives
\[
\sum_{|I|=m}\frac{m!}{I!}|x^I|^2
=
\left(\sum_{k=0}^n|x_k|^2\right)^m=1.
\]
Multiplication by \((m+n)!/(m!n!)\) yields the asserted constant value.
\end{proof}
\subsubsection{}
We begin with the following continuous Cauchy--Binet identity, generally called \textit{Andr\'eief identity}. We give a proof of the version used later.

\begin{prop}\label{thm:andreief}
Let \((X,\sigma)\) be a probability space, let \(r\geqslant1\), \(e_1,\ldots,e_r\in L^2(\sigma)\), and the function \(w:X\to[0,\infty)\) be bounded and measurable. Set the matrix $A=(A_{ij})_{1\leqslant i,j\leqslant r}$, where
\[
A_{ij}=\int_Xw(x)e_i(x)\overline{e_j(x)}\,\mathrm d\sigma(x),\text{ and }
\Delta(x_1,\ldots,x_r)=\det(e_i(x_j))_{1\leqslant i,j\leqslant r}.
\]
Then
\[
\det A
=
\frac1{r!}\int_{X^r}|\Delta(x_1,\ldots,x_r)|^2
\prod_{k=1}^rw(x_k)\,\mathrm d\sigma^{\otimes r}.
\]
\end{prop}

\begin{proof}
Cauchy--Schwarz inequality shows that every \(e_i\overline{e_j}\) is integrable, so boundedness of \(w\) permits Fubini's theorem. The Leibniz formula gives
\[
\det A=\sum_{\rho\in \mathfrak S_r}\operatorname{sgn}(\rho)I_\rho,
\]
where
\[
I_\rho
=
\int_{X^r}\prod_{i=1}^r
w(x_i)e_i(x_i)\overline{e_{\rho(i)}(x_i)}
\,\mathrm d\sigma^{\otimes r}.
\]
On the other hand, we have 
\[
|\Delta(x_1,\ldots,x_r)|^2
=
\sum_{\tau,\pi\in \mathfrak S_r}\operatorname{sgn}(\tau)\operatorname{sgn}(\pi)
\prod_{j=1}^re_{\tau(j)}(x_j)\overline{e_{\pi(j)}(x_j)}.
\]
For fixed \((\tau,\pi)\), relabel \(y_i=x_{\tau^{-1}(i)}\). The corresponding integral is \(I_{\pi\circ\tau^{-1}}\), and
\[
\operatorname{sgn}(\tau)\operatorname{sgn}(\pi)
=
\operatorname{sgn}(\pi\circ\tau^{-1}).
\]
For each \(\rho\in \mathfrak S_r\), there are exactly \(r!\) pairs satisfying \(\pi\circ\tau^{-1}=\rho\). Thus the integral on the right is \(r!\det A\), proving the identity.
\end{proof}
Using the Andr\'eief identity, we obtain the following result. 
\begin{lemm}\label{lem:projection-marginal}
We keep the above notation. Let \(e_1,\ldots,e_r\) be orthonormal in \(L^2(\sigma)\), and we put
\[
\Delta(x_1,\ldots,x_r)=\det(e_i(x_j)).
\]
Define a measure \(P\) on \(X^r\) by
\[
\mathrm d P=\frac1{r!}|\Delta|^2\,\mathrm d\sigma^{\otimes r}.
\]
Then \(P\) is a probability measure, and its \(k\)-th coordinate marginal has density
\[
\frac{1}{r}B(x)\text{, and }
B(x)=\sum_{i=1}^r|e_i(x)|^2.
\]
In particular, if \(B=r\) almost everywhere, every coordinate marginal is \(\sigma\).
\end{lemm}

\begin{proof}
Proposition~\ref{thm:andreief} with \(w\equiv1\) gives
\[
\frac1{r!}\int_{X^r}|\Delta|^2\,\mathrm d\sigma^{\otimes r}=1,
\]
since the unweighted Gram matrix is the identity. Thus \(P\) has total mass one.

For a bounded measurable function \(\varphi\), expand \(|\Delta|^2\) and integrate in \(x_2,\ldots,x_r\). The orthonormality gives
\[
\int_Xe_{\tau(j)}(x_j)\overline{e_{\pi(j)}(x_j)}\,\mathrm d\sigma(x_j)
=
\delta_{\tau(j),\pi(j)}
, \ j\geqslant2,
\]
where $\delta_{i,j}$ is the usual Kronecker symbol. A term survives only when \(\tau(j)=\pi(j)\) for all \(j\geqslant2\), which forces \(\tau=\pi\). The surviving contribution is
\[
(r-1)!\sum_{i=1}^r|e_i(x_1)|^2=(r-1)!B(x_1).
\]
After multiplying by \(1/r!\), the first marginal has density \(B/r\). Simultaneous coordinate permutations preserve \(|\Delta|^2\mathrm d\sigma^{\otimes r}\), so all coordinate marginals agree.
\end{proof}

\begin{prop}\label{prop:weighted-det}
Let \(e_1,\ldots,e_r\) be orthonormal in \(L^2(\sigma)\). Suppose
\[
\sum_{i=1}^r|e_i(x)|^2=r
\quad\text{for almost every }x,
\]
and let \(w\geqslant0\) be bounded with \(\log w\in L^1(\sigma)\), and $T_w=((T_w)_{ij})_{1\leqslant i,j\leqslant r}$ be the matrix, where 
\[
(T_w)_{ij}=\int_Xw(x)e_i(x)\overline{e_j(x)}\,\mathrm d\sigma(x).
\]
Then we have
\[
\det T_w
\geqslant
\exp\left(r\int_X\log w\,\mathrm d\sigma\right).
\]
\end{prop}

\begin{proof}
For \(\epsilon>0\), let \(w_\epsilon=w+\epsilon\). By Proposition~\ref{thm:andreief} and Lemma~\ref{lem:projection-marginal}, we have
\[
\det T_{w_\epsilon}
=
\int_{X^r}\prod_{k=1}^rw_\epsilon(x_k)\,\mathrm d P.
\]
The function \(\log(\ndot)\) is concave, so Jensen's inequality and the fact that every coordinate marginal of \(P\) is \(\sigma\) give
\[
\log\det T_{w_\epsilon}
\geqslant
\int_{X^r}\sum_{k=1}^r\log w_\epsilon(x_k)\,\mathrm dP
=
r\int_X\log(w+\epsilon)\,\mathrm d\sigma.
\]
As \(\epsilon\) decreases to zero, \(T_{w_\epsilon}\) converges to \(T_w\). The positive parts of \(\log(w+\epsilon)\) are uniformly bounded, while the negative parts increase monotonically to the negative part of \(\log w\). Hence monotone convergence yields
\[
\log\det T_w\geqslant r\int_X\log w\,\mathrm d\sigma.
\]
\end{proof}
\subsubsection{}
In this part, we compare the spherical Mahler measure with the supremum norm. 
\begin{prop}\label{prop:affine-fs}
Let \(n\geqslant1\), \(S^{2n+1}\subset\mathbb{C}^{n+1}\) carry normalized \(U(n+1)\)-invariant probability measure \(\sigma\), and 
\[\begin{array}{rrcl}
  q:&S^{2n+1}&\longrightarrow&\mathbb P^n(\mathbb{C})\\
  &x&\mapsto&[x].
\end{array}\]
In the chart
\[
U_0=\{[1:u_1:\cdots:u_n]\mid u\in\mathbb{C}^n\},
\]
the pushforward \(q_*\sigma\) has density
\[
\mathrm d\mu(u)=\frac{n!}{\pi^n}(1+\|u\|_2^2)^{-(n+1)}\,\mathrm dV(u).
\]
If \(s\) is homogeneous of degree \(\delta\), then for \(x\in S^{2n+1}\) with \(x_0\ne0\), and \(u=(x_1/x_0,\ldots,x_n/x_0)\), we have
\[
|s(x)|=|s(1,u)|(1+\|u\|_2^2)^{-\delta/2}.
\]
Consequently,
\[
\int_{S^{2n+1}}\log|s|\,\mathrm d\sigma
=
\int_{\mathbb{C}^n}
\left[
\log|s(1,u)|-\frac{\delta}{2}\log(1+\|u\|_2^2)
\right]\mathrm d\mu(u)
\]
whenever either side is defined as an extended integral.
\end{prop}

\begin{proof}
Let \(Z=(Z_0,\ldots,Z_n)\) have independent standard complex Gaussian coordinates of $\mathbb C^{n+1}$. The direction \(X=Z/\|Z\|_2\) has law \(\sigma\), and \(q(X)=[Z]\). On the full-measure set \(Z_0\ne0\), write \(z'=z_0u\). The real Jacobian determinant of the map \((z_0,u)\mapsto(z_0,z_0u)\) is \(|z_0|^{2n}\). Thus, for every nonnegative measurable function \(\Phi\), we have 
\[
\mathbb E\Phi(Z'/Z_0)
=
\int_{\mathbb{C}^n}\Phi(u)\pi^{-(n+1)}
\int_{\mathbb{C}}e^{-|z_0|^2(1+\|u\|^2)}
|z_0|^{2n}\,\mathrm dA(z_0)\,\mathrm dV(u).
\]
By polar coordinates, the inner integral is
\[
\pi n!(1+\|u\|^2)^{-(n+1)}.
\]
This proves the density formula. Writing
\[
x=\frac{x_0}{|x_0|}
\frac{(1,u)}{(1+\|u\|^2)^{1/2}},
\]
homogeneity proves the pointwise identity. The exceptional set \(x_0=0\) has measure zero, so integration gives the final result. 
\end{proof}

\begin{prop}\label{prop:polar-disintegration}
Let \(n\geqslant2\), \(\mu\) be the probability measure in Proposition~\ref{prop:affine-fs}, and \(\sigma_{2n-1}\) be normalized rotation-invariant probability on \(S^{2n-1}\subset\mathbb{C}^n\). We define
\[
\mathrm d\nu_n(w)
=
\frac n\pi\frac{|w|^{2n-2}}{(1+|w|^2)^{n+1}}\,\mathrm dA(w).
\]
Then for every nonnegative measurable function \(F\), we have
\[
\int_{\mathbb{C}^n}F(u)\,\mathrm d\mu(u)
=
\int_{S^{2n-1}}\int_{\mathbb{C}}F(w\theta)\,\mathrm d\nu_n(w)\,\mathrm d\sigma_{2n-1}(\theta).
\]
Moreover,
\[
\int_0^{+\infty} t^{n-1}(1+t)^{-(n+1)}\,\mathrm d t=\frac1n.
\]
For \(n=1\), the measure \(\mu\) on \(\mathbb{C}\) is already \(\nu_1\), so there is no separate spherical fiber.
\end{prop}

\begin{proof}
Write \(u=r\eta\) in polar coordinates. Since
\[
\mathrm d V(u)=r^{2n-1}\,\mathrm d r\,\mathrm d S(\eta)\text{, and }
|S^{2n-1}|=\frac{2\pi^n}{(n-1)!},
\]
the density of \(\mu\) gives
\[
\int F\,\mathrm d\mu
=
2n\int_0^{+\infty}
\frac{r^{2n-1}}{(1+r^2)^{n+1}}
\left(\int_{S^{2n-1}}F(r\eta)\,\mathrm d\sigma_{2n-1}(\eta)\right)\,\mathrm d r.
\]
Rotation invariance allows one to average the phase of \(r\eta\). Writing \(w=re^{i\phi}\) and \(\mathrm d A(w)=r\,\mathrm dr\,\mathrm d\phi\) gives exactly the displayed double integral. With \(t=r^2\), the normalization becomes
\[
n\int_0^{+\infty} t^{n-1}(1+t)^{-(n+1)}\,\mathrm d t
=
nB(n,1)=1,
\]
where $B(x,y)$ is the beta function. For \(n=1\), the density in Proposition~\ref{prop:affine-fs} is
\(\pi^{-1}(1+|u|^2)^{-2}\mathrm d A(u)=\mathrm d\nu_1(u)\).
\end{proof}

\begin{lemm}\label{lem:log-moment}
For \(n\in\mathbb N^+\), we have
\[
\int_{\mathbb{C}}\log(1+|w|^2)\,\mathrm d\nu_n(w)=\mathcal H_n.
\]
\end{lemm}

\begin{proof}
Polar coordinates followed by the substitution \(t=r^2\) give
\[
\int_{\mathbb{C}}\log(1+|w|^2)\,\mathrm d \nu_n(w)
=
n\int_0^\infty
\log(1+t)\frac{t^{n-1}}{(1+t)^{n+1}}\,\mathrm d t.
\]
Put \(x=t/(1+t)\). Then
\[
t^{n-1}(1+t)^{-(n+1)}\,\mathrm d t=x^{n-1}\,\mathrm d x,
\]
and the integral becomes
\[
n\int_0^1[-\log(1-x)]x^{n-1}\,\mathrm d x.
\]
Using the nonnegative expansion
\(\displaystyle-\log(1-x)=\sum\limits_{k\geqslant 1}\frac{x^k}{k}\) and monotone convergence, this equals
\[
n\sum_{k\geqslant1}\frac1{k(n+k)}.
\]
The identity
\[
\frac1{k(k+n)}=\frac1n\left(\frac1k-\frac1{k+n}\right)
\]
makes the partial sums telescope to \(\mathcal H_n/n\). Multiplying by \(n\) proves the claim.
\end{proof}

\begin{lemm}\label{lem:circle-mean-polynomial}
Let \(\displaystyle f(w)=a\prod_{j=1}^d(w-\alpha_j)\) be a nonzero polynomial of degree \(d\geqslant0\) over $\mathbb C$. For a \(\rho>0\), we write
\(\displaystyle A_f(\rho)=(2\pi)^{-1}\int_0^{2\pi}\log|f(\rho e^{i\phi})|\,\mathrm d\phi\).
Then, for every \(\rho>0\) outside the finite set \(\{|\alpha_1|,\ldots,|\alpha_d|\}\), we have
\[
A_f(\rho)=\log|a|+\sum_{j=1}^d\log\max\{\rho,|\alpha_j|\}.
\]
In particular, if \(f(0)\ne0\), then \(A_f(\rho)\geqslant\log|f(0)|\) for every such \(\rho\); and if \(d\geqslant1\) and \(\rho>\max_j|\alpha_j|\), then
\[A_f(\rho)-\log|f(0)|=\sum_{j=1}^d\log\left(\frac{\rho}{|\alpha_j|}\right)>0.\]

\end{lemm}

\begin{proof}
By additivity of the circle mean over the factorization, it suffices to treat a single factor \(w-\alpha\). For any complex \(c\) with \(|c|<1\), the uniformly convergent power series for the analytic logarithm on the unit disk gives
\[
\frac1{2\pi}\int_0^{2\pi}\log|1-ce^{i\phi}|\,\mathrm d\phi=0.
\]
Indeed, on the circle, we have
\(\displaystyle\log|1-ce^{i\phi}|=\mathrm{Re}\left(\log(1-ce^{i\phi})
\right)=-\mathrm{Re}\left(\sum_{k\geqslant1}c^k\frac{e^{ik\phi}}{k}\right)\).
Uniform convergence permits termwise integration, and each integral of \(e^{ik\phi}\) vanishes for \(k\geqslant1\).

If \(|\alpha|>\rho\), then
\(|\rho e^{i\phi}-\alpha|=|\alpha|\,|1-(\rho/\alpha)e^{i\phi}|\), so the circle mean is \(\log|\alpha|\). If \(\rho>|\alpha|\), then
\(|\rho e^{i\phi}-\alpha|=\rho\,|1-(\alpha/\rho)e^{-i\phi}|\); the substitution \(\phi\mapsto-\phi\) preserves the mean, so it is \(\log\rho\). Therefore
\(A_{w-\alpha}(\rho)=\log\max\{\rho,|\alpha|\}\) whenever \(\rho\ne|\alpha|\), which proves the displayed identity. The final claims follow from
\(\log\max\{\rho,|\alpha_j|\}\geqslant\log|\alpha_j|\) and
\(\displaystyle\log|f(0)|=\log|a|+\sum_{j=1}^d\log|\alpha_j|\).
\end{proof}
\subsubsection{}
With the above preliminaries, we now prove Proposition~\ref{thm:spherical}. 
\begin{proof}[Proof of Proposition~\ref{thm:spherical}]
Choose an \(x_0\in S^{2n+1}\) with \(|s(x_0)|=M\). Since \(s\ne0\), \(M>0\). A unitary change of coordinates preserves \(\sigma\) and \(M\), so we may assume \(x_0=(1,0,\ldots,0)\).

Let \(\mu\) be the measure introduced in Proposition~\ref{prop:affine-fs}, and we define
\[
G(u)=\log|s(1,u)|-\frac{\delta}{2}\log(1+\|u\|_2^2).
\]
For \(n\geqslant2\), Proposition~\ref{prop:polar-disintegration} gives, for
\(\theta\in S^{2n-1}\), we denote
\[
f_\theta(w)=s(1,w\theta).
\]
The function \(\log|f_\theta|\) is locally integrable, since near each zero it is a constant plus an integer multiple of \(\log|w-a|\). At infinity its growth is bounded by a constant plus \(\delta\log(1+|w|)\), which is integrable for \(\nu_n\) by Lemma~\ref{lem:log-moment}. Thus each fiber integral is finite.

The function \(f_\theta\) is a one-variable polynomial with
\(f_\theta(0)=s(x_0)\ne0\). By Lemma~\ref{lem:circle-mean-polynomial}, its circle
mean \(A_{f_\theta}(\rho)\) is at least
\(\log|f_\theta(0)|=\log M\) for every \(\rho>0\) outside the finite set of
moduli of its zeros. The radial distribution of the rotation-invariant
probability measure \(\nu_n\) is absolutely continuous with respect to
Lebesgue measure on \((0,\infty)\), so this finite exceptional set is
\(\nu_n\)-null. Integrating the circle means against \(\nu_n\), and then using
Lemma~\ref{lem:log-moment}, gives
\begin{equation}\label{tag:5.1}
h(\theta)\colon=
\int_{\mathbb{C}}
\left(
\log|f_\theta(w)|
-\frac{\delta}{2}\log(1+|w|^2)
\right)\mathrm d\nu_n(w)
\geqslant
c\colon=\log M-\frac{\delta}{2}\mathcal H_n.
\end{equation}
The case \(n=1\) is the same argument with the single fiber \(\mathbb{C}\).

Since \(G(u)=\log|s(x)|\) for
\(x=(1,u)/(1+\|u\|_2^2)^{1/2}\), we have \(G\leqslant\log M\). Put
\(A=\max\{\log M,0\}\), so that \(\displaystyle\int_{\mathbb C^n} G^+\,\mathrm d\mu\leqslant A\). For \(N>0\), set
\(G_N=\max\{G,-N\}\). Applying the nonnegative disintegration formula to the positive and negative parts gives
\[
\int_{\mathbb{C}^n}G_N\,\mathrm d\mu
=
\int_{S^{2n-1}}h_N(\theta)\,\mathrm d\sigma_{2n-1}(\theta),
\]
where
\[
h_N(\theta)
=
\int_{\mathbb{C}}\max\{G(w\theta),-N\}\,\mathrm d\nu_n(w).
\]
Because \(G_N\geqslant G\), \eqref{tag:5.1} implies \(h_N(\theta)\geqslant c\). Therefore
\[
\int_{\mathbb{C}^n}\min\{G^-,N\}\,\mathrm d\mu
\leqslant A-c.
\]
Monotone convergence yields \(\displaystyle\int_{\mathbb C^n} G^-\,\mathrm d\mu<\infty\), so \(G\in L^1(\mu)\). Proposition~\ref{prop:affine-fs} now gives
\[
\int_{S^{2n+1}}\log|s|\,\mathrm d\sigma
=
\int_{\mathbb{C}^n}G\,\mathrm d\mu
=
\int_{S^{2n-1}}h(\theta)\,\mathrm d\sigma_{2n-1}(\theta)
\geqslant c.
\]
This proves the inequality and integrability.

It remains to characterize equality. Assume equality holds. Let
\[
A_\theta(\rho)
=
\frac1{2\pi}\int_0^{2\pi}
\log|f_\theta(\rho e^{i\phi})|\,\mathrm d\phi.
\]
The radial density in Proposition~\ref{prop:polar-disintegration} is strictly positive for \(\rho>0\). Hence equality forces
\(A_\theta(\rho)=\log M\) for almost every \((\theta,\rho)\). For almost every such \(\theta\), factor
\[
f_\theta(w)=a\prod_{j=1}^d(w-\alpha_j),\  a\neq 0.
\]
Since \(f_\theta(0)=s(x_0)\ne0\), every \(\alpha_j\ne0\). By Lemma~\ref{lem:circle-mean-polynomial}, for almost every \(\rho>0\) (the
finitely many exceptional radii are Lebesgue-null), we have
\[
A_\theta(\rho)
=
\log|a|+\sum_{j=1}^d\log\max\{\rho,|\alpha_j|\},
\]
whereas
\[
\log|f_\theta(0)|
=
\log|a|+\sum_{j=1}^d\log|\alpha_j|.
\]
If \(d\geqslant1\), the strict-growth clause of Lemma~\ref{lem:circle-mean-polynomial}
gives, for every \(\rho>\max_j|\alpha_j|\),
\[
A_\theta(\rho)-\log M
=
\sum_{j=1}^d\log\frac{\rho}{|\alpha_j|}>0,
\]
which leads to a contradiction. Thus \(d=0\) almost everywhere, and \(f_\theta\) is constant for almost every \(\theta\).

Write
\[
s(1,w\theta)=\sum_{k=0}^{\delta}p_k(\theta)w^k.
\]
Each \(p_k\) is the restriction to \(S^{2n-1}\) of a homogeneous polynomial in the variables \(z_1,\ldots,z_n\). The preceding constancy implies \(p_k=0\) almost everywhere for \(k\geqslant1\). Continuity gives \(p_k=0\) on the whole sphere, and homogeneity then gives \(p_k\equiv0\) on \(\mathbb{C}^n\). Hence
\[
s(1,u)=s(x_0)
\qquad(u\in\mathbb{C}^n),
\]
so \(s=s(x_0)T_0^\delta\) in the chosen coordinates. Undoing the unitary change gives
\(s=c\,\ell^\delta\).

Conversely, suppose \(s=c\,\ell^\delta\), with \(c\ne0\) and \(\ell\ne0\). Choose a unitary map \(U\) and \(a>0\) such that
\(\ell(Uz)=az_0\). It suffices to consider \(s=ca^\delta T_0^\delta\). Proposition~\ref{prop:affine-fs} and Lemma~\ref{lem:log-moment} give
\[
\int_{S^{2n+1}}\log|z_0|\,\mathrm d\sigma=-\frac12\mathcal H_n.
\]
Therefore
\[
\int_{S^{2n+1}}\log|s|\,\mathrm d\sigma
=
\log|c|+\delta\log a-\frac{\delta}{2}\mathcal H_n.
\]
Since \(\|s\|_{\sup,S}=|c|a^\delta\), equality follows.
\end{proof}
\subsubsection{}
In this part, we prove Theorem~\ref{lower bound of multiplication of bombieri norm} using Proposition~\ref{thm:spherical}.
\begin{prop}\label{prop:degree-difference}
Let \(K\) be a number field, and let \(\overline E_{D,\sym}\) be the Hermitian vector bundle equipped with symmetric norms at infinite places same as before. Let $\delta$ be a positive integer, \(m=D-\delta\geqslant0\), and we denote by \(\displaystyle
r=r(n,m)=\binom{n+m}{n}
\) for simplicity. For \(0\ne s\in E_\delta\) and any \(K\)-basis \(t_1,\ldots,t_r\) of \(E_{m,K}\),
\[
\widehat{\deg}_n(\overline E^s_{m,\mathrm{sym}})
-
\widehat{\deg}_n(\overline E_{m,\mathrm{sym}})
=
-\sum_{v\in M_{K,\infty}}
\frac{[K_v:\mathbb Q_v]}{2[K:\mathbb Q]}
\log
\frac{\det\operatorname{Gram}_{\mathrm{sym},v}(st_1,\ldots,st_r)}
{\det\operatorname{Gram}_{\mathrm{sym},v}(t_1,\ldots,t_r)},
\]
where $\operatorname{Gram}_{\mathrm{sym},v}(t_1,\ldots,t_r)$ is the Gram matrix of the basis $t_1,\ldots,t_r$ with respect to the inner product induced by $\|\ndot\|_{\sym,v}$ at $v\in M_{K,\infty}$. 
\end{prop}
\begin{proof}
At finite places the two bundles have the same underlying lattice \(E_m\) and the same norm given by model on the determinant line by Proposition \ref{Gauss lemma}, so the finite contributions cancel. At an infinite place, the determinant line norm is the square root of the Gram determinant. Thus the local contribution is
\[
-\frac{[K_v:\mathbb Q_v]}{2[K:\mathbb Q]}
\log
\frac{\det\operatorname{Gram}_{\mathrm{sym},v}(st_1,\ldots,st_r)}
{\det\operatorname{Gram}_{\mathrm{sym},v}(t_1,\ldots,t_r)}
\]
at $v\in M_{K,\infty}$. Multiplication by \(s\neq0\) is injective on the polynomial ring over \(K_v\), so the induced form remains positive definite. Summing over the infinite places proves the formula.
\end{proof}
The lemma below states that the ratio of Gram determinants in Proposition \ref{prop:degree-difference} is independent of the choice of the basis. 
\begin{lemm}\label{lem:basis-independence}
At a fixed infinite place, let \(t=(t_1,\ldots,t_r)\) be a basis and let \(t'=tM\), with \(M\in\operatorname{GL}_r(\mathbb C)\). Then
\[
\frac{\det\operatorname{Gram}(st'_1,\ldots,st'_r)}
{\det\operatorname{Gram}(t'_1,\ldots,t'_r)}
=
\frac{\det\operatorname{Gram}(st_1,\ldots,st_r)}
{\det\operatorname{Gram}(t_1,\ldots,t_r)},
\]
where $\operatorname{Gram}$ denotes the Gram matrix with respect a fixed inner product. In particular, the ratio in Proposition~\ref{prop:degree-difference} is independent of the choice \(K\)-basis.
\end{lemm}

\begin{proof}
If \(G\) and \(H\) are the Gram matrices of \(t\) and \(st\), respectively, then the Gram matrices of \(tM\) and \((st)M\) are \(M^*GM\) and \(M^*HM\). Their determinants are respectively \(|\det M|^2\det G\) and \(|\det M|^2\det H\), so the quotient is unchanged.
\end{proof}

\begin{theo}\label{thm:archimedean}
Let \(n\in\mathbb N^+\), \(\delta\geqslant1\), \(m\geqslant0\), and \(r=\binom{n+m}{n}\). Let \(V_m\) be the space of degree \(m\) homogeneous polynomials on \(\mathbb C^{n+1}\), equipped with the norm induced by the symmetric product, and let \(s\) be a nonzero homogeneous polynomial of degree \(\delta\). For any basis \(t=\{t_1,\ldots,t_r\}\) of \(V_m\), the ratio
\[
R(s;t):=
\frac{\det\operatorname{Gram}_{\sym}(st_1,\ldots,st_r)}
{\det\operatorname{Gram}_{\sym}(t_1,\ldots,t_r)}
\]
satisfies
\[
\log R(s;t)
\geqslant
2r\left(
\log\|s\|_{\mathrm{sym}}
-\frac{\delta}{2}\mathcal H_n
-\frac12\log\binom{n+\delta}{n}
\right).
\]
\end{theo}

\begin{proof}
By Lemma~\ref{lem:basis-independence}, we choose an \(L^2(\sigma)\)-orthonormal basis \(e=\{e_1,\ldots,e_r\}\) of \(V_m\). Lemma~\ref{prop:bombieri-l2} gives
\[
\operatorname{Gram}_{\mathrm{sym}}(e_1,\ldots,e_r)
=
r(n,m)I_r.
\]
We denote by $\operatorname{Gram}_{\mathrm{sym}}(se_1,\ldots,se_r)=\left(\operatorname{Gram}_{\mathrm{sym}}(se_i,se_j)\right)_{1\leqslant i,j\leqslant r}$. Polarization of the same identity gives
\[
\operatorname{Gram}_{\mathrm{sym}}(se_i,se_j)
=
r(n,m+\delta)\,T_{ij}\text{, where }T_{ij}=\int_{S^{2n+1}}|s(x)|^2e_i(x)\overline{e_j(x)}\,d\sigma(x).
\]
Let $T=(T_{ij})_{1\leqslant i,j\leqslant r}$. Then
\[
R(s;e)
=
\left(\frac{r(n,m+\delta)}{r(n,m)}\right)^r\det T
\geqslant \det T.
\]
By Lemma~\ref{lem:bergman}, the Bergman function $\displaystyle\sum_{|I|=m}|e_I(x)|^2$ of the \(e_i\) is the constant \(r(n,m)\). Applying Proposition~\ref{prop:weighted-det} with \(w=|s|^2\) and Proposition~\ref{thm:spherical}, we obtain
\[
\det T
\geqslant
\exp\left(2r\int_S\log|s|\,\mathrm d\sigma\right)
\geqslant
\exp\left(2r\left(\log\|s\|_{\sup,S}-\frac{\delta}{2}\mathcal H_n\right)\right).
\]
Finally, Proposition~\ref{prop:bombieri-l2} in degree \(\delta\) gives
\[
\|s\|_{\mathrm{sym}}^2
=
\binom{n+\delta}{n}\|s\|_{L^2(\sigma)}^2
\leqslant
\binom{n+\delta}{n}\|s\|_{\sup,S}^2.
\]
Therefore
\[
\log\|s\|_{\sup,S}
\geqslant
\log\|s\|_{\mathrm{sym}}
-\frac12\log\binom{n+\delta}{n},
\]
and the claimed estimate follows.
\end{proof}

\begin{prop}\label{prop:finite-norm}
Let \(s\in E_\delta=\operatorname{Sym}^\delta(\mathcal E)\) be a nonzero section and let \(v\in M_{K,f}\). For the norm given by model
\[
\|x\|_v
=
\inf\left\{|a|_v\mid
a\in K_v^\times,\ 
a^{-1}x\in E_\delta\otimes\widehat{\mathcal O}_{K,v}
\right\},
\]
one has \(\|s\|_v\leqslant1\). Consequently,
\[
-\sum_{v\in M_{K,f}}
\frac{[K_v:\mathbb Q_v]}{[K:\mathbb Q]}\log\|s\|_v
\geqslant0.
\]
\end{prop}

\begin{proof}
The element \(s\) belongs to the lattice
\(E_\delta\otimes_{\mathcal O_K}\widehat{\mathcal O}_{K,v}\). Thus \(a=1\) is admissible in the defining infimum, and \(\|s\|_v\leqslant1\). Hence \(\log\|s\|_v\leqslant0\), and the weighted sum with the preceding minus sign is nonnegative.
\end{proof}
Finally, we prove Theorem~\ref{lower bound of multiplication of bombieri norm} using the above preparations. 
\begin{proof}[Proof of Theorem~\ref{lower bound of multiplication of bombieri norm}]
Write \(m=D-\delta\) for simplicity. Proposition~\ref{prop:degree-difference} gives
\[
\Delta\colon=
\widehat{\deg}_n(\overline E^s_{m,\mathrm{sym}})
-
\widehat{\deg}_n(\overline E_{m,\mathrm{sym}})
=
-\sum_{v\in M_{K,\infty}}
\frac{[K_v:\mathbb Q_v]}{2[K:\mathbb Q]}\log R_v,
\]
where 
\[R_v=\frac{\det\operatorname{Gram}_{\mathrm{sym},v}(st_1,\ldots,st_r)}
{\det\operatorname{Gram}_{\mathrm{sym},v}(t_1,\ldots,t_r)}\]
is the corresponding Gram ratio in Proposition \ref{prop:degree-difference} at $v\in M_{K,\infty}$. We apply Theorem~\ref{thm:archimedean} at all complex places.

At a real place, we complexify the real Hermitian fiber. The induced Hermitian norm, sphere measure, and multiplication map are the complex ones applied in Theorem~\ref{thm:archimedean}, while the degree formula retains the local weight \([K_v:\mathbb Q_v]\). Thus, at every infinite place, we have
\[
\log R_v
\geqslant
2r(n,m)\left(
\log\|s\|_{\mathrm{sym},v}
-\frac{\delta}{2}\mathcal H_n-\frac12\log\binom{n+\delta}{n}
\right).
\]
Summing over all $v\in M_{K,\infty}$, we obtain
\begin{equation}\label{tag:6.1}
\Delta
\leqslant
r(n,m)\left(
-\sum_{v\mid\infty}
\frac{[K_v:\mathbb Q_v]}{[K:\mathbb Q]}\log\|s\|_{\mathrm{sym},v}
+\frac{\delta}{2}\mathcal H_n+\frac12\log\binom{n+\delta}{n}
\right),
\end{equation}
because the infinite local weights sum to one.

The finite-place contribution in \(\widehat{\deg}_n(s)\) is nonnegative by Proposition~\ref{prop:finite-norm}. Therefore the expression in parentheses in \eqref{tag:6.1} is at most
\[
\widehat{\deg}_n(s)+\frac{\delta}{2}\mathcal H_n+\frac12\log\binom{n+\delta}{n}.
\]
Since \(r=r(n,D-\delta)\), the asserted inequality follows. If \(D=\delta\), then \(m=0\) and \(r(n,m)=1\), and the same argument applies to the degree-zero basis vector \(1\).
\end{proof}
\subsubsection{}
We obtain a lower bound for $\adeg_n(\F_{D,J})$ from Theorem \ref{lower bound of multiplication of bombieri norm}.
\begin{prop}\label{uniform lower bound of arithmetic HS etape 1}
With the notations in Theorem \ref{lower bound of multiplication of bombieri norm}, let $\F_{D,J}$ be as in Definition \ref{arithmetic hilbert function}, where $n\geqslant2$ and $\mathscr X$ is the Zariski closure of the hypersurface defined by $0\neq s\in E_{\delta,K}$. Then we have
\begin{eqnarray*}
\adeg_n(\F_{D,J})
&\geqslant&\adeg_n(\E_{D,\sym})-\adeg_n(\E_{D-\delta,\sym})\\
& &-r(n,D-\delta)\left(\adeg_n(s)+\frac{\delta}{2}\mathcal H_n+\frac12\log\binom{n+\delta}{n}\right)\\
& &-r_1(n,D)R_0(n,D)
\end{eqnarray*}
when $D\geqslant\delta$ and $\delta\geqslant1$, where $r_1(n,D)=\rg(F_D)$, and $R_0(n,D)$ is introduced in \eqref{symmetric norm vs John norm}. In addition, we have
\begin{eqnarray*}
\frac{\wmu(\F_{D,J})}{D}&\geqslant&-\frac{r(n,D-\delta)}{Dr_1(n,D)}\adeg_n(s)-\frac{r(n,D-\delta)\left(\mathcal H_n\delta+\log{n+\delta\choose n}\right)}{2Dr_1(n,D)}\\
& &+\frac{C(n,D)-C(n,D-\delta)}{2Dr_1(n,D)}+\left(1+\frac{\delta r(n,D-\delta)}{Dr_1(n,D)}\right)\wmu(\sE)-\frac{1}{2}\log(n+1)
\end{eqnarray*}
under the same conditions. 
\end{prop}
\begin{proof}
Equip the kernel $S\otimes E_{D-\delta}$ in \eqref{exact sequence of F_D} with the subspace norms from $\E_{D,\sym}$. At finite places, Gauss's lemma identifies these norms with the tensor-product norms. At infinite places, Theorem \ref{thm:archimedean}, applied to a local generator of $S$, gives
\begin{eqnarray*}
\adeg_n\left(\overline{S\otimes E_{D-\delta}}\right)
&\leqslant&\adeg_n(\E_{D-\delta,\sym})\\
&&+r(n,D-\delta)\left(\adeg_n(s)+\frac{\delta}{2}\mathcal H_n+
\frac12\log\binom{n+\delta}{n}\right),
\end{eqnarray*}
where $\adeg_n(S)=\adeg_n(s)$. Additivity of degree in the exact sequence and \eqref{symmetric norm vs John norm} prove the first inequality.

By Proposition \ref{arithmetic hilbert for Pn}, its right-hand side, divided by $Dr_1(n,D)$, is
\begin{eqnarray*}
&&-\frac{r(n,D-\delta)}{Dr_1(n,D)}\adeg_n(s)
-\frac{r(n,D-\delta)\left(\mathcal H_n\delta+\log\binom{n+\delta}{n}\right)}{2Dr_1(n,D)}\\
&&-\frac{C(n,D)-C(n,D-\delta)}{2Dr_1(n,D)}
+\left(1+\frac{\delta r(n,D-\delta)}{Dr_1(n,D)}\right)\wmu(\sE)
-\frac{R_0(n,D)}{D}.
\end{eqnarray*}
For every $m\geqslant0$, we claim $C(n,m+1)\leqslant C(n,m)$. In fact, send each multi-index $I$ of degree $m$ to $I+(1,0,\ldots,0)$; its factorial ratio is multiplied by $(i_0+1)/(m+1)\leqslant1$, and the remaining logarithms are nonpositive. Thus $C(n,D)-C(n,D-\delta)\leqslant0$. 

Replacing its negative by its positive gives a weaker lower bound. Finally,
\[
0\leqslant\frac{R_0(n,D)}D\leqslant\frac{\log r(n,D)}{2D}
\leqslant\frac12\log(n+1),
\]
which proves the second inequality.
\end{proof}
\subsection{A numerical estimate}
The lower bound in Proposition \ref{uniform lower bound of arithmetic HS etape 1} is not explicit enough for some further applications. In this part, we provide a more explicit version and generalize it to the case where $X$ is a hypersurface in a projective subspace.
\subsubsection{}
First, we provide an explicit lower bound of $\frac{\wmu(\F_{D,J})}{D}$ by elementary calculations, which will be applied directly in the study of the determinant method. See \cite{Chen1,Chen2,Liu-global_determinant,Liu2026} for the role of this lower bound.
\begin{prop}\label{numerical result of arithmetic hilbert of hypersurface}
We keep all the notations in Proposition \ref{uniform lower bound of arithmetic HS etape 1}. For $n\geqslant2$, let the constants
\[
\begin{aligned}
B_0(n)
={}&
-\frac{\mathcal H_n+n}{2}
-\frac12\log(n+1)
\\
&\quad
-2^{n-1}(n-1)!
\Bigg[
\frac{(n+1)\mathcal H_n-n}{n!}
+\frac{1}{2(n-1)!}\left(\frac{n}{e}+1\right)
\\
&\qquad\qquad\qquad\qquad
+\frac{n}{2n!}
\left|
(2-n^2-2n)\mathcal H_n
+n^2+3n-5
+(2-n)\log(2\pi)
\right|
\\
&\qquad\qquad\qquad\qquad
+2\left(
\frac{25}{12}
+\frac{7}{12e}
+\frac{7n(n+1)(2n+1)}{12}
-\frac{7}{2}
\right)
\Bigg],
\end{aligned}
\]
which depends only on $n$, where $\displaystyle\mathcal H_n=1+\cdots+\frac{1}{n}$. Then the uniform lower bound
\begin{eqnarray*}
\frac{\wmu(\F_{D,J})}{D}&\geqslant&-\frac{r(n,D-\delta)}{Dr_1(n,D)}\adeg_n(s)+\left(1+\frac{\delta r(n,D-\delta)}{Dr_1(n,D)}\right)\wmu(\sE)\\
& &+B_0(n)
\end{eqnarray*}
holds for all $D\geqslant\delta$ and $\delta\geqslant1$.
\end{prop}
\begin{proof}
From Proposition \ref{uniform lower bound of arithmetic HS etape 1}, we need separate lower bounds for 
\[\frac{C(n,D)-C(n,D-\delta)}{2Dr_1(n,D)}\text{ and }-\frac{r(n,D-\delta)\left(\mathcal H_n\delta+\log{n+\delta\choose n}\right)}{2Dr_1(n,D)}.\] 

\medskip
\noindent\textbf{Step I. --} First we introduce some inequalities for binomial coefficients. By 
\[
\binom{n+\delta}{n}\leqslant (\delta+1)^n \leqslant e^{n\delta},
\]
we have
\[
\log\binom{n+\delta}{n}\leqslant n\delta.
\]
Hence
\[
\mathcal{H}_n\delta+\log\binom{n+\delta}{n}
\leqslant (\mathcal{H}_n+n)\delta.
\]

We also need the following combinatorial estimate, which follows from the identity
\[
r_1(n,D)=\sum_{k=0}^{\delta-1}\binom{n+D-k-1}{n-1}
\geqslant \delta \binom{n+D-\delta}{n-1}.
\]
Thus
\[
\frac{r(n,D-\delta)}{D r_1(n,D)}
\leqslant
\frac{\binom{n+D-\delta}{n}}{D\,\delta\,\binom{n+D-\delta}{n-1}}
=
\frac{D+1-\delta}{D\,n\,\delta}
\leqslant \frac1\delta.
\]
Therefore,
\begin{equation}\label{tag:A}
-\frac{r(n,D-\delta)\left(\mathcal{H}_n\delta+\log\binom{n+\delta}{n}\right)}{2D r_1(n,D)}
\geqslant
-\frac{\mathcal{H}_n+n}{2}.
\end{equation}

\medskip
\noindent\textbf{Step II. --} From \eqref{estimation of C} (see also Theorem \ref{thm:estimate of C(n,D)}), for \(n\geqslant2\), we have 
$$ C(n,D)=M_n(D)+R_n(D), $$
where
$$ M_n(D)=a_nD^{n+1}+b_nD^n\log D+c_nD^n, $$
with
$$ a_n=\frac{1-\mathcal H_{n+1}}{n!},\ b_n=\frac1{2(n-1)!}, $$
and
$$ c_n= \frac{(2-n^2-2n)\mathcal H_n+n^2+3n-5+(2-n)\log(2\pi)} {2n!}. $$

Moreover, we have
$$ |R_n(D)|\leqslant \Lambda_n D^n, $$
where we may take
$$\displaystyle \Lambda_n= \frac{25}{12}+\frac{7}{12e} +\frac{7n(n+1)(2n+1)}{12}-\frac{7}{2}.$$  

Differentiating \(M_n\) with respect to $D\in\mathbb R$ gives
$$ M_n'(D) =(n+1)a_nD^n +b_nD^{n-1}(n\log D+1) +n c_nD^{n-1}. $$
For \(x\geqslant1\), using \(\log x\leqslant x/e\), we have
$$ |M_n'(D)|\leqslant K_nD^n, $$
where
$$\displaystyle K_n=(n+1)|a_n| +\left(\frac ne+1\right)b_n +n|c_n|. $$
For $D>\delta$, the mean-value theorem gives
$$ M_n(D)-M_n(D-\delta)\geqslant-\delta K_nD^n. $$

At $D=\delta$, set $M_n(0)=R_n(0)=0$. The estimate $\log D\leqslant D/e$ gives $|M_n(D)|\leqslant K_nD^{n+1}$, so the same bound holds at this endpoint. Combining this with the estimate of the error term gives
\begin{equation}\label{estimate of C(n,D)-C(n,D-delta)}
C(n,D)-C(n,D-\delta) \geqslant-\delta K_nD^n-2\Lambda_nD^n.  
\end{equation}

\noindent\textbf{Step III. --} We claim
\begin{equation}\label{estimate of r_1(n,D)}
r_1(n,D) =\sum_{j=D-\delta+1}^{D}\binom{n+j-1}{n-1} \geqslant \frac{\delta D^{n-1}} {2^n(n-1)!}. 
\end{equation}
Indeed, if \(\delta\leqslant D/2\), every \(j\) in the sum is at least \(D/2\); if \(\delta>D/2\), at least \(D/2\) terms have \(j\geqslant D/2\). In either case \eqref{estimate of r_1(n,D)} follows.

Hence we have

$$ \frac{D^n}{D r_1(n,D)} \leqslant \frac{2^n(n-1)!}{\delta} \leqslant 2^n(n-1)!. $$

Combining \eqref{estimate of C(n,D)-C(n,D-delta)} and \eqref{estimate of r_1(n,D)}, we obtain

$$ \frac{C(n,D)-C(n,D-\delta)} {2D r_1(n,D)} \geqslant -2^{\,n-1}(n-1)!\bigl(K_n+2\Lambda_n\bigr).  $$

\noindent\textbf{Final step. --}
Combining the above estimates of
\[\frac{C(n,D)-C(n,D-\delta)}{2Dr_1(n,D)}\text{ and }-\frac{r(n,D-\delta)\left(\mathcal H_n\delta+\log{n+\delta\choose n}\right)}{2Dr_1(n,D)}\]
gives the desired result.
\end{proof}

\begin{rema}
Let $\mathscr X$ be a projective and flat scheme over $\spec\O_K$, $\overline{\mathscr L}$ be a Hermitian ample line bundle on $\mathscr X$, and $\G_D=\overline{H^0(\mathscr X,\mathscr L|_{\mathscr X}^{\otimes D})}$ be a Hermitian vector bundle equipped with the induced norms. By \cite[Lemma 4.8]{Bost2001}, there exists a constant $c_1>0$ depending only on $\mathscr X$ and $\mathscr L$ such that for any $D\in \mathbb N\smallsetminus\{0\}$, we have
\[\wmu(\G_D)\geqslant-c_1D.\]
The result of Proposition \ref{numerical result of arithmetic hilbert of hypersurface} can be viewed as an example of \cite[Lemma 4.8]{Bost2001} when $\mathscr X$ is a hypersurface and $\mathscr L=\O_{\sE}(1)|_{\mathscr X}$, since $F_D\cong H^0(\mathscr X,\O(1)|_{\mathscr X}^{\otimes D})$ when $\mathscr X\rightarrow\spec\O_K$ is a flat projective hypersurface, and the terms depending on $\adeg_n(s)$ and $\wmu(\overline{\sE})$ are all $O(D)$.
\end{rema}

\subsubsection{}
Using the lower bound for hypersurfaces in Proposition \ref{numerical result of arithmetic hilbert of hypersurface}, we consider the same issue for a projective variety lying in a linear subspace as a hypersurface. More precisely, let $\overline{\sE}$ be a Hermitian vector bundle over $\spec\O_K$, and let $\overline{\mathcal F}$ be a Hermitian quotient vector bundle of $\sE$, and then we have $\mathbb P(\mathcal F)\subset\mathbb P(\mathcal E)$ and $\mathbb P(\mathcal F_K)\subset\mathbb P(\mathcal E_K)$. Let $X$ be a projective variety embedded in $\mathbb P(\sE_K)$, which is a hypersurface in $\mathbb P(\mathcal F_K)$.

We introduce the following auxiliary result.
\begin{lemm}\label{push forward of sup norms}
Let $\overline{\sE}$ be a Hermitian vector bundle over $\spec\O_K$, and $\overline{\mathcal F}$ be a Hermitian quotient vector bundle of $\overline{\sE}$. We denote by $f:\mathbb P(\mathcal F)\hookrightarrow\mathbb P(\sE)$ the induced closed immersion. Let $X$ be a projective variety in $\mathbb P(\sE_K)$ of dimension $d$, which is a hypersurface in $\mathbb P(\mathcal F_K)$, and $\mathscr X$ be the Zariski closure of $X$ in $\mathbb P(\mathcal F)$. Then 
\begin{eqnarray*}
& &\left(H^0\left(f_*\mathscr X,\O_{\sE}(1)|_{f_*\mathscr X}^{\otimes D}\right),(\|\ndot\|_{\sup,v})_{v\in M_{K,\infty}}\right)\\
&=&\left(H^0\left(\mathscr X,\O_{\mathcal F}(1)|_{\mathscr X}^{\otimes D}\right),(\|\ndot\|_{\sup,v})_{v\in M_{K,\infty}}\right)
\end{eqnarray*}
as normed vector bundles over $\spec\O_K$, where the two families of supremum norms are defined on the global section spaces of $f_*\mathscr X$ and $\mathscr X$ respectively.
\end{lemm}
\begin{proof}
For every $v\in M_{K,\infty}$, we equip $\O_{\sE}(1)$ and $\O_{\mathcal F}(1)$ with the corresponding Fubini--Study norms. Then we define the Hermitian line bundles $\overline{\O_{\sE}(1)}$ on $\mathbb P(\sE)$ and $\overline{\O_{\mathcal F}(1)}$ on $\mathbb P(\mathcal F)$ respectively. We have $f^*\overline{\O_{\sE}(1)}=\overline{\O_{\mathcal F}(1)}$ directly by definition.

In this case, for every $D\in\mathbb N_+$ and $v\in M_{K,\infty}$, the sup norm on $H^0\left(\mathscr X,\O_{\mathcal F}(1)|_{\mathscr X}^{\otimes D}\right)$ coincides with that on $H^0\left(f_*\mathscr X,\O_{\sE}(1)|_{f_*\mathscr X}^{\otimes D}\right)$, which gives the assertion.
\end{proof}
The quotient John norms on $F_D$ also agree for these two embeddings. Indeed, the John ellipsoid of the supremum unit ball on $E_{D,\C}$ is invariant under the unitary group, so its norm is proportional to the symmetric norm. Cauchy--Schwarz inequality shows that the symmetric unit ball is contained in the supremum unit ball, while a pure power of a unit linear form has both norms equal to one. Maximality of the John ellipsoid therefore gives $\|\ndot\|_J=\|\ndot\|_{\sym}$. Symmetric powers of a Hermitian quotient carry the quotient symmetric norms, as follows by an orthogonal splitting. Transitivity of quotient norms now gives the assertion for $F_D$. Thus we obtain the following uniform lower bound of arithmetic Hilbert--Samuel function of certain arithmetic varieties.
\begin{prop}\label{uniform lower bound via push forward}
With the notations and conditions in Lemma \ref{push forward of sup norms}, assume $d\in\mathbb N^+$ and suppose the hypersurface $X$ in $\mathbb P(\sF_K)$ is determined by the nonzero global section $s\in H^0\left(\mathbb P(\sF_K),\O(\delta)\right)$, and then we have
\begin{eqnarray*}
\frac{\wmu(\F_{D,J})}{D}&\geqslant&-\frac{r(d+1,D-\delta)}{Dr_1(d+1,D)}\adeg_n(s)+\left(1+\frac{\delta r(d+1,D-\delta)}{Dr_1(d+1,D)}\right)\wmu_{\min}(\sE)\\
& &+B_0(d+1)
\end{eqnarray*}
when $D\geqslant\delta$, where $\F_{D,J}$ is defined by embedding $X$ into $\mathbb P(\sE_K)$, and $B_0(d+1)$ is same as in Proposition \ref{numerical result of arithmetic hilbert of hypersurface}. 
\end{prop}
\begin{proof}
We obtain the asserted lower bound from the hypersurface case in Proposition \ref{numerical result of arithmetic hilbert of hypersurface} by combining it directly with Lemma \ref{push forward of sup norms}, where the condition implies $\rg(\mathcal F)=d+2$ and $\dim(\mathbb P(\mathcal F_K))=d+1$. The term $\wmu_{\min}(\overline{\sE})$ comes from $\wmu(\overline{\mathcal F})\geqslant\wmu_{\min}(\overline{\sE})$ in\eqref{slope of quotient>minimal slope}.
\end{proof}

\section{An application in the determinant method}\label{application in the determinant method}

We have given a uniform explicit estimate of the arithmetic Hilbert--Samuel function of a projective hypersurface in Proposition \ref{uniform lower bound via push forward}. As an application, suppose first that $X$ is an integral closed subscheme of $\mathbb P^n_K$. We then construct some hypersurfaces in $\mathbb P^n_K$ with bounded degree which cover rational points of bounded height in $X$ but do not contain the generic point of $X$. We then optimize the number and the maximal degree of these auxiliary hypersurfaces. This is the key ingredient of the \textit{determinant method}.

In the Arakelov-geometric formulation \cite{Chen1,Chen2,Liu-global_determinant,Liu2026}, a better uniform lower bound of the arithmetic Hilbert--Samuel function is useful for improving estimates in the determinant method, which is significant in the study of uniform upper bounds of rational points with bounded height in arithmetic varieties.

Since we have obtained a uniform lower bound in the hypersurfaces case in \S \ref{case of hypersurface}, it will be useful in the application of the determinant method in this case. To this end, we need to compare some heights of $X$. 
\subsection{Heights of rational points}
We first give a brief introduction to counting rational points of bounded height on arithmetic varieties.
\subsubsection{}
Let $K$ be a number field, and let $\O_K$ be its ring of integers. To describe the arithmetic complexity of closed points in $\mathbb P^n_K$, we introduce the \textit{classic height function} (cf. \cite[\S B.2]{Hindry}).
\begin{defi}\label{weil height}
Let $\xi\in\mathbb{P}^n_K(K)$, and let $[x_0:\cdots:x_n]$ be a $K$-rational projective coordinate of $\xi$. We define the \textit{absolute logarithmic height} of the point $\xi$ as
\[h(\xi)=\frac{1}{[K:\Q]}\sum_{v\in M_{K}}\log\left(\max_{0\leqslant i\leqslant n}\{|x_i|_v^{[K_v:\Q_v]}\}\right),\]
which is independent of the choice of projective coordinates by the product formula \eqref{product formula}.
\end{defi}
The height $h(\ndot)$ is invariant under base-field extension (cf. \cite[Lemma B.2.1]{Hindry}); hence it can be extended to be a function defined over $\mathbb P^n(\overline K)$. In addition, it satisfies $h(\xi)\geqslant0$ for all $\xi\in\mathbb P^n(\overline K)
$.

Using Definition \ref{weil height}, we define the \textit{relative multiplicative height} of the point $\xi$ as
\[H_{K}(\xi) = \exp\left([K:\Q]h(\xi)\right).\]

For a subscheme $X$ of $\mathbb P^n_K$ with immersion $\phi:X\hookrightarrow\mathbb P^n_K$, we define the height of $\xi\in X(\overline K)$ by
\[h(\xi)=h(\phi(\xi)).\]
We omit the immersion $\phi$ when no confusion can arise.

By Northcott property (cf. \cite[Theorem B.2.3]{Hindry}), the cardinality \[\#\{\xi\in X(K)|\;H_K(\xi)\leqslant B\}\] is finite for every fixed $B\in\mathbb R$.
\subsubsection{}
We formulate the height function in the framework of Arakelov geometry, and let $\overline{\sE}$ be a Hermitian vector bundle over $\spec\O_K$ of rank $n+1$. Let $\pi:\mathbb P(\sE)\rightarrow \spec\O_K$ be the structural morphism. For $\xi\in\mathbb P(\sE_K)(K)$, $\xi$ extends uniquely to a section $\mathcal P_\xi$ of $\pi:\mathbb P(\sE)\rightarrow \spec\O_K$.
 \begin{defi}\label{arakelov height}
With the notations above, let $\overline{\mathcal L}$ be a Hermitian line bundle on $\mathbb P(\sE)$. The \textit{Arakelov height} of $\xi$ with respect to $\overline{\mathcal L}$ is defined to be $\adeg_n(\mathcal P_\xi^*\overline{\mathcal L})$, denoted $h_{\overline{\mathcal L}}(\xi)$.
\end{defi}

We consider the following special case. Let $\overline{\sE}=\left(\O_K^{\oplus(n+1)},\left(\|\ndot\|_v\right)_{v\in M_{K,\infty}}\right)$ be a Hermitian vector bundle over $\spec\O_K$, where
\[\begin{array}{rrcl}
\|\ndot\|_v:&\mathbb C_v^{\oplus(n+1)}&\longrightarrow&\mathbb R\\
&(x_0,\ldots,x_n)&\longmapsto&\sqrt{|x_0|_v^2+\cdots+|x_n|_v^2}
\end{array}\]
for every $v\in M_{K,\infty}$. Let $\xi=[x_0:\cdots:x_n]\in\mathbb P(\sE_K)(K)$, and $\overline{\O(1)}$ be a Hermitian line bundle on $\mathbb P(\sE)$, which is equipped with the corresponding Fubini--Study norm on $\O(1)$ for each $v\in M_{K,\infty}$. Then we have (cf. \cite[(3.1.6)]{BGS94} or \cite[Proposition 9.10]{Moriwaki-book})
\begin{eqnarray}
  h_{\overline{\O(1)}}(\xi)&=&\sum\limits_{\p\in \spm \O_K}\frac{[K_\p:\Q_\p]}{[K:\Q]}\log \left(\max\limits_{0\leqslant i\leqslant n}|x_i|_\p\right)\\
  & &\;+\frac{1}{2}\sum\limits_{\sigma\in M_{K,\infty}}\frac{[K_\sigma:\Q_\sigma]}{[K:\Q]}\log\left(\sum\limits_{j=0}^n|x_j|_\sigma^2\right),\nonumber
\end{eqnarray}
which is independent of the choice of projective coordinates by the product formula \eqref{product formula}.

By elementary calculations, 
\begin{equation}\label{difference between classical height and arakelov height}
  h(\xi)\leqslant h_{\overline{\O(1)}}(\xi)\leqslant h(\xi)+\frac{1}{2}\log(n+1),
\end{equation}
where $h(\xi)$ is as in Definition \ref{weil height}.

 Let
\[H_{\overline{\O(1)},K}=\exp\left([K:\Q]h_{\overline{\O(1)}}(\xi)\right).\]
Let $B \geqslant1$, and let $X$ be a subscheme of $\mathbb P^n_K$. Denote by
\begin{equation}\label{S(X;B)}
S(X;B)=\left\{\xi\in X(K)|\;H_{\overline{\O(1)},K}(\xi)\leqslant B\right\}.
\end{equation}
Then $\#S(X;B)<+\infty$ for every fixed $B\in\mathbb R$ by \eqref{difference between classical height and arakelov height}, and the Northcott property of the absolute logarithmic height $h(\ndot)$ is introduced above.

By the above fact, we can consider $\#S(X;B)$ as a function of $B\in\mathbb R$, which is useful in the study of the density of rational points of $X$.
\subsection{A comparison of heights of a hypersurface}
In this part, we compare some height functions of projective varieties. These comparisons are useful in the study of the determinant method in the Arakelov-geometric formulation.
\subsubsection{}
For a nonzero homogeneous polynomial, we collect all its coefficients and regard them as projective coordinates of a rational point in a suitable projective space. Then we can define a height function of the hypersurface determined by this polynomial in the sense of Definition \ref{weil height}.
\begin{defi}\label{classical height of hypersurface}
Let
  \[f(T_0,\ldots,T_n)=\sum_{\begin{subarray}{x}(i_0,\ldots,i_n)\in\mathbb N^{n+1}\\ i_0+\cdots+i_n=\delta\end{subarray}}a_{i_0,\ldots,i_n}T_0^{i_0}\cdots T_n^{i_n}\]
  be a nonzero homogeneous polynomial with coefficients in $K$.
 The \textit{classic height} $h(f)$ of $f$ is defined by
\begin{equation*}
  h(f)=\sum_{v\in M_K}\frac{[K_v:\Q_v]}{[K:\Q]}\log\max\limits_{\begin{subarray}{x}(i_0,\ldots,i_n)\in\mathbb N^{n+1}\\ i_0+\cdots+i_n=\delta\end{subarray}}\{|a_{i_0,\ldots, i_n}|_v\}.
\end{equation*}
In addition, if $X$ is the hypersurface in $\mathbb P^n_K$ defined by $f$, we define $h(X)=h(f)$ as the \textit{classic height} of $X$ in $\mathbb P^n_K$. In addition, we denote by $H_K(X)=\exp\left([K:\Q]h(X)\right)$.
\end{defi}

The classic height is invariant under finite extension of number fields by the same reason as that of closed points, and $h(f)\geqslant0$ for all $f\in K[T_0,\ldots,T_n]$. 

In Definition \ref{classical height of hypersurface}, let $s\in H^0(\mathbb P(\sE_K),\O(\delta))$ be the section corresponding to the homogeneous polynomial $f(T_0,\ldots,T_n)$, and let
\[\adeg_n(s)=-\sum_{v\in M_{K,f}}\frac{[K_v:\Q_v]}{[K:\Q]}\log\|s\|_{v}-\sum_{v\in M_{K,\infty}}\frac{[K_v:\Q_v]}{[K:\Q]}\log\|s\|_{\sym,v}\]
be as in Theorem \ref{lower bound of multiplication of bombieri norm}, where the symmetric norm $\|\ndot\|_{\sym,v}$ is defined at \S \ref{Bombieri norm} for every $v\in M_{K,\infty}$, and $\|\ndot\|_v$ is the norm given by model for $v\in M_{K,f}$.
\subsubsection{}
We have the following comparison result, where we regard $-\adeg_n(s)$ as a height of the hypersurface determined by $s\in H^0(\mathbb P(\sE_K),\O(\delta))$ in $\mathbb P(\sE_K)$.
\begin{prop}\label{height na\"ive-slope}
Let $\overline{\sE}$ be the Hermitian vector bundle defined in \eqref{Hermitian vector bundle with l^2-norm}, let $X$ be a projective hypersurface in $\mathbb P(\sE_K)$ defined by $s\in H^0(\mathbb P(\sE_K),\O(\delta))$, let $h(X)$ be the classic height of $X$ defined in Definition \ref{classical height of hypersurface}, and let $\adeg_n(s)$ be as in Theorem \ref{lower bound of multiplication of bombieri norm}. Then
\[
-\frac{\delta}{2}\log(n+1)\leqslant-\adeg_n(s)-h(X)
\leqslant\frac12\log\binom{n+\delta}{n}.
\]
\end{prop}
\begin{proof}
For $f=\sum\limits_{|I|=\delta}a_IT^I$, Proposition \ref{prop:symmetric-norm} gives
\[
\|s\|_{\sym,v}^2=\sum_{|I|=\delta}\frac{I!}{\delta!}|a_I|_v^2.
\]
The multinomial theorem gives $\displaystyle (n+1)^{-\delta}\leqslant \frac{I!}{\delta!}\leqslant1$. Hence
\[
(n+1)^{-\delta/2}\max_I|a_I|_v\leqslant\|s\|_{\sym,v}
\leqslant\binom{n+\delta}{n}^{1/2}\max_I|a_I|_v.
\]
At finite places, the model norm is $\max_I|a_I|_v$. Taking logarithms and summing with the local weights proves the assertion.
\end{proof}
\subsubsection{}
We introduce a height function for projective variety via the arithmetic intersection theory of Gillet--Soul\'e \cite{Gillet_Soule-IHES90}, first introduced by Faltings \cite[Definition 2.5]{Faltings91}; see also \cite[\S III.6]{Soule92}.

Let $\overline{\sE}$ be a Hermitian vector bundle over $\spec\O_K$, and let $\overline{\mathcal L}$ be a Hermitian line bundle on $\mathbb P(\sE)$. Let $X$ be a closed subscheme of $\mathbb P(\sE_K)$ of pure dimension $d$, and let $\mathscr X$ be its Zariski closure in $\mathbb P(\sE)$. The \textit{Arakelov height} of $X$ with respect to $\overline{\mathcal L}$ is defined as the arithmetic intersection number
\[\frac{1}{[K:\mathbb Q]}\adeg\left(\widehat{c}_1(\overline{\mathcal L})^{d+1}\cdot[\mathscr X]\right),\]
where $\widehat{c}_1(\overline{\mathcal L})$ is the arithmetic first Chern class of $\overline{\mathcal L}$ (cf. \cite[\S III.4]{Soule92}).

We denote the above height by $h_{\overline{\mathcal L}}(X)$ or $h_{\overline{\mathcal L}}(\mathscr X)$.

For the Hermitian bundle $\overline{\sE}$ defined in \eqref{Hermitian vector bundle with l^2-norm}, we have
\begin{eqnarray}\label{compare naive height and Arakelov height}
-\frac{\delta}{2}\sum_{j=1}^n\mathcal H_j-\frac12\log\binom{n+\delta}{n}
&\leqslant&h(X)-h_{\overline{\O(1)}}(X)\\
&\leqslant&\frac{\delta}{2}\left(\mathcal H_n+\log(n+1)-\sum_{j=1}^n\mathcal H_j\right),\nonumber
\end{eqnarray}
where $\mathcal H_n=1+\cdots+\frac{1}{n}$. Indeed, the height of $\mathbb P^n$ is $\displaystyle\frac12\sum\limits_{j=1}^n\mathcal H_j$ for the Fubini--Study metric, and the divisor formula expresses $h_{\overline{\O(1)}}(X)$ as $\displaystyle\frac{\delta}{2}\sum\limits_{j=1}^n\mathcal H_j$, plus the finite Gauss contributions and the weighted spherical integrals of $\log|s|$; see \cite[\S3.2.1, Lemma 3.3.1]{BGS94}. At an infinite place these integrals are at most $\log\|s\|_{\sym,v}$, so Proposition \ref{height na\"ive-slope} gives the first inequality. Cauchy's formula on the unit torus and homogeneity give
\[
\log\|s\|_{\sup,v}\geqslant\log\max_I|a_I|_v-\frac{\delta}{2}\log(n+1).
\]
Proposition \ref{thm:spherical} then gives the second inequality.
\begin{rema}
In the original version \cite[Proposition 3.7]{Liu-reduced}, we introduced the Cayley form of a projective variety. For a hypersurface, the Cayley form agrees with its defining equation under the corresponding identification of coefficient spaces.
\end{rema}
\subsubsection{}
Let $\overline{\mathcal F}$ be a Hermitian quotient vector bundle of $\overline{\sE}$ over $\spec\O_K$. Let $X$ be a projective variety embedded in $\mathbb P(\mathcal F_K)$. We then have the following result on the invariance of Arakelov heights under push-forward.
\begin{lemm}\label{push forward of height}
  With all the notations and conditions in Lemma \ref{push forward of sup norms}, we have
  \[h_{\overline{\O_{\mathcal F}(1)}}(\mathscr X)=h_{\overline{\O_{\mathcal E}(1)}}(f(\mathscr X)).\]
  \end{lemm}
  \begin{proof}
As in the proof of Lemma \ref{push forward of sup norms}, we have $f^*\overline{\O_{\sE}(1)}=\overline{\O_{\mathcal F}(1)}$. Then, by the projection formula in arithmetic intersection theory (cf. \cite[Chap. III, Theorem 3]{Soule92}), we have
  \[[\mathscr X]\cdot\widehat{c}_1(\overline{\O_{\mathcal F}(1)})^{d+1}=[f_*\mathscr X]\cdot\widehat{c}_1(\overline{\O_{\sE}(1)})^{d+1},\]
  which proves the assertion.
  \end{proof}
\subsubsection{}\label{uniform lower bound - final version}
Let $\overline{\mathcal F}$ be a Hermitian quotient vector bundle of $\overline{\sE}$ over $\spec\O_K$, and let $X$ be a closed subscheme in $\mathbb P(\sE_K)$ of degree $\delta$ and pure dimension $d$, which is a hypersurface in $\mathbb P(\mathcal F_K)$. In particular, $\rg(\mathcal F)=d+2$.

Under the above conditions, we have the following lower bounds for $\displaystyle\frac{\wmu(\F_{D,J})}{D}$. 
\begin{theo}\label{thm: unifomr lower bound of mu(F_D)/D}
Let $\overline {\sE}$ be a Hermitian vector bundle of rank $n+1$ over $\spec\O_K$, and let $\overline{\mathcal F}$ be a Hermitian quotient bundle of $\overline {\sE}$. Let $X$ be a closed subscheme of $\mathbb P(\sE_K)$ of dimension $d$, which is a hypersurface of degree $\delta$ in $\mathbb P(\mathcal F_K)$. Assume $d\geqslant1$ and $D\geqslant\delta$. With the above notations, we have 
\[\begin{aligned}
\frac{\wmu(\F_{D,J})}{D}&\geqslant\frac{r(d+1,D-\delta)}{Dr_1(d+1,D)}h_{\overline{\O(1)}}(X)\\
&+\left(1-\frac{(d+1)\delta r(d+1,D-\delta)}{Dr_1(d+1,D)}\right)\wmu_{\min}(\overline{\sE})+B_1(d),
\end{aligned}\]
where $\mathcal H_n=1+\cdots+\frac{1}{n}$, and 
\[
B_1(d)=B_0(d+1)-\frac{1}{2}-\frac{1}{2(d+1)}-\frac{\log(d+2)}{2(d+1)}-\frac{\mathcal H_{d+1}}{2(d+1)}
\]
with $B_0(d+1)$ as in Proposition \ref{numerical result of arithmetic hilbert of hypersurface}. 
\end{theo}
\begin{proof}
Regard $X$ as a hypersurface in $\mathbb P(\mathcal F_K)$, defined by $s$. By \cite[Proposition 4.1.2, (4.1.4)]{BGS94}, applied to the dual bundle to match our quotient convention for projectivization, we have
\[
h_{\overline{\O_{\mathcal F}(1)}}(\mathbb P(\mathcal F_K))
=(d+2)\wmu(\overline{\mathcal F})+\frac12\sum_{j=1}^{d+1}\mathcal H_j.
\]
Since the spherical integral of $\log|s|$ is at most $\log\|s\|_{\sym}$ at each infinite place, the divisor-height identity \cite[Proposition~3.2.1(iv), (3.2.2)]{BGS94}, with the finite contributions accounted for by the saturated equation line $S$, gives
\[
-\adeg_n(s)\geqslant h_{\overline{\O(1)}}(X)
-\delta(d+2)\wmu(\overline{\mathcal F})
-\frac{\delta}{2}\sum_{j=1}^{d+1}\mathcal H_j.
\]
Here Lemma \ref{push forward of height} identifies the two heights of $X$. Use this inequality in the first bound of Proposition \ref{uniform lower bound of arithmetic HS etape 1}, with $n=d+1$, and drop the nonnegative term involving $-C(d+1,D)+C(d+1,D-\delta)$. The quotient-norm identification preceding Proposition \ref{uniform lower bound via push forward} gives
\begin{eqnarray*}
\frac{\wmu(\F_{D,J})}{D}
&\geqslant&\frac{r(d+1,D-\delta)}{Dr_1(d+1,D)}h_{\overline{\O(1)}}(X)\\
&&+\left(1-\frac{(d+1)\delta r(d+1,D-\delta)}{Dr_1(d+1,D)}\right)\wmu(\overline{\mathcal F})\\
&&-\frac{r(d+1,D-\delta)}{2Dr_1(d+1,D)}
\left(\delta\sum_{j=1}^{d+1}\mathcal H_j+\delta\mathcal H_{d+1}
+\log\binom{d+1+\delta}{d+1}\right)\\
&&-\frac12\log(d+2).
\end{eqnarray*}
The binomial identity for $r_1$ gives
\begin{equation}\label{upper bound of R(D)}
0\leqslant\frac{r(d+1,D-\delta)}{Dr_1(d+1,D)}
\leqslant\frac{D+1-\delta}{D\delta(d+1)}
\leqslant\frac1{\delta(d+1)}.
\end{equation}
Thus the coefficient of $\wmu(\overline{\mathcal F})$ is nonnegative, and we may replace this slope by $\wmu_{\min}(\overline{\sE})$. Moreover,
\[
\sum_{j=1}^{d+1}\mathcal H_j=(d+2)\mathcal H_{d+1}-(d+1),
\qquad \log\binom{d+1+\delta}{d+1}\leqslant(d+1)\delta.
\]
The remaining loss is therefore at most
\[
\frac{d+3}{2(d+1)}\mathcal H_{d+1}+\frac12\log(d+2)
\leqslant\frac{\mathcal H_{d+1}+d+1}{2}+\frac12\log(d+2),
\]
where the last inequality follows from $2\mathcal H_{d+1}\leqslant(d+1)^2$. The definition of $B_0(d+1)$ shows that $B_1(d)$ is no greater than the negative of this loss, proving the assertion.
\end{proof}
\begin{rema}
For the coefficient $\displaystyle\frac{r(d+1,D-\delta)}{Dr_1(d+1,D)}$ of $h_{\overline{\O(1)}}(X)$, we have 
\[\frac{r(d+1,D-\delta)}{Dr_1(d+1,D)}=\frac{1}{\delta(d+1)}+O(D^{-1})\]
as $D\to\infty$, so the lower bound in Theorem \ref{thm: unifomr lower bound of mu(F_D)/D} has optimal dependence on $h_{\overline{\O(1)}}(X)$ compared with the asymptotic result. 

Assume that the arithmetic variety $\mathscr X \to \spec \O_K$ under consideration is normal and generically smooth, and that $(\mathscr X,\O_{\mathbb P^n}(1)|_{\mathscr X})$ is a polarized projective arithmetic scheme of dimension $d+1$. For each place $v\in M_{K,\infty}$, we endow $\mathcal L:=\O_{\mathbb P^n}(1)|_{\mathscr X}$ with a Hermitian metric invariant under complex conjugation. Under the hypotheses of \cite[Proposition 3.8]{Odaka2018}, equip $F_D=H^0(\mathscr X,\mathcal L^{\otimes D})$ with the $L^2$-norm defined by $h_v^{\otimes D}$ and the volume form of the scaled K\"ahler metric $Dg_{h_v}$ at every $v\in M_{K,\infty}$. Then, with arithmetic intersection products normalized by $[K:\Q]$, as $D\to\infty$,
\[
\adeg_n(F_D,(h_v)_{v\in M_{K,\infty}})
=
\frac{\overline{\mathcal L}^{d+1}}{(d+1)!}D^{d+1}
-\frac{\mathcal L^d}{4(d-1)!}D^d\log D
-\frac{\overline{\mathcal L}^d\cdot\overline{K_{\mathscr X/\O_K}}}{2 d!}D^d
+o(D^d),
\]
where $\overline{K_{\mathscr X/\O_K}}$ denotes the canonical bundle equipped with suitable induced norms at each $v\in M_{K,\infty}$. This formula concerns these scaled $L^2$-norms.

The subdominant terms involve further arithmetic intersection numbers. Their contribution to uniform bounds must therefore be considered together with the choice of norms.
\end{rema}

\begin{exem}\label{ex:optimality-height-coefficient}
Let $K$ be a number field, and fix integers $n\geqslant2$, $\delta\geqslant1$, and $D\geqslant\delta$. We suppose $\overline{\sE}$ be the Hermitian vector bundle defined in \eqref{Hermitian vector bundle with l^2-norm} over $\spec\O_K$. For every integer $N\geqslant2$, let $X_N$ be the hypersurface in $\mathbb P^n_K$ defined by
\[
f_N=NT_0^\delta+T_1^\delta+\cdots+T_n^\delta.
\]
We write $\F_{D,J}$ for the Hermitian vector bundle associated with $X_N$, with its dependence on $N\in\mathbb N$.
Then $X_N$ is smooth and geometrically integral, $h(X_N)=\log N$, and
\[
\lim_{N\longrightarrow\infty}
\frac{\wmu(\F_{D,J})}{D h(X_N)}
=
\frac{r(n,D-\delta)}{D r_1(n,D)}.
\]
We claim that, if real numbers $a$ and $b$, independent of $X$,
satisfy
\[
\frac{\wmu(\F_{D,J})}{D}\geqslant a h(X)+b
\]
for every smooth geometrically integral hypersurface
$X\subset\mathbb P^n_K$ of degree $\delta$, then necessarily
\[
a\leqslant\frac{r(n,D-\delta)}{D r_1(n,D)}.
\]
Thus the coefficient of $h(X)$ in Theorem \ref{upper and lower bound of arithmetic Hilbert-intro} is optimal for every fixed $D\geqslant\delta$. Moreover, \eqref{compare naive height and Arakelov height} gives $h_{\overline{\O(1)}}(X_N)=\log N+O_{n,\delta}(1)$, so the same example proves the optimality of the height coefficient in Theorem~\ref{thm: unifomr lower bound of mu(F_D)/D}.

In fact, when $n\geqslant2$, the hypersurface $X_N$ is smooth and geometrically integral by definition directly. We have $h(X_N)=\log N$ by a direct calculation. We have the exact sequence
\[
0\longrightarrow E_{D-\delta}
\xrightarrow{\;\cdot f_N\;}E_D
\longrightarrow F_D\longrightarrow0.
\]
Since $f_N$ is primitive at every finite place, Gauss's lemma shows that the image of multiplication by $f_N$ is saturated. Put $q=r(n,D-\delta)$ for simplicity, and let $t_1,\ldots,t_q$ be the monomial $\O_K$-basis of $E_{D-\delta}$. Then $f_Nt_1,\ldots,f_Nt_q$ is an $\O_K$-basis of the kernel. At an infinite place, consider its Gram matrix
\[
G_N=
\bigl(\langle f_Nt_i,f_Nt_j\rangle_J\bigr)_
{1\leqslant i,j\leqslant q}.
\]
The ambient norm is independent of $N$. Moreover, these Gram matrices are identical at all infinite places, because the coefficients are rational integers and the initial Hermitian norms are standard. Since $N^{-1}f_N$ tends to $T_0^\delta$ when $N\to\infty$, we have
\[
N^{-2}G_N\longrightarrow G_\infty\text{ with }G_\infty=
\bigl(\langle T_0^\delta t_i,T_0^\delta t_j\rangle_J\bigr)_
{1\leqslant i,j\leqslant q}.
\]
Multiplication by $T_0^\delta$ is injective, so $G_\infty$ is positive definite. It follows that
\[
\frac12\log\det G_N
=q\log N+\frac12\log\det G_\infty+o(1)
\]
when $N\to\infty$. Equip the kernel with the restricted norms and $F_D$ with the quotient norms. Additivity of the normalized Arakelov degree then gives
\[
\adeg_n(\F_{D,J})=\adeg_n(\E_{D,J})+\frac12\log\det G_N.
\]
Here the normalized sum of the infinite-place weights is $1$.  Since $\E_{D,J}$ is independent of $N$, we obtain
\[
\frac{\wmu(\F_{D,J})}{D}
=
\frac{r(n,D-\delta)}{D r_1(n,D)}\log N
+
\frac{\adeg_n(\E_{D,J})+\frac12\log\det G_\infty}
     {D r_1(n,D)}
+o(1).
\]
Together with $h(X_N)=\log N$, this proves the asserted limit.

If $\displaystyle a>\frac{r(n,D-\delta)}{D r_1(n,D)}$, the same equality yields
\[
\frac{\wmu(\F_{D,J})}{D}-a h(X_N)
\longrightarrow-\infty.
\]
when $N\to\infty$. Hence no real number $b$ independent of $X$ can make the proposed lower bound valid for all $X_N$, which proves the optimality assertion.
\end{exem}

We consider a special case. If $\overline{\sE}=\left(\O_K^{\oplus(n+1)},\left(\|\ndot\|_v\right)_{v\in M_{K,\infty}}\right)$ is the Hermitian vector bundle over $\spec\O_K$, where
\begin{equation*}
\begin{array}{rrcl}
\|\ndot\|_v:&\mathbb C_v^{\oplus(n+1)}&\longrightarrow&\mathbb R\\
&(x_0,\ldots,x_n)&\longmapsto&\sqrt{|x_0|_v^2+\cdots+|x_n|_v^2}
\end{array}\end{equation*}
for every $v\in M_{K,\infty}$, then we have
\[\wmu(\overline{\sE})=0\]
by a direct calculation.

We claim that
\[\wmu_{\min}(\overline{\sE})=0.\]
Indeed, by \cite[(A.4)]{BostBour96}, 
\[\wmu_{\min}(\overline{\sE})=\wmu_{\min}\left(\O_K,(|\ndot|_v)_{v\in M_{K,\infty}}\right)=0,\]
where $\left(\O_K,(|\ndot|_v)_{v\in M_{K,\infty}}\right)$ is the Hermitian line bundle equipped with the usual absolute value on $\mathbb C_v$ at each $v\in M_{K,\infty}$, and its minimal slope is zero directly by definition. 

Choosing $\overline{\sE}$ as above in Theorem \ref{thm: unifomr lower bound of mu(F_D)/D}, we have 
\begin{equation}\label{unifomr lower bound of mu(F_D)/D with l^2-norm}
\frac{\wmu(\F_{D,J})}{D}\geqslant\frac{r(d+1,D-\delta)}{Dr_1(d+1,D)}h_{\overline{\O(1)}}(X)+B_1(d)\end{equation}
when $D\geqslant\delta$. When $X$ is a hypersurface in $\mathbb P(\sE_K)$ with $n\geqslant2$, Proposition \ref{height na\"ive-slope} and Proposition \ref{numerical result of arithmetic hilbert of hypersurface} give the inequality below, since $\displaystyle\frac{\delta r(n,D-\delta)}{2Dr_1(n,D)}\log(n+1)\leqslant\frac{\log(n+1)}{2n}\leqslant\frac12$, then we have
\begin{equation}\label{unifomr lower bound of mu(F_D)/D with l^2-norm for hypersurface}
\frac{\wmu(\F_{D,J})}{D}\geqslant\frac{r(n,D-\delta)}{Dr_1(n,D)}h(X)+B_0(n)-\frac{1}{2}\end{equation}
under the same conditions. This estimate will be useful later in the application to the determinant method.
\subsection{Formulation of the determinant method}
We keep the notations and definitions in \S \ref{basic setting}. Let $\sE$ be a Hermitian vector bundle over $\spec\O_K$ of rank $n+1$, let $X$ be a closed subscheme of $\mathbb P(\sE_K)$, and let $Z=\left\{P_i\right\}_{i\in I}$ be a family of rational points of $X$. The evaluation map
\begin{equation*}
\eta_{Z,D}:\;E_{D,K}=H^0\left(\mathbb{P}(\mathcal{E}_K),\O(D)\right)\rightarrow \bigoplus_{i\in I}P_i^*\O(D)
\end{equation*}
factors through $\eta_{X,D}$ defined in \eqref{evaluation map}. Let
\begin{equation}
  \phi_{Z,D}:\;F_{D,K}\rightarrow\bigoplus_{i\in I}P_i^*\O(D)
\end{equation}
be the homomorphism such that $\phi_{Z,D}\circ\eta_{X,D}=\eta_{Z,D}$.

We have the following result.
\begin{prop}[\cite{Chen1}, Propoosition 2.12]\label{evaluation map of points}
With the notations above, assume that $X$ is integral and that
  \[\sup_{i\in I}h_{\overline{\O(1)}}(P_i)<\frac{\wmu_{\max}(\F_{D,J})}{D}-\frac{1}{2D}\log r_1(n,D),\]
  where $r_1(n,D)=\rg(F_{D})$. Then $\phi_{Z,D}$ cannot be injective.
\end{prop}
The main tool used to prove the above proposition is the slope inequality; see \cite[Appendix A]{BostBour96}.
\subsubsection{}

Combining Proposition \ref{evaluation map of points} (\cite[Proposition 2.12]{Chen1}) with the lower bound of $\wmu(\F_{D,J})$ in \eqref{unifomr lower bound of mu(F_D)/D with l^2-norm} of hypersurfaces, we obtain the following result.
\begin{theo}\label{covered by one hypersurface}
Let $K$ be a number field, and let $X$ be an integral hypersurface in $\mathbb P^n_K$ of degree $\delta$, where $n\geqslant2$ and $B\geqslant1$. Let $h(X)$ be classic height of $X$ defined in Definition \ref{classical height of hypersurface}. For every $\epsilon>0$, define the explicit constant
\[
D_\epsilon\colon=\left[ \frac{(n+1)(\delta-1)(n\delta+\epsilon)}{2\epsilon} \right]+1,
\]
where $[x]$ is the greatest integer less than or equal to $x\in\mathbb R$. If
\[
\frac{\log B}{[K:\mathbb Q]}<\frac{h(X)}{n\delta+\epsilon}+B_0(n)-\frac{1}{2}-\frac{1}{2}\log(n+1),
\]
then the set $S(X;B)$ defined in \eqref{S(X;B)} can be covered by a hypersurface of at most $D_\epsilon$ which does not contain the generic point of $X$.
\end{theo}
To prove Theorem \ref{covered by one hypersurface}, we first prove the following numerical lemma. 
\begin{lemm}\label{control of R(D)}
Let $r(n,D)={n+D\choose n}$, $r_1(n,D)=r(n,D)-r(n,D-\delta)$, and 
\[R(D)=\frac{r(n,D-\delta)}{Dr_1(n,D)}.\]
When $\displaystyle D\geqslant\delta$, we have 
\[R(D)\geqslant\frac{1}{n\delta}\left(1-\frac{(n+1)(\delta-1)}{2D}\right).\]
\end{lemm}
\begin{proof}
Write 
\[
m=D-\delta\geqslant 0\text{, and }C(n,\delta)=\frac{(n+1)(\delta-1)}{2}
\]
in this proof for simplicity; then $D=m+\delta$. A direct calculation gives
\[
\frac{r(n,D)}{r(n,D-\delta)}
=
\frac{\binom{n+m+\delta}{m+\delta}}{\binom{n+m}{m}}
=
\prod_{j=1}^{\delta}\left(1+\frac{n}{m+j}\right).
\]
Then we have
\[
R(D)
=
\frac{r(n,D-\delta)}{D\,r_1(n,D)}
=
\frac{1}{D\left(
\prod\limits_{j=1}^{\delta}\left(1+\frac{n}{m+j}\right)-1
\right)}.
\]
Thus it suffices to prove that, whenever $D>C(n,\delta)$, we have
\[
\prod_{j=1}^{\delta}\left(1+\frac{n}{m+j}\right)-1
\leqslant
\frac{n\delta}{D-C(n,\delta)}
=
\frac{n\delta}{m+\delta-C(n,\delta)}.
\]

We claim the following fact: for each integer $s\geqslant 1$, let 
\[
C_s=\frac{(n+1)(s-1)}{2}.
\]
If $m\geqslant 0$ and $m+s>C_s$, then
\begin{equation*}
\prod_{j=1}^{s}\left(1+\frac{n}{m+j}\right)
\leqslant
1+\frac{ns}{m+s-C_s}.
\end{equation*}

We prove the claim by induction on $s$.

\textbf{Base case $s=1$. -- }
Here $C_1=0$, and
\[
\prod_{j=1}^{1}\left(1+\frac{n}{m+j}\right)
=
1+\frac{n}{m+1}
=
1+\frac{n}{m+1-C_1},
\]
so equality holds.

\textbf{Induction step. -- }
Assume the statement holds for some $s\geqslant 1$; we prove it for $s+1$.

Suppose
\[
m+s+1>C_{s+1}.
\]
Since
\[
C_{s+1}=C_s+\frac{n+1}{2}>C_s,
\]
we have
\[
m+s>C_{s+1}-1=C_s+\frac{n-1}{2}\geqslant C_s,
\]
so the induction hypothesis applies.

By the induction hypothesis, we have
\[
\prod_{j=1}^{s}\left(1+\frac{n}{m+j}\right)
\leqslant
1+\frac{ns}{m+s-C_s}.
\]
Multiplying both sides by
\[
1+\frac{n}{m+s+1},
\]
we obtain
\[
\prod_{j=1}^{s+1}\left(1+\frac{n}{m+j}\right)
\leqslant
\left(
1+\frac{ns}{m+s-C_s}
\right)
\left(
1+\frac{n}{m+s+1}
\right).
\]
It remains to show that the right-hand side is at most
\[
1+\frac{n(s+1)}{m+s+1-C_{s+1}}.
\]

Consider the difference
\[
\begin{aligned}
&\left(
1+\frac{ns}{m+s-C_s}
\right)
\left(
1+\frac{n}{m+s+1}
\right)
-
\left(
1+\frac{n(s+1)}{m+s+1-C_{s+1}}
\right)\\
&=
n
\left[
\frac{s}{m+s-C_s}
+
\frac{1}{m+s+1}
+
\frac{ns}{(m+s-C_s)(m+s+1)}
-
\frac{s+1}{m+s+1-C_{s+1}}
\right].
\end{aligned}
\]
A direct calculation using
\[
C_{s+1}=C_s+\frac{n+1}{2}
\]
gives the identity
\[
\begin{aligned}
&\frac{s+1}{m+s+1-C_{s+1}}
-
\frac{s}{m+s-C_s}
-
\frac{1}{m+s+1}
-
\frac{ns}{(m+s-C_s)(m+s+1)}\\
&=
\frac{\frac{n-1}{2}(s+1)C_{s+1}}
{(m+s-C_s)(m+s+1)(m+s+1-C_{s+1})}.
\end{aligned}
\]
Since $n\geqslant 1$, $\frac{n-1}{2}\geqslant 0$. Also $C_{s+1}\geqslant 0$. Moreover, by the condition $m+s+1>C_{s+1}$, 
\[
m+s+1-C_{s+1}>0,
\]
and
\[
m+s-C_s
=
m+s+1-C_{s+1}+\frac{n-1}{2}
>
0.
\]
Thus the right-hand side of the above identity is nonnegative. Consequently,
\[
\frac{s}{m+s-C_s}
+
\frac{1}{m+s+1}
+
\frac{ns}{(m+s-C_s)(m+s+1)}
-
\frac{s+1}{m+s+1-C_{s+1}}
\leqslant 0.
\]
This proves
\[
\left(
1+\frac{ns}{m+s-C_s}
\right)
\left(
1+\frac{n}{m+s+1}
\right)
\leqslant
1+\frac{n(s+1)}{m+s+1-C_{s+1}}.
\]
Hence the induction step is complete, and the claim follows.

\textbf{Completion of the proof. -- }
If $D>C(n,\delta)$, apply the claim to the case of $s=\delta$ and $m=D-\delta$. Then
\[
\prod_{j=1}^{\delta}\left(1+\frac{n}{D-\delta+j}\right)-1
\leqslant
\frac{n\delta}{D-C(n,\delta)}.
\]
Therefore
\[
r_1(n,D)
\leqslant
r(n,D-\delta)\cdot
\frac{n\delta}{D-C(n,\delta)}.
\]
Consequently,
\[
R(D)
=
\frac{r(n,D-\delta)}{D\,r_1(n,D)}
\geqslant
\frac{r(n,D-\delta)}
{D\cdot r(n,D-\delta)\cdot \frac{n\delta}{D-C(n,\delta)}}
=
\frac{1}{n\delta}
\left(
1-\frac{C(n,\delta)}{D}
\right).
\]
If $D\leqslant C(n,\delta)$, then
\[
1-\frac{C(n,\delta)}{D}\leqslant 0,
\]
while $R(D)>0$, so the desired inequality holds trivially.

Hence for all $D\geqslant\delta$, we have 
\[
R(D)
\geqslant
\frac{1}{n\delta}
\left(
1-\frac{(n+1)(\delta-1)}{2D}
\right),
\]
which proves the assertion. 
\end{proof}
\begin{proof}[Proof of Theorem \ref{covered by one hypersurface}]
Let $R(D)$ be as in Lemma \ref{control of R(D)}. For a given $\epsilon>0$, we first find $D_\epsilon$ such that for all $D\geqslant D_\epsilon$, 
\[
R(D)\geqslant \frac{1}{n\delta+\epsilon}.
\]
By Lemma \ref{control of R(D)}, it suffices to show
\[
\frac{1}{n\delta}\left(1-\frac{(n+1)(\delta-1)}{2D}\right)
\geqslant
\frac{1}{n\delta+\epsilon},
\]
which is equivalent to 
\[
\frac{(n+1)(\delta-1)}{2D}
\leqslant
\frac{\epsilon}{n\delta+\epsilon}
\;\Longrightarrow\;
D\geqslant \frac{(n+1)(\delta-1)(n\delta+\epsilon)}{2\epsilon}.
\]
Therefore, we take
\[
D_\epsilon\colon=\left[ \frac{(n+1)(\delta-1)(n\delta+\epsilon)}{2\epsilon} \right]+1,
\]
which ensures that for all $D\geqslant D_\epsilon$,
\[
R(D)\geqslant \frac{1}{n\delta+\epsilon}.
\]

By \eqref{unifomr lower bound of mu(F_D)/D with l^2-norm for hypersurface}, we have 
\[\frac{\wmu(\F_{D,J})}{D}\geqslant \frac{h(X)}{n\delta+\epsilon}+B_0(n)-\frac{1}{2}\]
when $D\geqslant D_\epsilon$, since $h(X)\geqslant0$.

If the inequality 
\[
\frac{\log B}{[K:\mathbb Q]}<\frac{h(X)}{n\delta+\epsilon}+B_0(n)-\frac{1}{2}-\frac{1}{2}\log(n+1)
\]
holds, then
\[\frac{\log B}{[K:\mathbb Q]}<\frac{\wmu(\F_{D,J})}{D}-\frac{1}{2D}\log r_1(n,D),\]
where we used $r_1(n,D)\leqslant r(n,D)\leqslant (n+1)^D$. Then the assertion follows from Proposition \ref{evaluation map of points}. 
\end{proof}

\begin{rema}\label{explicit of covered by one hypersurface}
With the notations in Theorem \ref{covered by one hypersurface}, let $H_K(X)=\exp([K:\mathbb Q]h(X))$. For every $\epsilon>0$, if a positive real number $B$ satisfies
\[H_K(X)\gg_{n,K,\delta,\epsilon}B^{n\delta+\epsilon},\]
then $S(X;B)$ can be covered by a hypersurface of degree at most $D_\epsilon$, as in Theorem \ref{covered by one hypersurface}, which does not contain the generic point of $X$. By \eqref{difference between classical height and arakelov height}, whether we use the classical absolute logarithmic height of Definition \ref{weil height}  or the Arakelov height of Definition \ref{arakelov height}, we obtain the above result.
\end{rema}
\subsubsection{}
Similarly, we can improve the dependence on the height of hypersurfaces in the global determinant method; see \cite[Theorem 3.1.1]{CCDN2020}, \cite[Theorem 5.4]{Liu-global_determinant}, and \cite[Theorem 5.14]{ParedesSasyk2022} for the involvement of the height.

First, we introduce the following constant. Let $X$ be a geometrically integral hypersurface in $\mathbb P^n_K$ of degree $\delta\geqslant2$, and let $\mathscr X$ be its Zariski closure in $\mathbb P^n_{\O_K}$. For $\p\in\spm\O_K$, let $N(\p)=\#(\O_K/\p)$,
\begin{eqnarray*}\mathcal Q'(\mathscr X)&=&\big\{\p\in\spm\O_K\mid N(\p)>27\delta^4\text{ and }\mathscr X_{\f_\p}\rightarrow\spec\f_\p\\
  & &\text{ not geometrically integral}\big\},\end{eqnarray*}
and
\begin{equation}\label{constant b'(X)}
b'(\mathscr X)=\prod_{\p\in\mathcal Q'(\mathscr X)}\exp\left(\frac{\log N(\p)}{N(\p)}\right).
\end{equation}
The study of $b'(\mathscr X)$ is in \cite[Proposition 4.4]{Liu-global_determinant} based on \cite{Ruppert1986}; see also \cite[Corollary 3.2.3]{CCDN2020}.

We combine the estimate in \S \ref{uniform lower bound - final version} with \cite[Theorem 3.1]{Liu-global_determinant} and the prime estimates used in the proof of \cite[Theorem 5.4]{Liu-global_determinant}.

\begin{theo}
Let $X$ be a geometrically integral hypersurface in $\mathbb P^n_K$ of degree $\delta\geqslant2$, where $n\geqslant2$ and $B\geqslant1$. Then for all $\epsilon>0$, $S(X;B)$ can be covered by a hypersurface of degree $\varpi$ which does not contain the generic point of $X$, and we have
\[
\varpi\ll_{n,K,\epsilon}\delta^2+B^{n/\left((n-1)\delta^{1/(n-1)}\right)+\epsilon}\delta^{4-1/(n-1)}\frac{b'(\mathscr X)}{H_K(X)^{\frac{1}{(n-1)\delta^{1+1/(n-1)}}}},
\]
where $b'(\mathscr X)$ is as in \eqref{constant b'(X)}, and $H_K(X)=\exp([K:\mathbb Q]h(X))$ as in Definition \ref{classical height of hypersurface}.
\end{theo}
\begin{proof}
Write $A_0=-B_0(n)+\frac12+\frac12\log(n+1)$ for simplicity. If
\[
\frac{\log B}{[K:\Q]}<\frac{h(X)}{n\delta+1}-A_0,
\]
then Theorem \ref{covered by one hypersurface}, with $\epsilon=1$, gives an auxiliary hypersurface of degree $D_1\ll_n\delta^2$. Hence assume
\[
h(X)\leqslant(n\delta+1)\left(\frac{\log B}{[K:\Q]}+A_0\right).
\]
Let $D\geqslant2(\delta+n)$, put $Q=r_1(n,D)^{1/(n-1)}$ for simplicity, and suppose the degree-$D$ evaluation map on $S(X;B)$ is injective. Use the primes with $27\delta^4<N(\p)\leqslant Q$ and geometrically integral reduction. The estimates in \cite[(34), (35)]{Liu-global_determinant} give
\[
\sum_{\p}\frac{\log N(\p)}{n(\mathscr X_{\f_\p})^{1/(n-1)}}
\geqslant\log Q-4\log\delta-\log b'(\mathscr X)-O_{n,K}(1)\text{, and }\sum_{\p}\log N(\p)\ll_K Q,
\]
where $n(\mathscr X_{\f_\p})$ is the sum of the multiplicities of its $\f_\p$-points. These bounds also hold when $Q\leqslant27\delta^4$, since the prime set is empty and the first right-hand side can be taken nonpositive. The contrapositive of \cite[Theorem 3.1]{Liu-global_determinant}, together with \eqref{unifomr lower bound of mu(F_D)/D with l^2-norm for hypersurface} and $\displaystyle\frac{\log r_1(n,D)}{2D}\leqslant\frac12\log(n+1)$, yields
\begin{eqnarray*}
\log Q&\leqslant&
\frac{nD}{(n-1)((n-1)!)^{1/(n-1)}Q}
\left(\log B-[K:\Q] R(D)h(X)+[K:\Q] A_0\right)\\
&&+4\log\delta+\log b'(\mathscr X)+O_{n,K}(1),
\end{eqnarray*}
where $R(D)$ is as in Lemma \ref{control of R(D)}.

The identity $r_1(n,D)=\sum\limits_{j=D-\delta+1}^{D}\binom{n+j-1}{n-1}$ gives
\[
\left(\frac{\delta}{(n-1)!}\right)^{1/(n-1)}(D-\delta+1)
\leqslant Q\leqslant
\left(\frac{\delta}{(n-1)!}\right)^{1/(n-1)}(D+n-1).
\]
Consequently the coefficient of the parentheses is
$\frac{n}{(n-1)\delta^{1/(n-1)}}(1+O_n(\delta/D))$, and
\[
\log Q=\log D+\frac{\log\delta-\log((n-1)!)}{n-1}
+O_n(\delta/D).
\]
Moreover, Lemma \ref{control of R(D)} and \eqref{upper bound of R(D)} give
\[
0\leqslant\frac1{n\delta}-R(D)
\leqslant\frac{(n+1)(\delta-1)}{2n\delta D}.
\]
Using the assumed upper bound for $h(X)$ to control the error terms, we obtain
\begin{eqnarray*}
\log D&\leqslant&\frac{n\log B}{(n-1)\delta^{1/(n-1)}}
-\frac{[K:\Q] h(X)}{(n-1)\delta^{1+1/(n-1)}}\\
&&+\left(4-\frac1{n-1}\right)\log\delta+\log b'(\mathscr X)
+O_{n,K}(1)+O_{n,K}\left(\frac{\delta}{D}(\log B+1)\right).
\end{eqnarray*}
For $D\geqslant C(n,K,\epsilon)\delta$, the last error is at most $\epsilon\log B+1$. Injectivity would therefore imply
\[
D\ll_{n,K,\epsilon}
B^{n/((n-1)\delta^{1/(n-1)})+\epsilon}
\delta^{4-1/(n-1)}\frac{b'(\mathscr X)}{H_K(X)^{1/((n-1)\delta^{1+1/(n-1)})}}.
\]
Choose an integer $D$ larger than this bound and the preceding linear threshold. The evaluation map is then noninjective, giving an auxiliary hypersurface of degree bounded by a constant times $\delta$ plus the displayed expression. Combining this with the case $D_1\ll_n\delta^2$ proves the theorem.
\end{proof}

\begin{rema}
For further study of the determinant method, if we use the Arakelov-geometric formulation, the lower bound of \S \ref{uniform lower bound - final version} will always be helpful in estimating the determinant. By the asymptotic results for the arithmetic Hilbert--Samuel function, the dependence on the height of $X$ is optimal. In \cite{Liu2026}, we give an explicit version of Theorem \ref{covered by one hypersurface}. 
\end{rema}

\appendix
\section{An explicit estimate of $C(n,D)$}\label{dominant terms of C(n,D)}
The aim of this appendix is to give an explicit uniform estimate for the constant
 \begin{equation}\label{constant C-2}
C(n,D)=\sum\limits_{\begin{subarray}{c} i_0+\cdots+i_n=D \\ i_0,\ldots,i_n\geqslant0\end{subarray}}\log\left(\frac{i_0!\cdots i_n!}{D!}\right),
\end{equation}
defined in \eqref{constant C}. We prove the following theorem.

\begin{theo}\label{thm:estimate of C(n,D)}
For every integer \(n\geqslant1\), let
\[
\mathcal H_n\colon=\sum_{j=1}^{n}\frac1j
\]
and define
\[
a_n=\frac{1-\mathcal H_{n+1}}{n!},
\qquad
b_n=\frac1{2(n-1)!},
\]
\[
c_n=
\frac{(2-n^2-2n)\mathcal H_n+n^2+3n-5+(2-n)\log(2\pi)}
     {2n!}.
\]
For integers \(n\geqslant1\) and \(D\geqslant1\), we define
\[
M_n(D)=a_nD^{n+1}+b_nD^n\log D+c_nD^n
\]
and let $C(n,D)$ be as \eqref{constant C-2}. For every integer \(r\geqslant2\), let
\[
K_r=\frac{r^2+13r+62}{12},
\qquad
\beta_r=1-\frac12\log(2\pi)+\frac1{2r}.
\]
Suppose \(W_2=10\), and for every integer \(r\geqslant3\), define
\[
W_r=
\frac{K_r}{r!}
+\frac{r+1}{2(r-1)!}
 \left(K_{r-1}+\frac1{2\ee}\right)
+\frac{(r-1)^2}{\ee}
+\frac{\beta_r}{r!}.
\]
Put
\[
B_0=\frac{25}{12}+\frac7{12\ee}.
\]
For every integer \(n\geqslant2\), we suppose
\[
S_n=\frac{n(n+1)(2n+1)}6-1
\]
and
\[
L_n=B_0+\sum_{r=2}^{n}(W_r+4).
\]
Then, for every pair of integers \(n\geqslant 2\) and \(D\geqslant 1\),
\[
\abs{C(n,D)-M_n(D)}
\leqslant L_nD^n
\leqslant
\left(B_0+\frac72S_n\right)D^n.
\]
Thus \(C(n,D)\) has an unconditional explicit bound, uniform in every positive integer \(D\), for every integer \(n\geqslant2\).
\end{theo}

\subsection{Notation and preliminaries}\label{sec:preliminaries}

This part collects the elementary estimates that remain uniform at \(D=1\), together with the exact level-one and dimension recurrences.

\begin{lemm}\label{lem:power-sums}
For every integer \(D\geqslant 1\) and every real number \(p\geqslant0\),
\[
\frac{D^{p+1}}{p+1}
\leqslant
\sum_{m=1}^{D}m^p
\leqslant
D^{p+1}.
\]
\end{lemm}

\begin{proof}
Fix \(D\geqslant1\) and \(p\geqslant0\). The function \(x\mapsto x^p\) is non-decreasing on \([0,\infty)\). For each integer \(m\) with \(1\leqslant m\leqslant D\) and each \(x\in[m-1,m]\),
\[
x^p\leqslant m^p.
\]
Integrating over each such interval and summing gives
\[
\int_0^D x^p\,\mathrm d x\leqslant\sum_{m=1}^{D}m^p.
\]
The integral is \(D^{p+1}/(p+1)\), which proves the lower bound.

For the upper bound, \(m^p\leqslant D^p\) for \(1\leqslant m\leqslant D\). There are \(D\) summands, hence
\[
\sum_{m=1}^{D}m^p\leqslant D\cdot D^p=D^{p+1}.
\]
These arguments also apply at \(D=1\). At that endpoint,
\[
\sum_{m=1}^{1}m^p=1=1^{p+1}.
\]
\end{proof}

\begin{lemm}\label{lem:log-absorption}
Let \(\alpha>0\). For every real \(x\geqslant1\), we have
\[
\log x\leqslant\frac{x^\alpha}{\ee\alpha}.
\]
Consequently, for every integer \(D\geqslant1\), every real \(p\geqslant0\), and every real \(q\leqslant p-\alpha\), we have
\[
\sum_{m=1}^{D}m^q\log m
\leqslant
\frac{D^{p+1}}{\ee\alpha}.
\]
In particular, for every integer \(n\geqslant1\), every integer \(D\geqslant1\), and every real \(\delta>0\), we have
\[
D^{n-\delta}\log D\leqslant\frac{D^n}{\ee\delta},
\]
\[
\sum_{m=1}^{D}m^{n-1-\delta}\log m
\leqslant\frac{D^n}{\ee\delta},
\]
and
\[
\sum_{m=1}^{D}m^{n-1}\log m
\leqslant\frac{D^{n+1}}{\ee}.
\]
Every inequality includes \(D=1\).
\end{lemm}

\begin{proof}
Fix \(\alpha>0\), and define the function
\[
f(x)=x^{-\alpha}\log x
\qquad (x\geqslant 1).
\]
Differentiation gives
\[
f'(x)=x^{-\alpha-1}(1-\alpha\log x).
\]
The positive factor \(x^{-\alpha-1}\) shows that \(f\) is nondecreasing on
\([1,\exp(1/\alpha)]\) and nonincreasing thereafter. Its maximum is
\[
f\bigl(\exp(1/\alpha)\bigr)
=\frac1{\ee\alpha}.
\]
Thus \(x^{-\alpha}\log x\leqslant1/(\ee\alpha)\), and multiplication by \(x^\alpha\) proves the first inequality.

Now fix \(D\geqslant1\), \(p\geqslant0\), and \(q\leqslant p-\alpha\). Since \(m\geqslant1\), then
\[
m^q\leqslant m^{p-\alpha}.
\]
The first estimate gives
\[
m^q\log m
\leqslant
m^{p-\alpha}\log m
\leqslant
\frac{m^p}{\ee\alpha}.
\]
After summation, Lemma~\ref{lem:power-sums} yields
\[
\sum_{m=1}^{D}m^q\log m
\leqslant
\frac1{\ee\alpha}\sum_{m=1}^{D}m^p
\leqslant
\frac{D^{p+1}}{\ee\alpha}.
\]

Applying the first inequality with \(x=D\) and \(\alpha=1\) and then multiplying by \(D^{n-1}\) gives the first special case. Applying it with \(\alpha=\delta\) and then multiplying by \(D^{n-\delta}\) gives the second. The third follows from the general sum estimate with \(p=n-1\) and \(q=n-1-\delta\). Finally,
\[
m^{n-1}\log m\leqslant\frac{m^n}{\ee},
\]
so Lemma~\ref{lem:power-sums} with exponent \(n\) proves the last estimate. At \(D=1\), every logarithmic left-hand side is zero.
\end{proof}
We recall the following result from \cite{Rob1955}.
\begin{lemm}\label{lem:stirling-remainder}
For every positive integer \(m\), we have
\[\left(m+\frac12\right)\log m-m
+\frac12\log(2\pi)+\frac1{12m+1}
\leqslant \log(m!),\]
and 
\[\log(m!)\leqslant\left(m+\frac12\right)\log m-m+\frac12\log(2\pi)+\frac1{12m}.\]
\end{lemm}

We have the following estimate for $C(1,D)$. 
\begin{prop}\label{prop:level-one}
For every integer \(D\geqslant1\), let
\[
C_1(D)=
2\sum_{k=0}^{D}\log(k!)-(D+1)\log(D!),
\]
and
\[
R_1(D)=
C_1(D)+\frac12D^2-\frac12D\log D
+\left(1-\frac12\log(2\pi)\right)D.
\]
Then we have
\[
\abs{R_1(D)}
\leqslant
\frac{25}{12}+\frac7{12}\log D.
\]
Equivalently,
\[
C_1(D)
=
-\frac12D^2+\frac12D\log D
+\left(-1+\frac12\log(2\pi)\right)D+R_1(D).
\]
\end{prop}

\begin{proof}
Lemma~\ref{lem:stirling-remainder} gives that, for every positive integer \(m\), we have 
\[
\left(m+\frac12\right)\log m-m+\frac12\log(2\pi)
+\frac1{12m+1}
\leqslant\log(m!),
\]
and
\[
\log(m!)
\leqslant
\left(m+\frac12\right)\log m-m+\frac12\log(2\pi)
+\frac1{12m}.
\]
Thus there is a real number \(\rho_m\), depending on $m$, such that
\[
0<\rho_m\leqslant\frac1{12m}
\]
and
\[
\log(m!)
=
\left(m+\frac12\right)\log m-m+\frac12\log(2\pi)+\rho_m.
\]

For every integer \(D\geqslant 2\), the definition of \(C_1(D)\) gives
\[
\begin{aligned}
C_1(D)-C_1(D-1)
&=2\log(D!)-(D+1)\log(D!)+D\log((D-1)!)\\
&=\log((D-1)!)-(D-1)\log D.
\end{aligned}
\]
Put \(m=D-1\). Substitution of the Stirling expression gives
\[
C_1(D)-C_1(D-1)
=
\frac12\log D
-\left(m+\frac12\right)\log\left(1+\frac1m\right)
-m+\frac12\log(2\pi)+\rho_m.
\]
The increment of
\[
-\frac12D^2+\frac12D\log D
+\left(-1+\frac12\log(2\pi)\right)D
\]
is
\[
-m-\frac32+\frac12\log(2\pi)+\frac12\log D
+\frac12m\log\left(1+\frac1m\right).
\]
Consequently,
\[
R_1(D)-R_1(D-1)
=
\frac32+\rho_m
-\left(\frac32m+\frac12\right)
 \log\left(1+\frac1m\right).
\]

For every real \(x\geqslant0\), we claim
\[
x-\frac{x^2}{2}\leqslant\log(1+x)\leqslant x.
\]
For the upper bound, the function \(x-\log(1+x)\) vanishes at zero and has derivative \(x/(1+x)\geqslant0\). For the lower bound, the function
\(\log(1+x)-x+x^2/2\) vanishes at zero and has derivative \(x^2/(1+x)\geqslant0\).

By the upper logarithmic estimate and \(\rho_m>0\), we obtain
\[
R_1(D)-R_1(D-1)
>
\frac32-\left(\frac32m+\frac12\right)\frac1m
=
-\frac1{2m}.
\]
By the lower logarithmic estimate and \(\rho_m\leqslant1/(12m)\), together with
\[
\left(\frac32m+\frac12\right)
\left(\frac1m-\frac1{2m^2}\right)
=
\frac32-\frac1{4m}-\frac1{4m^2},
\]
we obtain
\[
R_1(D)-R_1(D-1)
\leqslant
\frac1{12m}+\frac1{4m}+\frac1{4m^2}
=
\frac1{3m}+\frac1{4m^2}
\leqslant
\frac7{12m}.
\]
Therefore
\[
\abs{R_1(D)-R_1(D-1)}
\leqslant
\frac7{12(D-1)}
\qquad(D\geqslant2).
\]

At \(D=1\), \(C_1(1)=0\), and hence
\[
R_1(1)=\frac32-\frac12\log(2\pi).
\]
By the fact \(1<\pi<4\), we obtain
\[
0<\frac12\log(2\pi)<\frac12\log8<\frac32.
\]
It follows that
\[
\abs{R_1(1)}<\frac32.
\]

For \(D\geqslant2\), by summing the increment estimate, we obtain
\[
\abs{R_1(D)}
\leqslant
\frac32+\frac7{12}\sum_{m=1}^{D-1}\frac1m.
\]
The harmonic sum satisfies
\[
\sum_{m=1}^{D-1}\frac1m
\leqslant
1+\int_1^{D-1}\frac{\mathrm d x}{x}
=
1+\log(D-1)
\leqslant
1+\log D.
\]
Therefore
\[
\abs{R_1(D)}
\leqslant
\frac32+\frac7{12}(1+\log D)
=
\frac{25}{12}+\frac7{12}\log D.
\]
The same bound at \(D=1\) follows from the preceding estimate for \(R_1(1)\).
\end{proof}

\begin{prop}\label{prop:exact-recursion}
Let \(n\geqslant2\) and \(D\geqslant1\) be integers. Let
\[
C(n,D)=
\sum_{\substack{i_0,\ldots,i_n\geqslant0\\i_0+\cdots+i_n=D}}
\log\left(\frac{i_0!\cdots i_n!}{D!}\right)
\]
and
\[
Q(n,D)=
\sum_{m=0}^{D}\binom{m+n-1}{n-1}\log\binom{D}{m}.
\]
We define \(C(n-1,0)=0\), consistently with the displayed sum. Then we have
\[
C(n,D)=\sum_{m=0}^{D}C(n-1,m)-Q(n,D).
\]
Moreover,
\[
Q(n,D)
=
\sum_{k=1}^{D}
\left[
\binom{k+n-1}{n}-\binom{D-k+n}{n}
\right]\log k.
\]
Both identities hold at \(D=1\).
\end{prop}

\begin{proof}
For a tuple \((i_0,\ldots,i_n)\) in the definition of \(C(n,D)\), we denote
\[
m=i_0+\cdots+i_{n-1}.
\]
Then \(i_n=D-m\), and
\[
\log\left(
\frac{i_0!\cdots i_{n-1}!(D-m)!}{D!}
\right)
=
\log\left(\frac{i_0!\cdots i_{n-1}!}{m!}\right)
-\log\binom{D}{m}.
\]
For a fixed \(m\), summing the first term over the nonnegative tuples
\((i_0,\ldots,i_{n-1})\) with sum \(m\) gives \(C(n-1,m)\). The number of these tuples is
\(\displaystyle \binom{m+n-1}{n-1}\). Summing the second term therefore gives
\[
-\binom{m+n-1}{n-1}\log\binom{D}{m}.
\]
Summation over \(m=0,\ldots,D\) proves the first identity.

For \(0\leqslant m\leqslant D\), we have
\[
\log\binom{D}{m}
=
\sum_{k=1}^{m}\log\left(\frac{D-k+1}{k}\right),
\]
where the sum is empty for \(m=0\). Since all sums are finite, then we have
\[
\begin{aligned}
Q(n,D)
&=
\sum_{m=0}^{D}\binom{m+n-1}{n-1}
\sum_{k=1}^{m}\log\left(\frac{D-k+1}{k}\right)\\
&=
\sum_{k=1}^{D}
\left(
\sum_{m=k}^{D}\binom{m+n-1}{n-1}
\right)
\log\left(\frac{D-k+1}{k}\right).
\end{aligned}
\]
The finite hockey-stick identity gives
\[
\sum_{m=k}^{D}\binom{m+n-1}{n-1}
=
\binom{D+n}{n}-\binom{k+n-1}{n}.
\]
Also,
\[
\sum_{k=1}^{D}\log\left(\frac{D-k+1}{k}\right)
=
\log\left(\frac{D!}{D!}\right)=0.
\]
The constant term \(\displaystyle\binom{D+n}{n}\) therefore cancels, leaving
\[
Q(n,D)
=
\sum_{k=1}^{D}\binom{k+n-1}{n}
\bigl(\log k-\log(D-k+1)\bigr).
\]
In the second logarithmic sum, substitute \(j:=D-k+1\). Then
\[
\sum_{k=1}^{D}\binom{k+n-1}{n}\log(D-k+1)
=
\sum_{j=1}^{D}\binom{D-j+n}{n}\log j.
\]
Subtracting proves the second identity.

At \(D=1\), the first argument is unchanged, and the final sum is
\[
\left[\binom{n}{n}-\binom{n}{n}\right]\log1=0=Q(n,1).
\]
\end{proof}

\subsubsection*{The forcing quantity}\label{sec:forcing}

The goal of this part is to obtain an explicit estimate for \(Q(n,D)\) whose constants remain valid at \(D=1\).

\begin{prop}\label{prop:forcing-decomposition}
For integers \(n\geqslant2\) and \(D\geqslant1\), let
\[
r(n,x)=\binom{n+x}{n},\ x\in\mathbb N,
\]
\[
\SA(k,D)=
\sum_{m=2}^{D}
\bigl((m-1)^k-(D-m)^k\bigr)\log m,\ k\geqslant1,
\]
and
\[
Q(n,D)=
\sum_{m=0}^{D}\binom{n-1+m}{n-1}\log\binom{D}{m}.
\]
Then
\[
\begin{aligned}
Q(n,D)\geqslant{}&
\frac1{n!}\SA(n,D)
+\frac{n+1}{2(n-1)!}\SA(n-1,D)\\
&-(n-1)^2(D-1)^{n-1}\log D,
\end{aligned}
\]
and
\[
\begin{aligned}
Q(n,D)\leqslant{}&
\frac1{n!}\SA(n,D)
+\frac{n+1}{2(n-1)!}\SA(n-1,D)\\
&+(n-1)^2(D-1)^{n-1}\log D.
\end{aligned}
\]
At \(D=1\), every displayed sum is empty and both sides are zero.
\end{prop}

\begin{proof}
Let
\[
p_n(x)=
\frac{x^n}{n!}+\frac{n+1}{2(n-1)!}x^{n-1}\text{, and }
E_n(x)=r(n,x)-p_n(x).
\]
We first derive an exact summation identity. For \(1\leqslant j\leqslant D-1\), we denote
\[
a_j=r(n-1,j),
\
b_j=\log\binom Dj\text{, and }
A_j=\sum_{k=1}^{j}a_k.
\]
Then we have
\[
A_j
=
\sum_{k=1}^{j}\binom{n-1+k}{n-1}
=
\binom{n+j}{n}-1
=
r(n,j)-1.
\]
Since \(b_D=0\), by Abel's summation, we obtain
\[
\sum_{j=1}^{D-1}a_jb_j
=
\sum_{j=1}^{D-1}A_j(b_j-b_{j+1}).
\]
Furthermore,
\[
b_j-b_{j+1}
=
\log\frac{j+1}{D-j}.
\]
The terms \(m=0\text{ and }D\) in the sum \(Q(n,D)\) vanish. Substitution of the preceding identities, reindexing \(m=j+1\) in the numerator logarithms and \(m=D-j\) in the denominator logarithms, and using the fact
\[
\sum_{j=1}^{D-1}\log\frac{j+1}{D-j}
=
\log D
=
r(n,0)\log D,
\]
we obtain 
\[
Q(n,D)
=
\sum_{m=2}^{D}
\bigl(r(n,m-1)-r(n,D-m)\bigr)\log m.
\]
For \(D=1\), this is an empty sum. We may therefore assume \(D\geqslant2\) and put $N=D-1$ for simplicity. 

The product formula
\[
r(n,x)=\frac1{n!}\prod_{j=1}^{n}(x+j)
\]
shows that the coefficients of \(x^n\) and \(x^{n-1}\) are respectively
\[
\frac1{n!}
\quad\text{and}\quad
\frac{1+\cdots+n}{n!}
=
\frac{n+1}{2(n-1)!}.
\]
Thus \(E_n(x)\) is a polynomial of degree at most \(n-2\), all of whose coefficients are nonnegative. Also,
\[
E_n(0)=1.
\]
If \(x\geqslant1\), every monomial \(x^j\) satisfying \(0\leqslant j\leqslant n-2\) is at most \(x^{n-2}\). Hence
\[
0\leqslant E_n(x)\leqslant e_nx^{n-2},
\]
where
\[
e_n=E_n(1)
=
n+1-\frac1{n!}-\frac{n+1}{2(n-1)!}
<n+1.
\]

Substituting \(r(n,x)=p_n(x)+E_n(x)\) into the exact identity gives
\[
Q(n,D)
=
\frac1{n!}\SA(n,D)
+\frac{n+1}{2(n-1)!}\SA(n-1,D)
+T_n(D),
\]
where
\[
T_n(D)=
\sum_{m=2}^{D}
\bigl(E_n(m-1)-E_n(D-m)\bigr)\log m.
\]
Then it remains to prove
\[
\abs{T_n(D)}
\leqslant
(n-1)^2N^{n-1}\log D.
\]

For \(n=2\), we have
\[
r(2,x)=\frac{x^2}{2}+\frac{3x}{2}+1,
\]
so \(E_2(x)=1\). Every summand defining \(T_2(D)\) is zero.

For \(n=3\), we have
\[
E_3(x)=1+\frac{11}{6}x.
\]
Writing \(j=m-1\), we obtain
\[
T_3(D)
=
\frac{11}{6}\sum_{j=1}^{N}(2j-N)\log(j+1).
\]
Since \(0\leqslant\log(j+1)\leqslant\log D\) and \(\abs{2j-N}\leqslant N\), then we have
\[
\abs{T_3(D)}
\leqslant
\frac{11}{6}N^2\log D
<
4N^2\log D.
\]

Next, we suppose \(n\geqslant4\). Positivity and the bound for \(E_n\) give
\[
\begin{aligned}
\abs{T_n(D)}
&\leqslant
\log D
\left(
E_n(N)+E_n(0)+2\sum_{x=1}^{N-1}E_n(x)
\right)\\
&\leqslant
\log D
\left(
e_nN^{n-2}+1+2e_n\sum_{x=1}^{N-1}x^{n-2}
\right).
\end{aligned}
\]
Since \(x^{n-2}\) is increasing, then 
\[
\sum_{x=1}^{N-1}x^{n-2}
\leqslant
\int_1^N t^{n-2}\,\mathrm d t
\leqslant
\frac{N^{n-1}}{n-1}.
\]
For \(n=4\),
\[
e_4
=
5-\frac1{24}-\frac5{12}
=
\frac{109}{24}.
\]
By the facts \(N^2\leqslant N^3\), \(1\leqslant N^3\), and the preceding sum estimate, we obtain
\[
\abs{T_4(D)}
\leqslant
\left(
\frac{109}{24}+1+\frac{109}{36}
\right)N^3\log D
=
\frac{617}{72}N^3\log D
<
9N^3\log D.
\]

Finally, let \(n\geqslant5\). Since \(e_n<n+1\),
\[
\abs{T_n(D)}
<
\left(
n+2+\frac{2(n+1)}{n-1}
\right)N^{n-1}\log D.
\]
Here
\[
\frac{2(n+1)}{n-1}
=
2+\frac4{n-1}
\leqslant3
\]
and
\[
n+5\leqslant(n-1)^2,
\]
because
\[
(n-1)^2-(n+5)=(n-4)(n+1)\geqslant0.
\]
This proves the required estimate for \(T_n(D)\), and substitution into the exact decomposition yields both asserted inequalities.
\end{proof}

\begin{theo}\label{thm:signed-sum}
For every pair of integers \(n\geqslant2\) and \(D\geqslant1\), we denote
\[
\SA(n,D)=
\sum_{m=2}^{D}
\bigl((m-1)^n-(D-m)^n\bigr)\log m,
\]
\[
\mathcal H_n=\sum_{j=1}^{n}\frac1j,
\qquad
K_n=\frac{n^2+13n+62}{12}.
\]
Then
\[
\abs{
\SA(n,D)
-\frac{\mathcal H_n}{n+1}D^{n+1}
+\frac12D^n\log D
}
\leqslant
K_nD^n.
\]
The inequality is uniform in every positive integer \(D\), including \(D=1\).
\end{theo}

\begin{proof}
We begin with a midpoint estimate. If \(g\) is twice continuously differentiable on
\( \left[m-\frac{1}{2},m+\frac{1}{2}\right]\), Taylor's formula with integral remainder gives
\[
\abs{
g(m)-\int_{m-\frac{1}{2}}^{m+\frac{1}{2}}g(x)\,\mathrm d x
}
\leqslant
\frac1{24}
\sup_{m-\frac{1}{2}\leqslant x\leqslant m+\frac{1}{2}}\abs{g''(x)}.
\]
After subtracting \(g(m)+g'(m)(x-m)\), the remainder is bounded by
\(\tfrac12M(x-m)^2\), where \(M\) is the displayed supremum. The linear term integrates to zero, and
\[
\frac12\int_{-\frac{1}{2}}^{\frac{1}{2}}t^2\,\mathrm d t=\frac1{24}.
\]

Assume first that \(D\geqslant2\), and put
\[
f(x)=\bigl((x-1)^n-(D-x)^n\bigr)\log x,\ \frac32\leqslant x\leqslant D+\frac12.
\]
Direct differentiation gives
\[
\begin{aligned}
f''(x)={}&
-\frac{(x-1)^n-(D-x)^n}{x^2}\\
&+\frac{2n\bigl((x-1)^{n-1}+(D-x)^{n-1}\bigr)}x\\
&+n(n-1)
\bigl((x-1)^{n-2}-(D-x)^{n-2}\bigr)\log x.
\end{aligned}
\]
For \(x\in\left[m-\tfrac12,m+\tfrac12\right]\), \(2\leqslant m\leqslant D\), we have
\[
\abs{x-1}\leqslant D,
\
\abs{D-x}\leqslant D,
\
x\geqslant m-\frac12.
\]
The first term of \(f''\), summed over the cells, contributes at most
\[
2D^n\sum_{m=2}^{D}\frac1{\left(m-\frac{1}{2}\right)^2}
\leqslant
2D^n\sum_{j=1}^{\infty}\frac1{j^2}
\leqslant
4D^n,\]
where the last inequality follows from
\[
\sum_{j=1}^{\infty}\frac1{j^2}
=
1+\sum_{j=2}^{\infty}\frac1{j^2}
\leqslant
1+\int_1^\infty\frac{\mathrm d x}{x^2}
=
2.
\]
The second term contributes at most
\[
4nD^{n-1}\sum_{m=2}^{D}\frac1{m-\frac{1}{2}}
\leqslant
4nD^{n-1}(D-1)
\leqslant
4nD^n.
\]
For the third term, we have
\[
\abs{(x-1)^{n-2}-(D-x)^{n-2}}
\leqslant
\abs{x-1}^{n-2}+\abs{D-x}^{n-2}
\leqslant
2D^{n-2}.
\]
There are \(D-1\leqslant D\) cells, and
\[
\log\left(D+\frac12\right)\leqslant D
\]
by \(\log y\leqslant y-1\). Hence this term contributes at most
\[
2n(n-1)D^{n-2}\cdot D\cdot
\log\left(D+\frac12\right)
\leqslant
2n(n-1)D^n.
\]
Applying the midpoint estimate on every cell, we obtain
\[
\abs{\SA(n,D)-J_n(D)}
\leqslant
\frac{2n^2+2n+4}{24}D^n
=
\frac{n^2+n+2}{12}D^n,
\]
where
\[
J_n(D)=
\int_{3/2}^{D+1/2}
\bigl((x-1)^n-(D-x)^n\bigr)\log x\,\mathrm d x.
\]

We compare this integral with two exactly evaluable integrals. Suppose
\[
A=\int_{\frac{3}{2}}^{D+\frac{1}{2}}(x-1)^n\log x\,\mathrm d x,
\
B=\int_{\frac{3}{2}}^{D+\frac{1}{2}}(D-x)^n\log x\,\mathrm d x,
\]
and 
\[
A_0=\int_0^D u^n\log u\,\mathrm d u,
\
B_0=\int_0^D(D-x)^n\log x\,\mathrm d x.
\]
Then \(J_n(D)=A-B\). After substituting \(u=x-1\) for simplicity, we suppose
\[
L=\int_0^{\frac{1}{2}}u^n\log u\,\mathrm d u,
\]
\[
U=\int_{D-\frac{1}{2}}^{D}u^n\log u\,\mathrm d u,
\]
and
\[
V=
\int_{\frac{1}{2}}^{D-\frac{1}{2}}
u^n\log\left(1+\frac1u\right)\mathrm d u.
\]
The definitions give
\[
A-A_0=-L-U+V.
\]
Since \(L\leqslant0\), then
\[
0\leqslant-L
\leqslant
\int_0^1(-\log u)\,\mathrm d u
=
1.
\]
Also, we have \(\log\left(1+\frac{1}{u}\right)\leqslant \frac{1}{u}\), so
\[
0\leqslant V
\leqslant
\int_{\frac{1}{2}}^{D-\frac{1}{2}}u^{n-1}\,\mathrm d u
\leqslant
\frac{D^n}{n}.
\]

Let \(h(u)=u^n\log u\). On \(\left[D-\tfrac12,D\right]\), we have
\[
0\leqslant h'(u)
=
u^{n-1}(n\log u+1)
\leqslant
D^{n-1}(n\log D+1).
\]
Therefore
\[
\begin{aligned}
0
&\leqslant
\frac12D^n\log D-U=
\int_{D-\frac{1}{2}}^{D}\bigl(h(D)-h(u)\bigr)\,\mathrm d u\\
&\leqslant
D^{n-1}(n\log D+1)
\int_{D-\frac{1}{2}}^{D}(D-u)\,\mathrm d u\\
&=
\frac18D^{n-1}(n\log D+1)\leqslant
\frac18\left(\frac n\ee+1\right)D^n,
\end{aligned}
\]
where Lemma~\ref{lem:log-absorption} was applied in the last line. Hence, we obtain
\[
\abs{
A-A_0+\frac12D^n\log D
}
\leqslant
\left(
1+\frac18\left(\frac n\ee+1\right)+\frac1n
\right)D^n.
\]

For the second integral, we have
\[
B-B_0
=
-\int_0^{\frac{3}{2}}(D-x)^n\log x\,\mathrm d x
+
\int_D^{D+\frac{1}{2}}(D-x)^n\log x\,\mathrm d x.
\]
On \([0,\tfrac32]\), we have \(\abs{D-x}\leqslant D\), and
\[
\int_0^{\frac{3}{2}}\abs{\log x}\,\mathrm d x
=
1+\int_1^{\frac{3}{2}}\log x\,\mathrm d x
\leqslant
1+\int_1^{\frac{3}{2}}(x-1)\,\mathrm d x
=
\frac98<2.
\]
Thus, the first integral has absolute value less than \(2D^n\). On
\(\left[D,D+\tfrac12\right]\), we have \(\abs{D-x}\leqslant\tfrac12\), so the second has absolute value at most
\[
\frac12\left(\frac12\right)^nD
\leqslant
\frac18D
\leqslant
\frac18D^n.
\]
Consequently,
\[
\abs{B-B_0}<3D^n.
\]
Using the fact \(\frac{1}{\ee}<1\) and \(\frac{1}{n}\leqslant\frac{1}{2}\), we obtain
\[
\abs{
J_n(D)-(A_0-B_0)+\frac12D^n\log D
}
\leqslant
(n+5)D^n.
\]

It remains to evaluate \(A_0-B_0\). An elementary integration gives
\[
A_0
=
D^{n+1}
\left(
\frac{\log D}{n+1}-\frac1{(n+1)^2}
\right).
\]
After substituting \(x=Dt\), we have
\[
B_0
=
D^{n+1}
\left(
\frac{\log D}{n+1}
+
\int_0^1(1-t)^n\log t\,\mathrm d t
\right).
\]
We claim that
\[
\int_0^1(1-t)^n\log t\,\mathrm d t
=
-\frac{H_{n+1}}{n+1}.
\]
Indeed, put \(u=1-t\) and we have 
\[
-\log(1-u)=\sum_{k=1}^{\infty}\frac{u^k}{k}.
\]
Monotone convergence gives
\[
-\int_0^1u^n\log(1-u)\,\mathrm d u
=
\sum_{k=1}^{\infty}\frac1{k(n+k+1)}.
\]
Since
\[
\frac1{k(n+k+1)}
=
\frac1{n+1}
\left(
\frac1k-\frac1{k+n+1}
\right),
\]
the partial sums telescope to \(\mathcal H_{n+1}/(n+1)\). By the fact
\[
\mathcal H_{n+1}=\mathcal H_n+\frac1{n+1},
\]
we conclude that
\[
A_0-B_0
=
\frac{\mathcal H_n}{n+1}D^{n+1}.
\]
Thus, we have 
\[
\abs{
J_n(D)-\frac{\mathcal H_n}{n+1}D^{n+1}
+\frac12D^n\log D
}
\leqslant
(n+5)D^n.
\]
Combining this with the midpoint estimate gives
\[
\begin{aligned}
\abs{
\SA(n,D)-\frac{\mathcal H_n}{n+1}D^{n+1}
+\frac12D^n\log D
}
&\leqslant
\left(
n+5+\frac{n^2+n+2}{12}
\right)D^n\\
&=
K_nD^n
\end{aligned}
\]
for \(D\geqslant2\).

If \(D=1\), then \(\SA(n,1)=0\) and \(\log1=0\). Since \(\mathcal H_n\leqslant n\), then
\[
\frac{\mathcal H_n}{n+1}<1<K_n.
\]
The same estimate therefore holds at \(D=1\).
\end{proof}

\begin{theo}\label{thm:forcing-bound}
For every pair of integers \(n\geqslant2\) and \(D\geqslant1\), let
\[
\mathcal H_n=\sum_{j=1}^{n}\frac1j,
\
Q(n,D)=
\sum_{m=0}^{D}
\binom{m+n-1}{n-1}\log\binom{D}{m},
\]
\[
\beta_n=1-\frac12\log(2\pi)+\frac1{2n},
\]
and
\[
q_n(D)=
\frac{\mathcal H_n}{(n+1)!}D^{n+1}
-\frac1{2n!}D^n\log D
+\frac1{n!}
\left(
\frac{n+1}{2}\mathcal H_{n-1}-\beta_n
\right)D^n.
\]
For every integer \(r\geqslant2\), suppose
\[
K_r=\frac{r^2+13r+62}{12}.
\]
Let \(W_2=10\), and for every integer \(n\geqslant3\), let
\[
W_n=
\frac{K_n}{n!}
+\frac{n+1}{2(n-1)!}
 \left(K_{n-1}+\frac1{2\ee}\right)
+\frac{(n-1)^2}{\ee}
+\frac{\beta_n}{n!}.
\]
Then
\[
\abs{Q(n,D)-q_n(D)}
\leqslant
W_nD^n
\]
for every \(n\geqslant2\) and \(D\geqslant1\). Moreover,
\[
0<W_n\leqslant\frac52n^2.
\]
\end{theo}

\begin{proof}
We first treat the case \(n=2\). Let
\[
L_D=\sum_{m=0}^{D}\log\binom Dm.
\]
Expanding the binomial coefficients gives
\[
L_D
=
(D+1)\log(D!)
-2\sum_{m=0}^{D}\log(m!)
=
-C_1(D).
\]
The symmetry under \(m\mapsto D-m\) gives
\[
\sum_{m=0}^{D}m\log\binom Dm
=
\frac D2L_D.
\]
Therefore
\[
Q(2,D)
=
\sum_{m=0}^{D}(m+1)\log\binom Dm
=
\frac{D+2}{2}L_D
=
-\frac{D+2}{2}C_1(D).
\]
Proposition~\ref{prop:level-one} yields
\[
\begin{aligned}
Q(2,D)={}&
\frac14D^3-\frac14D^2\log D
+\left(1-\frac14\log(2\pi)\right)D^2\\
&-\frac12D\log D
+\left(1-\frac12\log(2\pi)\right)D
-\frac{D+2}{2}R_1(D).
\end{aligned}
\]
Since \(\mathcal H_2=3/2\), \(\mathcal H_1=1\), and
\[
\frac1{2!}
\left(
\frac32\mathcal H_1-\beta_2
\right)
=
\frac12
\left(
\frac32-\frac54+\frac12\log(2\pi)
\right)
=
\frac18+\frac14\log(2\pi),
\]
then we have
\[
q_2(D)
=
\frac14D^3-\frac14D^2\log D
+\left(
\frac18+\frac14\log(2\pi)
\right)D^2.
\]
Thus
\[
\begin{aligned}
Q(2,D)-q_2(D)={}&
\left(
\frac78-\frac12\log(2\pi)
\right)D^2
-\frac12D\log D\\
&+
\left(
1-\frac12\log(2\pi)
\right)D
-\frac{D+2}{2}R_1(D).
\end{aligned}
\]
The elementary bounds used in Proposition~\ref{prop:level-one} give
\(0<\log(2\pi)<3\). Hence the absolute values of the first and third displayed coefficients are at most \(1\). Also,
\[
\log D\leqslant\frac D\ee<\frac D2
\]
and
\[
\abs{R_1(D)}
\leqslant
\frac{25}{12}+\frac7{12}\log D
\leqslant
\frac{25}{12}+\frac7{24}D.
\]
Since \(D+2\leqslant3D\), then
\[
\frac{D+2}{2}\abs{R_1(D)}
\leqslant
\frac{57}{16}D^2.
\]
It follows that
\[
\abs{Q(2,D)-q_2(D)}
<
\left(
1+\frac14+1+\frac{57}{16}
\right)D^2
<
10D^2.
\]
This proves the assertion for \(n=2\).

We next record that
\[
0<\beta_n<1\text{ when }n\geqslant2.
\]
The elementary calculation
\[
\frac{22}{7}-\pi
=
\int_0^1
\frac{x^4(1-x)^4}{1+x^2}\,\mathrm d x
>0
\]
gives \(\pi<\frac{22}{7}\). Also, we have
\[
\ee>1+1+\frac12+\frac16=\frac83.
\]
Consequently, we have
\[
2\pi<\frac{44}{7}<\frac{64}{9}<\ee^2,
\]
so \(\log(2\pi)<2\), which proves \(\beta_n>0\). The perimeter of the regular hexagon inscribed in the unit circle gives \(\pi>3\). Moreover,
\[
\ee
=
\sum_{j=0}^{\infty}\frac1{j!}
<
1+1+\frac12
+\sum_{j=3}^{\infty}\frac1{2\cdot3^{j-2}}
=
\frac{11}{4}.
\]
Thus \(2\pi>6>\ee\), and \(\log(2\pi)>1>1/n\). Hence \(\beta_n<1\).

Now fix \(n\geqslant3\). For \(r\geqslant1\), we denote
\[
S_r(D)=
\sum_{m=2}^{D}
\bigl((m-1)^r-(D-m)^r\bigr)\log m.
\]
Proposition~\ref{prop:forcing-decomposition} gives
\[
\abs{
Q(n,D)
-
\left(
\frac{S_n(D)}{n!}
+\frac{n+1}{2(n-1)!}S_{n-1}(D)
\right)
}
\leqslant
(n-1)^2(D-1)^{n-1}\log D.
\]
Theorem~\ref{thm:signed-sum}, applied with \(r=n\) and \(r=n-1\), gives real numbers
\(\eta_n(D)\) and \(\eta_{n-1}(D)\) satisfying
\[
S_n(D)
=
\frac{\mathcal H_n}{n+1}D^{n+1}
-\frac12D^n\log D
+\eta_n(D)\text{ with }
\abs{\eta_n(D)}\le K_nD^n,
\]
and
\[
S_{n-1}(D)
=
\frac{\mathcal H_{n-1}}nD^n
-\frac12D^{n-1}\log D
+\eta_{n-1}(D)\text{ with }
\abs{\eta_{n-1}(D)}
\le
K_{n-1}D^{n-1}.
\]
Since
\[
\frac{n+1}{2(n-1)!}\frac{\mathcal H_{n-1}}n
=
\frac{n+1}{2n!}\mathcal H_{n-1},
\]
substitution gives
\[
\begin{aligned}
&
\frac{S_n(D)}{n!}
+\frac{n+1}{2(n-1)!}S_{n-1}(D)
-q_n(D)\\
&\quad=
\frac{\eta_n(D)}{n!}
+\frac{n+1}{2(n-1)!}\eta_{n-1}(D)
-\frac{n+1}{4(n-1)!}D^{n-1}\log D
+\frac{\beta_n}{n!}D^n.
\end{aligned}
\]
Lemma~\ref{lem:log-absorption} gives
\[
D^{n-1}\log D\leqslant\frac1\ee D^n
\]
and
\[
(D-1)^{n-1}\log D
\leqslant
D^{n-1}\log D
\leqslant
\frac1\ee D^n.
\]
Taking absolute values yields
\[
\begin{aligned}
\abs{Q(n,D)-q_n(D)}
&\leqslant
\left[
\frac{K_n}{n!}
+\frac{n+1}{2(n-1)!}K_{n-1}
+\frac{n+1}{4\ee(n-1)!}\right.\\
& \qquad \left.
+\frac{\beta_n}{n!}
+\frac{(n-1)^2}{\ee}
\right]D^n\\
&=
W_nD^n.
\end{aligned}
\]
Every summand defining \(W_n\) is non-negative, and \(K_n>0\), so \(W_n>0\).

For the quadratic bound, the case \(n=2\) is
\[
W_2=10=\frac52\cdot2^2.
\]
Let \(n\geqslant3\). An elementary calculation gives
\[
13n^2-(n^2+13n+62)
=
12n^2-13n-62
\geqslant7,
\]
so
\[
K_n\leqslant\frac{13}{12}n^2.
\]
Likewise,
\[
12n^2-(n^2+11n+50)
=
11n^2-11n-50
\ge16,
\]
and hence
\[
K_{n-1}
=
\frac{n^2+11n+50}{12}
\leqslant n^2.
\]
Furthermore,
\[
n!\geqslant6,
\qquad
\frac{n+1}{2(n-1)!}\leqslant1.
\]
The second inequality is an equality at \(n=3\). For \(n\geqslant4\), it follows first at \(n=4\) from \(2(n-1)!\geqslant12>n+1\), and thereafter by induction.

Since \(\ee>2\), then
\[
\frac1{2\ee}<\frac14,
\qquad
\frac1\ee<\frac12.
\]
By the fact \(n^2\geqslant9\) and \(0<\beta_n<1\), we obtain
\[
\begin{aligned}
W_n
&\leqslant
\frac{13}{72}n^2+n^2+\frac14
+\frac12(n-1)^2+\frac16\\
&\leqslant
\left(
\frac{13}{72}+1+\frac1{36}+\frac12+\frac1{54}
\right)n^2\\
&=
\frac{373}{216}n^2
<
2n^2
<
\frac52n^2.
\end{aligned}
\]
\end{proof}

\subsection{The proof of Theorem \ref{thm:estimate of C(n,D)}}\label{sec:main-proof}

In this part, we will give a proof of Theorem~\ref{thm:estimate of C(n,D)}.

\begin{prop}\label{prop:summation-defect}
Let
\[
\mathcal H_n=\sum_{j=1}^{n}\frac1j,
\]
and let \(a_n,b_n,c_n\) be the coefficients in Theorem~\ref{thm:estimate of C(n,D)}. For every integer \(n\geqslant1\) and \(D\geqslant1\), let
\[
M_n(D)=a_nD^{n+1}+b_nD^n\log D+c_nD^n.
\]
For every integer \(n\geqslant2\), let
\[
\kappa_n=
-1+\frac12\log(2\pi)-\frac1{2n}
+\frac{n+1}{2}\mathcal H_{n-1}
\]
and
\[
q_n(D):=
\frac{\mathcal H_n}{(n+1)!}D^{n+1}
-\frac1{2n!}D^n\log D
+\frac{\kappa_n}{n!}D^n.
\]
We denote
\[
F_n(D)=
\sum_{m=1}^{D}M_{n-1}(m)-q_n(D)-M_n(D).
\]
Then
\[
\begin{aligned}
F_n(D)={}&
a_{n-1}
\left(
\sum_{m=1}^{D}m^n
-\frac{D^{n+1}}{n+1}
-\frac{D^n}{2}
\right)\\
&+
b_{n-1}
\left(
\sum_{m=1}^{D}m^{n-1}\log m
-\frac{D^n}{n}\log D
+\frac{D^n}{n^2}
\right)\\
&+
c_{n-1}
\left(
\sum_{m=1}^{D}m^{n-1}
-\frac{D^n}{n}
\right).
\end{aligned}
\]
Moreover, if
\[
E_n:=
\frac{\abs{a_{n-1}}}{2}
+
\left(\frac14+\frac1\ee\right)\abs{b_{n-1}}
+
\abs{c_{n-1}},
\]
then
\[
\abs{F_n(D)}\leqslant E_nD^n
\]
for every \(n\geqslant2\) and \(D\geqslant1\).
\end{prop}

\begin{proof}
Substitution of the explicit coefficients from Theorem~\ref{thm:estimate of C(n,D)}, followed by collection over a common denominator and the use of
\[
\mathcal H_n=\mathcal H_{n-1}+\frac1n,
\]
we obtain 
\[
a_n
=
\frac{a_{n-1}}{n+1}
-\frac{\mathcal H_n}{(n+1)!},
\]
\[
b_n
=
\frac{b_{n-1}}n+\frac1{2n!},
\]
and
\[
c_n
=
\frac{c_{n-1}}n
-\frac{b_{n-1}}{n^2}
+\frac{a_{n-1}}2
-\frac{\kappa_n}{n!}.
\]

Expanding the definition of \(F_n(D)\), we obtain 
\[
\begin{aligned}
F_n(D)={}&
a_{n-1}\sum_{m=1}^{D}m^n
+
b_{n-1}\sum_{m=1}^{D}m^{n-1}\log m
+
c_{n-1}\sum_{m=1}^{D}m^{n-1}\\
&-\frac{\mathcal H_n}{(n+1)!}D^{n+1}
+\frac1{2n!}D^n\log D
-\frac{\kappa_n}{n!}D^n\\
&-a_nD^{n+1}-b_nD^n\log D-c_nD^n.
\end{aligned}
\]
Insert and subtract
\[
\frac{a_{n-1}}{n+1}D^{n+1}
+\frac{a_{n-1}}2D^n,
\]
\[
\frac{b_{n-1}}nD^n\log D
-\frac{b_{n-1}}{n^2}D^n,
\]
and
\[
\frac{c_{n-1}}nD^n.
\]
The coefficient of \(D^{n+1}\) outside the three resulting remainders is
\[
\frac{a_{n-1}}{n+1}
-\frac{\mathcal H_n}{(n+1)!}
-a_n
=
0.
\]
The coefficient of \(D^n\log D\) is
\[
\frac{b_{n-1}}n+\frac1{2n!}-b_n=0.
\]
The remaining coefficient of \(D^n\) is
\[
\frac{a_{n-1}}2
-\frac{b_{n-1}}{n^2}
+\frac{c_{n-1}}n
-\frac{\kappa_n}{n!}
-c_n
=
0.
\]
This proves the exact formula for \(F_n(D)\).

Since the function \(x^n\) is increasing on \([0,D]\), integral comparison gives
\[
\int_0^D x^n\,\mathrm d x
\leqslant
\sum_{m=1}^{D}m^n
\leqslant
\int_0^D x^n\,\mathrm d x+D^n.
\]
Therefore
\[
0
\leqslant
\sum_{m=1}^{D}m^n-\frac{D^{n+1}}{n+1}
\leqslant
D^n,
\]
and hence
\[
\abs{
\sum_{m=1}^{D}m^n
-\frac{D^{n+1}}{n+1}
-\frac{D^n}{2}
}
\leqslant
\frac{D^n}{2}.
\]
The same argument for \(x^{n-1}\) gives
\[
0
\leqslant
\sum_{m=1}^{D}m^{n-1}-\frac{D^n}{n}
\leqslant
D^{n-1}
\leqslant
D^n.
\]

For
\[
f(x)=x^{n-1}\log x,
\]
we have
\[
f'(x)
=
x^{n-2}\bigl((n-1)\log x+1\bigr)>0,\ x\geqslant1.
\]
Since \(f(1)=0\), monotone integral comparison yields
\[
\int_1^D f(x)\,\mathrm d x
\leqslant
\sum_{m=1}^{D}f(m)
\leqslant
\int_1^D f(x)\,\mathrm d x+f(D).
\]
Direct integration gives
\[
\int_1^D x^{n-1}\log x\,\mathrm d x
=
\frac{D^n}{n}\log D
-\frac{D^n}{n^2}
+\frac1{n^2}.
\]
Thus
\[
\begin{aligned}
0
&\leqslant
\sum_{m=1}^{D}m^{n-1}\log m
-\frac{D^n}{n}\log D
+\frac{D^n}{n^2}\\
&\leqslant
\frac1{n^2}+D^{n-1}\log D.
\end{aligned}
\]
Since \(n\geqslant2\) and \(D\geqslant1\),
\[
\frac1{n^2}\leqslant\frac{D^n}{4}.
\]
Lemma~\ref{lem:log-absorption} gives
\[
D^{n-1}\log D\leqslant\frac{D^n}{\ee}.
\]
Therefore
\[
0
\leqslant
\sum_{m=1}^{D}m^{n-1}\log m
-\frac{D^n}{n}\log D
+\frac{D^n}{n^2}
\leqslant
\left(\frac14+\frac1\ee\right)D^n.
\]
Applying the triangle inequality to the exact formula for \(F_n(D)\), we have
\[
\abs{F_n(D)}\leqslant E_nD^n.
\]
\end{proof}

\begin{lemm}\label{lem:defect-four}
For every integer \(n\geqslant2\), we have 
\[
E_n<4.
\]
Consequently, we have
\[
\abs{F_n(D)}\leqslant4D^n,\ n\geqslant2,\ D\geqslant1.
\]
\end{lemm}

\begin{proof}
Put \(k=n-1\) for simplicity. Since
\[
\mathcal H_{k+1}-1=\sum_{j=2}^{k+1}\frac1j\leqslant k
\]
and \(k!\geqslant k\), then 
\[
\abs{a_k}
=
\frac{\mathcal H_{k+1}-1}{k!}
\leqslant1.
\]
Also,
\[
\abs{b_k}=\frac1{2(k-1)!}\leqslant\frac12.
\]

It remains to give an estimate of \(c_k\). We have \(0<\log(2\pi)<3\), and a direct substitution gives
\[
\abs{c_1}<1,
\
c_2=-1,\
c_3
=
-\frac{65/6+\log(2\pi)}{12},
\
c_4
=
-\frac{137/6+2\log(2\pi)}{48}.
\]
Thus \(\abs{c_k}<2\) for \(1\leqslant k\leqslant 4\).

For \(k\geqslant 5\), the inequality \(\mathcal H_k\leqslant k\) gives
\[
\abs{c_k}
\leqslant
\frac{k^3+3k^2+4k-1}{2k!}.
\]
At \(k=5\),
\[
k^3+3k^2+4k-1=219\leqslant240=2k!.
\]
Let \(P(k)=k^3+3k^2+4k-1\), then
\[
(k+1)P(k)-P(k+1)
=
k^4+3k^3+k^2-10k-8>0
\]
for \(k\geqslant5\). Induction therefore gives \(P(k)\leqslant2k!\) for all \(k\geqslant5\), so
\[
\abs{c_k}\leqslant1.
\]
Hence \(\abs{c_k}<2\) for every \(k\geqslant1\).

Finally, \(\ee>2\), so
\[
\frac14+\frac1\ee<\frac34.
\]
It follows that
\[
E_n
<
\frac12+\frac38+2
=
\frac{23}{8}
<
4.
\]
The assertion for \(F_n(D)\) now follows from Proposition~\ref{prop:summation-defect}.
\end{proof}

Now we give the proof of Theorem~\ref{thm:estimate of C(n,D)}. 
\begin{proof}[Proof of Theorem~\ref{thm:estimate of C(n,D)}]
Let
\[
R_n(D)=C(n,D)-M_n(D).
\]

\noindent\textbf{Step 1. -- }
We establish the level-one bound. The coefficient formulas give
\[
a_1=-\frac12,
\qquad
b_1=\frac12,
\qquad
c_1=-1+\frac12\log(2\pi).
\]
Thus
\[
M_1(D)
=
-\frac12D^2+\frac12D\log D
+\left(-1+\frac12\log(2\pi)\right)D.
\]
Expanding the definition of \(C(1,D)\) gives
\[
C(1,D)
=
2\sum_{k=0}^{D}\log(k!)-(D+1)\log(D!)
=
C_1(D).
\]
Proposition~\ref{prop:level-one} therefore yields
\[
\abs{R_1(D)}
\leqslant
\frac{25}{12}+\frac7{12}\log D.
\]
Lemma~\ref{lem:log-absorption} gives \(\log D\leqslant D/\ee\), while
\(25/12\leqslant(25/12)D\). Consequently,
\[
\abs{R_1(D)}
\leqslant
\left(
\frac{25}{12}+\frac7{12\ee}
\right)D
=
B_0D.
\]

\noindent\textbf{Step 2. -- }
We derive the exact remainder recurrence. For \(n\geqslant2\), let \(Q(n,D)\) be as in Proposition~\ref{prop:exact-recursion}, and let \(q_n(D)\) and \(F_n(D)\) be as in Proposition~\ref{prop:summation-defect}. Proposition~\ref{prop:exact-recursion} gives
\[
C(n,D)
=
\sum_{m=1}^{D}C(n-1,m)-Q(n,D),
\]
because \(C(n-1,0)=0\). By definition, 
\[
F_n(D)
=
\sum_{m=1}^{D}M_{n-1}(m)-q_n(D)-M_n(D).
\]
Subtracting \(M_n(D)\) from the exact recurrence gives
\[
R_n(D)
=
\sum_{m=1}^{D}R_{n-1}(m)
-\bigl(Q(n,D)-q_n(D)\bigr)
+F_n(D).
\]

\noindent\textbf{Step 3. -- }
We propagate the uniform estimate. Theorem~\ref{thm:forcing-bound} gives
\[
\abs{Q(n,D)-q_n(D)}\leqslant W_nD^n
\]
and
\[
0<W_n\leqslant\frac52n^2.
\]
Lemma~\ref{lem:defect-four} gives
\[
\abs{F_n(D)}\leqslant4D^n.
\]
Suppose that
\[
\abs{R_{n-1}(m)}
\leqslant
L_{n-1}m^{n-1}
\]
for every integer \(m\geqslant1\), where \(L_1=B_0\). The exact recurrence, the triangle inequality, and Lemma~\ref{lem:power-sums} give
\[
\begin{aligned}
\abs{R_n(D)}
&\leqslant
L_{n-1}\sum_{m=1}^{D}m^{n-1}
+(W_n+4)D^n\\
&\leqslant
L_{n-1}D^n+(W_n+4)D^n\\
&=
(L_{n-1}+W_n+4)D^n.
\end{aligned}
\]
Suppose
\[
L_n=L_{n-1}+W_n+4.
\]
Starting with \(L_1=B_0\), induction proves
\[
\abs{R_n(D)}\leqslant L_nD^n
\]
for every \(n\geqslant2\) and \(D\geqslant1\). Iterating the recurrence for \(L_n\) gives
\[
L_n
=
B_0+\sum_{r=2}^{n}(W_r+4).
\]

\noindent\textbf{Step 4. -- }
We simplify the constant polynomially. For \(r\geqslant2\), we have $W_r\leqslant\frac52r^2$ and $4\leqslant r^2$. Hence, we have
\[
W_r+4\leqslant\frac72r^2.
\]
Therefore
\[
\begin{aligned}
L_n
&\leqslant
B_0+\frac72\sum_{r=2}^{n}r^2\\
&=
B_0+\frac72
\left(
\frac{n(n+1)(2n+1)}6-1
\right)\\
&=
B_0+\frac72S_n.
\end{aligned}
\]
Since \(R_n(D)=C(n,D)-M_n(D)\), this proves the absolute-value estimate and the two-sided bracket. Every estimate used above includes \(D=1\).
\end{proof}
\section*{Declaration}
\subsection*{Conflict of interest statement}
The author declares that there is no potential conflict of interest.
\subsection*{Ethical statement}
The author guarantees that the manuscript has neither been published elsewhere nor submitted for publication in another journal, and will not be submitted to other journals while being reviewed. The author certifies that they have written entirely the manuscript and that its main results are original, and cite all necessary sources. The author will inform immediately to the editors if any error is found in the manuscript after the submission. 

The author certifies that parts of the elementary calculation in \S \ref{calculation for the proof of Thm 4.2} and Appendix \ref{dominant terms of C(n,D)} are accomplished by AI tools under the guidance of the author. Some grammar checking work is also accomplished by AI. 
\subsection*{The data availability statement}
The data that support the findings of this study are openly available.
\backmatter

\bibliography{liu}
\bibliographystyle{smfplain}

\end{document}